\documentclass[12pt]{article}
\usepackage{amsfonts,amssymb}
\usepackage{hyperref}
\usepackage{titlesec}
\usepackage{titletoc}
\usepackage{amsmath}
\usepackage{enumitem}
\usepackage{listings}
\usepackage{verbatim}
\usepackage{graphicx}
\usepackage{geometry}
\usepackage{mathrsfs}
\usepackage{bbm}
\usepackage{mathtools}
\usepackage[splitrule]{footmisc} 
\numberwithin{equation}{section}
\usepackage{amsthm}
\usepackage{xcolor, tikz}
\newtheorem{theorem}{Theorem}
\newtheorem{pro}{Proposition}

\newtheorem{lem}{Lemma}

\newtheorem{rem}{Remark}

\usepackage[autostyle=false, style=english]{csquotes}
\MakeOuterQuote{"}
\usepackage{bm}

\DeclareMathOperator{\sign}{sgn}

\DeclareMathOperator{\rank}{rank}

\usepackage{cite}
\title{Quantizing Heavy-tailed Data in Statistical Estimation: (Near) Minimax Rates, Covariate Quantization, and Uniform Recovery}
\usepackage{bigfoot}

\DeclareNewFootnote{AAffil}[arabic]
\DeclareNewFootnote{ANote}[fnsymbol]

\usepackage{etoolbox}
\makeatletter
\patchcmd\maketitle{\def\@makefnmark{\rlap{\@textsuperscript{\normalfont\@thefnmark}}}}{}{}{}
\makeatother
\makeatletter
\def\thanksAAffil#1{
	\footnotemarkAAffil\protected@xdef\@thanks{\@thanks%
		\protect\footnotetextAAffil[\the \c@footnoteAAffil]{#1}}%
}
\def\thanksANote#1{%
	\footnotemarkANote%
	\protected@xdef\@thanks{\@thanks%
		\protect\footnotetextANote[\the \c@footnoteANote]{#1}}%
}
\makeatother
\author{%
	Junren Chen%
\thanksAAffil{Department of 
	Mathematics, The University of Hong Kong.}
$^{,}$\thanksANote{Corresponding author. Email: chenjr58@connect.hku.hk}
	\and %
Michael K. Ng\footnotemarkAAffil[1]
$^{,}$\thanksANote{Emails: mng@maths.hku.hk, di.wang@kaust.edu.sa}\and Di Wang\thanksAAffil{Division of CEMSE, King Abdullah University of Science and Technology (KAUST).}
$^{,}$\footnotemarkANote[2]
}
\date{}
\begin{document}
	\maketitle
\begin{abstract}
    Modern datasets often exhibit heavy-tailed behaviour, while  quantization is   inevitable in digital signal processing and many machine learning problems. This paper studies the quantization of heavy-tailed data in several fundamental statistical estimation problems where the underlying distributions have bounded moments of some order (no greater than $4$). We propose to truncate and properly dither the data prior to a uniform quantization. Our major standpoint is that (near)   minimax rates of estimation error  could be achieved by computationally tractable estimators based on the quantized data produced by the proposed scheme. In particular, concrete results are worked out for covariance estimation, compressed sensing (also interpreted as  sparse linear regression), and matrix completion, all agreeing that the quantization only slightly worsens the multiplicative factor. Additionally, while prior results focused on quantization of responses (i.e., measurements), we study compressed sensing where the covariates (i.e., sensing vectors) are also quantized; in this case,   our recovery program is non-convex because the covariance matrix estimator lacks   positive semi-definiteness, but all local minimizers are proved to enjoy near optimal error bound. Moreover, by  the concentration inequality of product process and a  covering argument, we establish near minimax uniform recovery guarantee for quantized compressed sensing with heavy-tailed noise. Finally, numerical simulations are provided to corroborate our theoretical results.
\end{abstract}
\section{Introduction}
Heavy-tailed distributions are ubiquitous in modern datasets, especially those arising in economy, finance, imaging, biology, see \cite{woolson2011statistical,sahu2019noising,ibragimov2015heavy,biswas2007statistical,swami2002some,kruczek2020detect} for instance. \textcolor{black}{In the recent  literature,  heavy-tailed distribution is often captured by bounded $l$-th moment, where $l$ is some fixed small scalar;}  this is essentially weaker than sub-Gaussian assumption, and thus outliers and extreme values appear much more frequently in data from heavy-tailed distributions (referred to as heavy-tailed data), which poses challenges for statistical analysis. In fact, many standard statistical procedures developed for sub-Gaussian data suffer from performance degradation in the heavy-tailed regime. Fortunately, the past decade has witnessed considerable progresses on  statistical estimation methods that are robust to heavy-tailedness, see \cite{catoni2012challenging,fan2017estimation,nemirovskij1983problem,minsker2015geometric,brownlees2015empirical,hsu2016loss,fan2021shrinkage,zhu2021taming,lugosi2019mean} for instance.

Departing momentarily from heavy-tailed data,
 quantization is an inevitable process in the era of digital signal processing, which maps signals to bitstreams so that they can be stored, processed  and transmitted.  In particular, the resolution of quantization should be selected to achieve a trade-off between accuracy and various data processing costs, and in some applications relatively low resolution would be preferable. For instance, in a distributed learning setting or a MIMO system, the frequent information transmission among multiple parties often result in prohibitive communication cost   \cite{konevcny2016federated,mo2018limited}, and  quantizing signals or data to fairly low resolution (while preserving satisfying utility) is an effective approach to reduce the cost \cite{zhang2017zipml,hanna2021quantization}. Under such a big picture, in recent years there has been rapidly growing literature on high-dimensional signal recovery from quantized data (see, e.g., \cite{boufounos20081,davenport20141,chen2022high,thrampoulidis2020generalized,jung2021quantized,dirksen2021non,dirksen2021covariance} for 1-bit quantization, \cite{thrampoulidis2020generalized,jung2021quantized,xu2020quantized,jacques2017time} for multi-bit uniform quantization), trying to understand the interplay between quantization and signal reconstruction (or parameter learning) in some fundamental estimation  problems. 


Independently, a set of robustifying techniques has been developed to overcome the challenge posed by  heavy-tailed data, and unidorm data quantization under uniform dither was shown to cost very little in some recovery problems. Considering ubiquitousness of heavy-tailed behaviour and data quantization, a natural question is to design quantization scheme for heavy-tailed data that only incurs minor information loss. For instance, when applied to statistical estimation problems with heavy-tailed data, an appropriate quantization scheme should enable   at least one faithful estimator from the quantized data, and  ideally an estimator nearly achieving optimal error rate. Despite the vast literature in this field, prior results that simultaneously take heavy-tailed data and quantization into account are surprisingly rare  --- only the ones presented in \cite{dirksen2021non} and our earlier work   \cite{chen2022high} regarding the dithered 1-bit quantizer, to the best of our knowledge. These results remain incomplete and exhibit some downsides. Specifically,  \cite{dirksen2021non} considered a computationally intractable program for quantized compressed sensing and used techniques hard to generalize to other problems, while the error rates in \cite{chen2022high} are inferior to the corresponding minimax ones (under unquantized sub-Gaussian data), as will be discussed in Section \ref{sec:quanht}. In a nutshell, a quantization scheme for heavy-tailed data arising in statistical estimation problems that allows for computationally tractable near-minimax estimators  is still lacking.

This paper aims to provide a solution to the above question and narrow the gap between heavy-tailed data and data quantization in the literature. In particular, we propose a unified quantization scheme for heavy-tailed data which,  when applied to the canonical estimation problems of (sparse) covariance matrix estimation, compressed sensing (or sparse linear regression) and matrix completion, allows for (near) minimax estimators that are either in closed-form or solved from convex programs.   Additionally, we present novel developments concerning   covariate (or sensing vector) quantization and uniform signal recovery in quantized compressed sensing with heavy-tailed data. 

 \subsection{Related Works and Our Contributions}
 This section is devoted to a review of the most relevant works. \textcolor{black}{Before that we note that a heavy-tailed random variable in this work is formulated by the moment constraint $\mathbbm{E}|X|^l\leq M$ where $M$ is oftentimes regarded as absolute constant and $l$ is some fixed small scalar (specifically, $l\leq 4$ in the present paper).} 
 \subsubsection{Statistical Estimation under Heavy-Tailed Data}\label{sec:HTesti}

 Compared to sub-Gaussian data, heavy-tailed data may contain many outliers and extreme values that are overly influential to traditional estimation methods. Hence, developing estimation methods that are robust to heavy-tailedness has become a recent focus in statistics literature, where heavy-tailed distributions are often only assumed to have bounded moments of some small order. In particular, significant efforts have been devoted to the fundamental problem of mean estimation for heavy-tailed distribution. For instance, effective techniques available in the literature include Catoni's mean estimator \cite{catoni2012challenging,fan2017estimation}, median of means \cite{nemirovskij1983problem,minsker2015geometric}, and trimmed mean \cite{lugosi2021robust,devroye2016sub}. While the strategies to achieve robustness are different, these methods indeed share the same core spirit of making the outliers less influential. To this end, the trimmed method (also referred to as truncation or shrinkage) may be the most intuitive --- it truncates overlarge data to some threshold so that they are more benign for the estimation procedure. For more in-depth discussions we refer to the recent survey \cite{lugosi2019mean}. Furthermore, these robust methods for estimating the mean have been applied to empirical risk minimization \cite{brownlees2015empirical,hsu2016loss} and various high-dimensional estimation problems \cite{fan2021shrinkage,zhu2021taming}, achieving near optimal guarantees. For instance, by invoking M-estimators with truncated data, (near) minimax rates can be achieved in high-diemnsional sparse linear regression, matrix completion, covariance estimation \cite{fan2021shrinkage}. In fact, techniques similarly to truncation has proven effective in some related problems, e.g., in non-convex algorithms for phase retrieval, truncated Wirtinger flow using a more selective spectral initialization and carefully trimmed gradient \cite{chen2015solving} improves on Wirtinger flow \cite{candes2015phase} in terms of sample size.

\textcolor{black}{
Recall that we capture heavy-tailedness by bounded moment of some small order, while regarding statistical estimation beyond sub-Gaussianity, there has been a line of works considering sub-exponential   or more generally sub-Weibull distributions \cite{qiu2020robust,sivakumar2015beyond,kuchibhotla2022moving,genzel2022generic}, which have heavier tail than sub-Gaussian ones but still possess finite moment up to arbitrary order.   Specifically, without truncation and quantization, sparse linear regression was studied under sub-exponential data in \cite{sivakumar2015beyond} and under sub-Weibull data in \cite{kuchibhotla2022moving}, and the obtained error rates match the   ones  in sub-Gaussian case  up to logarithmic factors. Additionally, under sub-exponential measurement matrix and noise, \cite{qiu2020robust} established uniform guarantee for 1-bit generative compressed sensing, while \cite{genzel2022generic} analysed generalized Lasso for a general nonlinear model. 
Because the tail assumptions in these works are substantially stronger than ours,\footnote{One exception is Theorem 4.6 in \cite{kuchibhotla2022moving} for sparse linear regression with sub-Weibull covariate and heavy-tailed noise with bounded second moment, but still this result does not involve quantization procedure.} we will not make special comparison with their results later; instead, we simply note here two key distinctions: 1) these works do   not utilize special treatments for heavy-tailed data like truncation in the present paper; 2) most of them (with the single exception of \cite{qiu2020robust} studying 1-bit quantization) do not focus on quantization, while     this work concentrates on quantization of heavy-tailed data.}

\subsubsection{Statistical Estimation from Quantized Data}\label{sec:quantiest}

\noindent{\textbf{Quantized Compressed Sensing.}} Due to the prominent role of quantization in signal processing and machine learning, quantized compressed sensing that studies the interplay between sparse signal reconstruction and  data quantization has become an active research branch in the field. In this work, we focus on memoryless quantization scheme\footnote{This means that the quantization for different measurements are independent. For other quantization schemes we refer to the recent survey \cite{dirksen2019quantized}.} that embraces simple hardware design. An important   model is 1-bit compressed sensing where only the sign of the measurement is retained \cite{boufounos20081,jacques2013robust,plan2012robust,plan2013one}, and more precisely this model concerns  the recovery of sparse $\bm{\theta^\star}\in \mathbb{R}^d$ from $\sign(\bm{X\theta^\star})$ with the sensing matrix $\bm{X}\in \mathbb{R}^{n\times d}$. However, 1-bit compressed sensing associated with the direct 
$\sign(\cdot)$ quantization suffers from some frustrating limitations, e.g., the loss of signal norm information, and the identifiability issue under some regular sensing matrix (e.g., under Bernoulli sensing matrix, see \cite{dirksen2021non}).\footnote{In fact, almost all existing guarantees using the 1-bit observations $\sign(\bm{X\theta^\star})$ are restricted to standard Gaussian sensing matrix consisting of i.i.d. $\mathcal{N}(0,1)$ entries, with the exceptions of \cite{ai2014one} for sub-Gaussian sensing matrix and \cite{dirksen2020one,stollenwerk2019one} for partial Gaussian circulant matrix.} Fortunately, these limitations can be overcome by random dithering prior to the quantization, under which the 1-bit measurements read as $\sign(\bm{X\theta^\star}+\bm{\tau})$ for some  suitably chosen random dither $\bm{\tau}\in \mathbb{R}^n$. Specifically, under Gaussian dither $\bm{\tau}\sim  \mathcal{N}(0, \bm{I}_n)$ and standard Gaussian sensing matrix $\bm{X}$, full reconstruction with norm information could be achieved, for which the key idea is $\sign(\bm{X\theta^\star}+\bm{\tau})=\sign([\bm{X}~\bm{\tau}][(\bm{\theta^\star})^\top,1]^\top)$, thus reducing the dithered model to the undithered model for the sparse signal $[(\bm{\theta^\star})^\top,1]^\top$ whose last entry is known beforehand \cite{knudson2016one}.\footnote{We note that this idea has been recently extended in \cite{chen2022uniform} to the related problem of phase-only compressed sensing, also see \cite{jacques2021importance,feuillen2020ell,boufounos2013sparse} for prior developments.} More surprisingly, under a  uniform random dither, recovery with norm can be achieved under rather general sub-Gaussian sensing matrix  \cite{dirksen2021non,jung2021quantized,thrampoulidis2020generalized,chen2022high} even with near optimal error rate.

Besides the 1-bit quantizer that retains the sign, the uniform quantizer maps $a\in \mathbb{R}$ to $\mathcal{Q}_\Delta(a)= \Delta\big(\lfloor \frac{a}{\Delta}\rfloor +\frac{1}{2} \big)$ for some pre-specified $\Delta>0$; here and hereafter, we refer to $\Delta$  as the quantization level, and note that smaller $\Delta$ represents higher resolution. While recovering $\bm{\theta^\star}$ from $\mathcal{Q}_\Delta(\bm{X\theta^\star})$ encounters identifiability issue,\footnote{For instance, if $\bm{X}\in \{-1,1\}^{n\times d}$ (typical example is the Bernoulli design where entries of $\bm{X}$ are i.i.d. zero-mean) and $\Delta=1$, then $\bm{\theta}_1:=1.1\bm{e}_1$ and $\bm{\theta}_2:=1.2\bm{e}_1+0.1\bm{e}_2$ can never be distinguished because   $\mathcal{Q}_1(\bm{X\theta}_1)=\mathcal{Q}_1(\bm{X\theta}_2)$ always holds.} it is again beneficial to use  random dithering to obtain the measurements  $\mathcal{Q}_\Delta(\bm{X\theta^\star}+\bm{\tau})$. More specifically, by using uniform dither the Lasso estimator \cite{thrampoulidis2020generalized,sun2022quantized} and projected back projection (PBP) method \cite{xu2020quantized} achieve   minimax rate in certain cases, and the derived error bounds for these estimators demonstrate that the dithered uniform quantization does not affect the scaling law but only slightly worsens the multiplicative factor. Although the aforementioned progresses regarding compressed sensing under dithered quantization were recently made,  the technique of dithering in quantization indeed has a long history and (at least) dates back to some early engineering work (e.g., \cite{roberts1962picture}), see \cite{gray1993dithered} for a brief introduction.

\noindent{\textbf{Other Estimation Problems with Quantized Data.}} Beyond compressed sensing or its corresponding more statistical setting of sparse linear regression,
some other statistical estimation problems were also   investigated under dithered 1-bit quantization. Specifically, \cite{dabeer2006signal} studied a general signal estimation problem under dithered 1-bit  quantization in a traditional setting where sample size tends to infinity, showing that only logarithmic rate loss is incurred by the quantization. Inspired by potential application in reduction of power consumption in a large scale massive MIMO system,
 \cite{dirksen2021covariance} proposed to collect 2 bits per entry from each sub-Gaussian sample and developed an estimator that is near minimax optimal in certain cases. Their estimator from coarsely quantized samples was   extended to high-dimensional sparse case in \cite{chen2022high}. Then, considering the ubiquitousness of binary observations in many recommendation systems, the authors of \cite{davenport20141} first approached the 1-bit matrix completion problem by maximum likelihood estimation with a nuclear norm constraint. Their method was developed by a series of follow-up works by using different regularizers/constraints to encourage low-rankness, or considering multi-bit quantization on a finite alphabet \cite{cai2013max,lafond2014probabilistic,klopp2015adaptive,cao2015categorical,bhaskar2016probabilistic}. Quantizing the observed entries by a dithered 1-bit quantizer, the 1-bit matrix completion result in \cite{chen2022high}   essentially departs from the standard likelihood approach and can tolerate pre-quantization noise with unknown distribution.

\subsubsection{Quantization of Heavy-Tailed Data in Statistical Estimation}
\label{sec:quanht}
From now on, we turn to existing results more closely related to this work and explain our contributions. Note that the results we just reviewed are    for estimation problems from either unquantized heavy-tailed data (Section \ref{sec:HTesti})  or quantized sub-Gaussian      data (Section \ref{sec:quantiest}). While quantization of heavy-tailed data (from distribution assumed to have bounded moment of some small order) is a natural question of significant practical value,
prior investigations turn out to be surprisingly rare, and the only results we are aware of were presented in \cite{chen2022high,dirksen2021non} concerning  dithered 1-bit quantization. Specifically, \cite[Thm. 1.11]{dirksen2021non} considered  heavy-tailed noise and  possibly heavy-tailed covariate, implying that a sharp uniform error rate is achievable (see their Example 1.10). However, their result is for a computationally intractable program (Hamming distance minimization) and hence of limited practical value. Another limitation is that their techniques are based on random hyperplane tessellations that is specialized to 1-bit compressed sensing but does not generalize to other estimation problems.
In contrast, \cite{chen2022high} proposed a unified quantization scheme that first truncates the data and then invokes a dithered 1-bit quantizer. Although this quantization scheme could (at least) be applied to sparse covariance matrix estimation, compressed sensing, and matrix completion while still enabling practical estimators, the main drawback is that the convergence rates of estimation errors are essentially slower than the corresponding minimax optimal ones (e.g., $\tilde{O}\big(\frac{\sqrt{s}}{n^{1/3}}\big)$ for 1-bit compressed sensing under heavy-tailed noise \cite[Thm. 10]{chen2022high}), and in certain cases the rates cannot be improved without changing the quantization process (e.g., \cite[Thm. 11]{chen2022high} complements \cite[Thm. 10]{chen2022high} with a nearly matching lower bound). Beyond that, the 1-bit compressed sensing results in \cite{chen2022high} are non-uniform. In a nutshell, \cite{dirksen2021non} proved sharp rate for  1-bit compressed sensing but used highly intractable program and techniques not extendable to other estimation regimes, while the more widely applicable scheme and practical estimators in \cite{chen2022high} suffer from slow error rates (when compared to the ones achieved from unquantized sub-Gaussian data).  

\vspace{1mm}

\noindent{\textbf{Our Main Contributions: (Near) Minimax Rates.}} We propose a unified quantization scheme for heavy-tailed data consisting of  three steps: 1) {\it truncation} that shrinks data to some threshold, 2) {\it dithering} that adds suitable random noise to the truncated data, and 3) {\it uniform quantization}. For sub-Gaussian or sub-exponential data the truncation step is inessential, and we simply set the threshold as $\infty$ in this case. Careful readers may notice that, we just replace the 1-bit quantizer in our prior work \cite{chen2022high} with the less extreme (multi-bit) uniform quantizer $\mathcal{Q}_\Delta(\cdot)$, but the gain turns out to be significant ---  we are now able to derive (near) optimal rates that are essentially faster than the ones in \cite{chen2022high}, see Theorems \ref{thm1}-\ref{thm9}. Compared to \cite{dirksen2021non}, besides the different quantizers, other major distinctions are that: 1) we utilize an additional truncation step, 2) our estimators are computationally feasible, and 3) we investigate multiple estimation problems with possibility of extensions to more estimation problems. Concerning the effect of    quantization, our error rates suggests a unified conclusion that {\it dithered uniform quantization does not affect   the scaling law but only slightly worsens  the multiplicative factor}, which generalizes similar findings for quantized compressed sensing in \cite{thrampoulidis2020generalized,sun2022quantized,xu2020quantized} towards two directions, i.e., to  the case where heavy-tailed data present and to some other estimation problems (i.e., covariance matrix estimation, matrix completion).  As an example, for quantized compressed sensing with sub-Gaussian sensing vector $\bm{x}_k$ but heavy-tailed measurement $y_k$ satisfying $\mathbbm{E}|y_k|^{2+\nu}\leq M$  for some $\nu>0$, we derive the $\ell_2$-norm error rate $O\big(\mathscr{L}\sqrt{\frac{s \log d}{n}}\big)$ where $\mathscr{L}=M^{1/(2l)} +\Delta$ (Theorem \ref{thm4}, $s,d,n$ are respectively the sparsity, signal dimension, measurement number), which is reminiscent of the rates in \cite{xu2020quantized,thrampoulidis2020generalized} in terms of the position of $\Delta$.    From a technical side, many of our analyses on the dithered quantizer are much cleaner than prior works  because we   make full use of the nice statistical properties of the quantization error and quantization noise   (Theorem \ref{lem1}),\footnote{Prior work that did not fully leveraged these properties  may incur extra technical complication, e.g., the symmetrization and contraction in \cite[Lem. A.2]{thrampoulidis2020generalized}.} see Section \ref{prequan}. Also, the property of quantization noise   motivates us to use triangular dither when covariance estimation is necessary, which departs from the uniform dither commonly adopted in prior works (e.g., \cite{thrampoulidis2020generalized,xu2020quantized,dirksen2021covariance,chen2022high}) and is novel to the literature of quantized compressed sensing. In our subsequent work \cite{chen2023quantized}, a clean analysis on quantized low-rank multivariate regression with possibly quantized covariates is provided, and we believe the innovations of this work will prove useful in other problems.

\subsubsection{Covariate Quantization in Quantized Compressed Sensing}
From now on we   concentrate on quantized compressed sensing, i.e., the recovery of a sparse signal  $\bm{\theta^\star}\in \mathbb{R}^d$ from the quantized version of $(\bm{x}_k,y_k:=\bm{x}_k^\top\bm{\theta^\star}+\epsilon_k)_{k=1}^n$ where $\bm{x}_k,y_k,\epsilon_k$ are the sensing vector, measurement and noise, respectively. Let us first clarify some terminology issue before proceeding.
Note that this formulation also models the  sparse linear regression problem (e.g., \cite{negahban2012unified,raskutti2011minimax}) where one wants to learn a sparse parameter $\bm{\theta^\star}\in \mathbb{R}^d$ from the given data $(\bm{x}_k,y_k)_{k=1}^n$ that are believed to follow the linear model $y_k=\bm{x}_k^\top\bm{\theta^\star}+\epsilon_k$, and in this regression problem $\bm{x}_k,y_k$ are commonly referred to as covariate and response, respectively. We are interested in both settings in this work (as further explained in Section \ref{sec4}), but for clearer presentation, we simply refer to the problem as  quantized compressed sensing, while calling $\bm{x}_k,y_k$ covariate and response, respectively.

Despite a large volume of results in   quantized compressed sensing, almost all of them are restricted to response quantization, thus the question of covariate quantization that allows for accurate subsequent recovery remains unsolved. Note that this question is meaningful especially when the problem is interpreted as sparse linear regression --- working with low-precision data in some (distributed) learning systems could significantly reduce communication cost and power consumption \cite{zhang2017zipml,hanna2021quantization}, which we will further demonstrate in Section \ref{sec:covarquan}. Therefore, it is of interest to understand how covariate quantization affects the learning of $\bm{\theta^\star}$. To the best of our knowledge, the only existing rigorous guarantees for quantized compressed sensing involving covariate quantization were obtained in \cite[Thms. 7-8]{chen2022high}. However, these results require $\mathbbm{E}(\bm{x}_k\bm{x}_k^\top)$ to be sparse \cite[Assumption 3]{chen2022high} (in order to employ their sparse covariance matrix estimator), and this assumption is non-standard in sparse linear regression and compressed sensing.\footnote{In fact, although   isotropic sensing vector    (i.e., $\mathbbm{E}(\bm{x}_k\bm{x}_k^\top)=\bm{I}_d$) has been conventional in compressed sensing, many results in the literature can be extended to sensing vector with general unknown covariance matrix and hence do not really rely on the sparsity of $\mathbbm{E}(\bm{x}_k\bm{x}_k^\top)$.}

\noindent
\textbf{Our Contribution.} Besides the above main contributions, we establish the estimation guarantees for quantized compressed sensing under covariate quantization that are free of  the non-standard assumption on the sparsity of $\mathbbm{E}(\bm{x}_k\bm{x}_k^\top)$. Like \cite{chen2022high}, our estimation methods are built upon the quantized covariance matrix estimator developed in Section \ref{sec:covariance}; but unlike \cite{chen2022high} that relies on the sparsity of $\mathbbm{E}(\bm{x}_k\bm{x}_k^\top)$ to ensure convexity, we instead deal with the non-convex program with an additional $\ell_1$-norm constraint, under which we prove that   all local minimizers deliver near minimax estimation errors (Theorems \ref{thm6}-\ref{thm7}). Our analysis bears resemblance to a line of works on non-convex M-estimator \cite{loh2011high,loh2017statistical,loh2013regularized} but also exhibits some essential differences (Remark \ref{rem4}). Further, we extract our techniques as a deterministic framework (Proposition \ref{framework}) and then use it to establish guarantees under dithered 1-bit quantization and covariate quantization as byproducts (Theorems \ref{sg1bitqccs}-\ref{ht1bitqccs}), which are comparable to \cite[Thms. 7-8]{chen2022high} but free of sparsity on $\mathbbm{E}(\bm{x}_k\bm{x}_k^\top)$.

\subsubsection{Uniform Signal Recovery in Quantized Compressed Sensing}

It is standard in compressed sensing to leverage a random sensing matrix, so a recovery guarantee can be  uniform or non-uniform. More precisely, a uniform guarantee ensures the recovery of all structured signals of interest with a single draw of the sensing ensemble, while a non-uniform guarantee is only valid for a  structured signal fixed before drawing the random ensemble, with the implication that a new realization of the sensing matrix is required for   sensing   a new signal. Uniformity is a highly desired property in compressed sensing, since in applications the measurement ensemble is typically fixed and is expected to work for all signals \cite{genzel2022unified}. Besides, the derivation of a uniform guarantee is often significantly harder than a non-uniform one, making uniformity an interesting theoretical problem in its own right.

A classical fact in linear compressed sensing is that the restricted  isometry property (RIP) of the sensing matrix implies uniform recovery of all sparse signals (e.g., \cite{Foucart2013}), but this is not the case when it comes to  nonlinear compressed sensing models, for which uniform recovery guarantees are still eagerly pursued so far. For instance, in the specific    quantization model involving 1-bit/uniform quantization with/without dithering, or the more general single index model $y_k = f\big(\bm{x}_k^\top\bm{\theta^\star}\big)$ with possibly unknown $f(\cdot)$, most representative results are non-uniform (e.g., \cite{plan2012robust,plan2017high,plan2016generalized,thrampoulidis2020generalized,xu2020quantized,sun2022quantized}).  We refer to \cite{plan2012robust,dirksen2021non,xu2020quantized,jung2021quantized} for   concrete uniform guarantees, some of which remain (near) optimal (e.g., \cite[Sect. 7.2{A}]{xu2020quantized}), while others suffer from essential degradation   compared to the non-uniform ones (e.g., \cite[Thm. 1.3]{plan2012robust}). It is worth noting the interesting recent work \cite{genzel2022unified} who provided a unified approach to   uniform guarantee in a series of non-linear models, but without the aid of some non-trivial embedding result, their uniform guarantees typically exhibit a decaying rate of $O(n^{-1/4})$ that is slower than the non-uniform one of $O(n^{-1/2})$ (Section 4 therein).  Turning back to our focus of compressed sensing from quantized heavy-tailed data, results in \cite[Sect. III]{chen2022high} are  non-uniform, while \cite[Thm. 1.11]{dirksen2021non} presents   sharp  uniform guarantee for the intractable program of hamming distance minimization under dithered 1-bit quantization.

\noindent{\textbf{Our Contribution.}} We additionally contribute to the literature a uniform guarantee for  constrained Lasso under dithered uniform quantization of heavy-tailed response. Specifically, we upgrade our non-uniform Theorem \ref{thm4} to its  uniform version Theorem \ref{uniformtheorem}, which states that using a single realization of the sub-Gaussian sensing matrix, heavy-tailed noise and uniform dither, all $s$-sparse signals within an $\ell_2$-ball can be uniformly recovered up to an $\ell_2$-norm error of $\tilde{O}\big(\sqrt{\frac{s}{n}}\big)$, thus matching the near minimax non-uniform rate in Theorem \ref{thm4} up to logarithmic factors. The proof relies on a concentration inequality for product process \cite{mendelson2016upper} and a careful covering argument inspired by \cite{xu2020quantized}. Due to the heavy-tailed noise, new treatment is needed before invoking the concentration result from \cite{mendelson2016upper}.

\subsection{Outline}
The remainder of this paper is structured as follows. We provide the notation and preliminaries in Section \ref{sec2}. We present the first set of main results (concerning the (near) optimal guarantees for three   estimation problems under quantized heavy-tailed data) in Section \ref{sec3}. Our second set of results (concerning covariate quantization and uniform recovery in quantized compressed sensing) is then presented in Section \ref{sec:qc-qcs}. To corroborate our theory, numerical results on   synthetic data  are reported in Section \ref{sec6}. We give some remarks to conclude the paper in Section \ref{sec7}. All the proofs are postponed to the Appendices.

\section{Preliminaries}\label{sec2}
We adopt the following conventions throughout the paper:

 1) We use  boldface symbols  (e.g., $\bm{A}$, $\bm{x}$) to denote matrices and vectors, and regular letters (e.g, $a$, $x$) for scalars. We write $[m] = \{1,...,m\}$ for positive integer $m$. We denote the complex unit by $\mathsf{i}$. The $i$-th entry for a vector $\bm{x}$ (likewise, $\bm{y}$, $\bm{\tau}$) is denoted by $x_i$ (likewise, $y_i$, $\tau_i$).

 2) Notation with "$\star$" as superscript denotes the desired underlying parameter or signal, e.g., $\bm{\Sigma}^\star$, $\bm{\theta}^\star$.   Moreover, notation marked by a tilde (e.g., $\bm{\widetilde{x}}$) and a dot (e.g., $\bm{\dot{\bm{x}}}$)    stands for the truncated data and quantized data, respectively.

3) We reserve $d$, $n$ for the problem dimension and sample size, respectively. In many cases $\bm{\widehat{\Upsilon}}$ denotes the estimation error, e.g., $\bm{\widehat{\Upsilon}}=\bm{\widehat{\theta}}-\bm{\theta^\star}$ if $\bm{\widehat{\theta}}$ is the estimator for the desired signal $\bm{\theta^\star}$.  We use $\Sigma_s$ to denote the set of $d$-dimensional $s$-sparse signals.


 4) For vector $\bm{x} \in \mathbb{R}^d$, we work with its transpose $\bm{x}^\top$,  $\ell_p$-norm $\|\bm{x}\|_p = (\sum_{i\in [d]} |x_i|^p)^{1/p}$ ($p\geq 1$), max norm $\|\bm{x}\|_{\infty} = \max_{i\in [d]}|x_i|$. 
We define the standard Euclidean sphere as $\mathbb{S}^{d-1}= \{\bm{x}\in\mathbb{R}^d:\|\bm{x}\|_2=1\}$.

5) For   matrix $\bm{A}= [a_{ij}]\in \mathbb{R}^{m\times n}$ with singular values $\sigma_1\geq \sigma_2\geq ...\geq \sigma_{\min\{m,n\}}$, recall the operator norm $\|\bm{A}\|_{op}= \sup_{\bm{v}\in \mathbb{S}^{n-1}}\|\bm{Av}\|_2=\sigma_1$, Frobenius norm $\|\bm{A}\|_F = (\sum_{i,j}a_{ij}^2)^{1/2}$, nuclear norm $\|\bm{A}\|_{nu} = \sum_{k=1}^{\min\{m,n\}}\sigma_k$, and max norm $\|\bm{A}\|_{\infty}= \max_{i,j}|a_{ij}|$. $\lambda_{\min}(\bm{A})$ (resp. $\lambda_{\max}(\bm{A})$) stands for the minimum eigenvalue (resp. maximum eigenvalue) of a symmetric $\bm{A}$.

6) We  denote universal constants by $C$, $c$, $C_i$ and $c_i$, whose value may vary from line to line.   We write $T_1 \lesssim T_2$ or $T_1 = O(T_2)$ if $T_1\leq CT_2$. Conversely, if $T_1\geq CT_2$ we write $T_1 \gtrsim T_2$ or $T_1 = \Omega(T_2)$. Also, we write $T_1 \asymp T_2$ if   $T_1=O(T_2)$ and $T_2 = \Omega(T_1)$ simultaneously hold.

7) We use $\mathscr{U}(\Omega)$ to denote the uniform distribution over $\Omega\subset \mathbb{R}^N$, $\mathcal{N}(\bm{\mu},\bm{\Sigma})$ to denote Gaussian distribution with mean $\bm{\mu}$ and covariance $\bm{\Sigma}$, $\mathsf{t}(\nu)$ to denote student's t distribution with degrees of freedom $\nu$.

\textcolor{black}{8) Our technique to handle heavy-tailedness is a data truncation step, for which we introduce the operator $\mathscr{T}_\zeta(\cdot)$ for some threshold $\zeta>0$. It is defined as $\mathscr{T}_\zeta(a)=\sign(a)\min\{|a|,\zeta\}$ for some $a\in \mathbb{R}$. To truncate vector we apply $\mathscr{T}_\zeta(\cdot)$ entry-wisely in most cases, with the exception of covariance matrix estimation under operator norm error (Theorem \ref{thm2}).}

9) $\mathcal{Q}_\Delta(.)$ is the uniform quantizer with quantization level  $\Delta>0$. It applies to scalar $a$ by $\mathcal{Q}_\Delta(a)=\Delta\big(\big\lfloor\frac{a}{\Delta}\big\rfloor+\frac{1}{2}\big)$, and we set $\mathcal{Q}_0(a) =a$. Given a threshold $\mu$, the hard thresholding of scalar $a$ is $\mathcal{T}_\mu(a)= a\cdot \mathbbm{1}(|a|\geq \mu)$. Both functions element-wisely apply to vectors or matrices.


\subsection{High-Dimensional Statistics}
 
Let $X$ be a real random variable, we present some basic knowledge on the sub-Gaussian and sub-exponential random variable. Then we also precisely formulate  the heavy-tailed distribution.

1) The sub-Gaussian norm is defined as $\|X\|_{\psi_2} = \inf\{t>0:\mathbbm{E}\exp(\frac{X^2}{t^2})\leq 2\}$. A random variable $X$ with  finite $\|X\|_{\psi_2}$ is said to be sub-Gaussian. Analogously to Gaussian variable, a sub-Gaussian random variable exhibits an exponentially-decaying probability tail and satisfies a  moment constraint: 
\begin{gather}
    \mathbbm{P}(|X|\geq t)\leq 2\exp\left(-\frac{ct^2}{\|X\|_{\psi_2}^2}\right);\label{2.2a}\\
    (\mathbbm{E}|X|^p)^{1/p}\leq C\|X\|_{\psi_2}\sqrt{p},~\forall~ p\geq 1.\label{2.2b}
\end{gather}   
Note that these two properties can also define $\|\cdot\|_{\psi_2}$ up to multiplicative constant, e.g., $\|X\|_{\psi_2}\asymp \sup_{p\geq 1}\frac{(\mathbbm{E}|X|^p)^{1/p}}{\sqrt{p}}$ (see \cite[Prop. 2.5.2]{vershynin2018high}). For a $d$-dimensional random vector $\bm{x}$ we define its   sub-Gaussian norm as  $\|\bm{x}\|_{\psi_2}=\sup_{\bm{v}\in\mathbb{S}^{d-1}}\|\bm{v}^\top\bm{x}\|_{\psi_2}$.

2) The sub-exponential norm is defined as $\|X\|_{\psi_1} = \inf\{t>0:\mathbbm{E}\exp(\frac{|X|}{t})\leq 2\}$, 
and $X$  is sub-exponential if $\|X\|_{\psi_1}<\infty$. The sub-exponential $X$ satisfies the following properties:
\begin{gather}
    \mathbbm{P}(|X|\geq t)\leq 2\exp\left(-\frac{ct}{\|X\|_{\psi_1}}\right);\nonumber\\
    (\mathbbm{E}|X|^p)^{1/p}\leq C\|X\|_{\psi_1}p,~\forall~p\geq 1.\label{semoment}
\end{gather}
 To relate $\|.\|_{\psi_1}$ and $\|.\|_{\psi_2}$ one has $\|XY\|_{\psi_2}\leq \|X\|_{\psi_1}\|Y\|_{\psi_1}$ \cite[Lem. 2.7.7]{vershynin2018high}.


3) In contrast to the moment constraints in (\ref{2.2b}) and (\ref{semoment}),    heavy-tailed distributions in this work are only assumed to satisfy bounded moments of some small order no greater than $4$, formulated for a random variable $X$ as $\mathbbm{E}|X|^l\leq M$ for some $M>0$ and $l\in (0,4]$. Following \cite[Def. 2.4, 2.5]{kuchibhotla2022moving}, we consider the following two moment assumptions for a heavy-tailed random vector $\bm{x}\in \mathbb{R}^d$ (again, $M>0$, $l\in (0,4]$): 
\begin{itemize}
[leftmargin=5ex,topsep=0.15ex]
 \setlength\itemsep{-0.1em}
    \item \textbf{Marginal Moment Constraint.} The weaker assumption that constraints the moment of each coordinate, formulated by $\sup_{i\in [d]}\mathbbm{E}|x_i|^l\leq M$.   
    \item \textbf{Joint Moment Constraint.} The stronger assumption that constraints the moments "toward all directions $\bm{v}\in \mathbb{S}^{d-1}$," formulated by $\sup_{\bm{v}\in \mathbb{S}^{d-1}}\mathbbm{E}|\bm{v}^\top\bm{x}|^l\leq M$.  
\end{itemize}

\subsection{Dithered Uniform Quantization}\label{prequan}
In this part, we describe the dithered uniform quantizer and its properties in detail. We also specify the choices of random dither in this work.

1) We first provide the detailed procedure of dithered quantization and     its general property.   Let $\bm{x}\in \mathbb{R}^N$ be the input signal with dimension $N\geq 1$ \textcolor{black}{whose entries may be random and dependent}. Independent of $\bm{x}$, we generate the random dither $\bm{\tau}\in \mathbb{R}^N$ with i.i.d. entries from some distribution,\footnote{Throughout this work, we suppose that a random dither is drawn independent of anything else (particularly, the signal to be quantized and other dithers), and the dither has i.i.d. entries if it is a vector.} and then quantize $\bm{x}$ to $\bm{\dot{x}}= \mathcal{Q}_\Delta(\bm{x}+\bm{\tau})$. Following \cite{gray1993dithered}, we refer to $\bm{w}:= \bm{\dot{x}}- (\bm{x}+\bm{\tau})$ as the quantization error, and $\bm{\xi}:= \bm{\dot{x}} - \bm{x}$ as the quantization noise. The principal properties of dithered quantization are provided in  Theorem \ref{lem1}. 

\begin{theorem}
\label{lem1}  {\rm 	(Adapted from   \cite[Thms. 1-2]{gray1993dithered})\textbf{.}}
Consider the dithered uniform quantization described above for the input signal $\bm{x}  $, with random dither $\bm{\tau} =[\tau_i]$, quantization error $\bm{w}  $ and quantization noise $\bm{\xi} =[\xi_i]$. Use $\mathsf{i}$ to denote the imaginary unit, and let $Y$ be the random variable having the same distribution as the random dither $\tau_i$.  

\noindent{\rm (a)} {\rm (Quantization Error){\bf \sffamily.}} If $f(u):=\mathbbm{E}(\exp(\mathsf{i} uY))$ satisfies $f\big(\frac{2\pi  l}{\Delta}\big)=0$ for all non-zero integer $l$, then $\bm{w}\sim \mathscr{U}([-\frac{\Delta}{2},\frac{\Delta}{2}]^N)$ is independent of $\bm{x}$.\footnote{Although the statement is a bit different, it can be implied by \cite[Thm. 1]{gray1993dithered} and the proof therein.}

\noindent{\rm(b)} {\rm (Quantization Noise){\bf \sffamily.}} Assume that $Z\sim \mathscr{U}([-\frac{\Delta}{2},\frac{\Delta}{2}])$ is independent of $Y$. Let $g(u):=\mathbbm{E}(\exp(\mathsf{i} uY))\mathbbm{E}(\exp(\mathsf{i} u Z))$. Given positive integer $p$, if the $p$-th order derivative $g^{(p)}(u)$ satisfies  $g^{(p)}\big( \frac{2\pi  l}{\Delta}\big)=0$ for all non-zero integer $l$, then the $p$-th conditional moment of $\xi_i$ does not depend on $\bm{x}$: $\mathbbm{E}[\xi_i^p|\bm{x}] = \mathbbm{E}(Y+Z)^p$. 
\end{theorem} 

\textcolor{black}{
We note that Theorem \ref{lem1} serves as the cornerstone for our analysis on the dithered uniform quantizer; for instance, (a) allows for applications of concentration inequalities in our analyses, and (b) inspires us to develop a covariance matrix estimator from quantized samples. The take-home message is that, adding appropriate dither before quantization can make the quantization error and quantization noise behave in a statistically nice manner. For example, the elementary form of Theorem \ref{lem1}(a) is that under a dither $\tau_i$ satisfying the condition there, the quantization noise $\mathcal{Q}_\Delta(x_i+\tau_i)-(x_i+\tau_i)$ follows $\mathscr{U}([-\frac{\Delta}{2},\frac{\Delta}{2}])$ under any given scalar $x_i$ \cite[Lem. 1]{gray1993dithered}.} 

2) We use {\it uniform dither} for quantization of the response in compressed sensing and matrix completion. More specifically, under  $\Delta>0$, we adopt the uniform dither $\tau_k\sim \mathscr{U}([-\frac{\Delta}{2},\frac{\Delta}{2}])$ for the response $y_k\in \mathbb{R}$, which is also a common choice in previous works (e.g., \cite{thrampoulidis2020generalized,xu2020quantized,dirksen2021non,jung2021quantized}).  For $Y\sim  \mathscr{U}([-\frac{\Delta}{2},\frac{\Delta}{2}])$, it can be calculated that 
\begin{equation}\label{calcu1}
     \mathbbm{E}(\exp(\mathsf{i}uY))= \int_{-\Delta/2}^{\Delta/2}~\frac{1}{\Delta}\big(\cos (ux) + \mathsf{i}\sin (ux)\big)~\mathrm{d}x = \frac{2}{\Delta u}\sin\Big(\frac{\Delta u}{2}\Big),
\end{equation}
and hence $\mathbbm{E}(\exp(\mathsf{i} \frac{2\pi l}{\Delta}  Y))=0$ holds for all non-zero integer $l$.  Therefore, the benefit of using $\tau_k \sim \mathscr{U}([-\frac{\Delta}{2},\frac{\Delta}{2}])$
is that the quantization errors $w_k=\mathcal{Q}_\Delta(y_k+\tau_k)-(y_k+\tau_k)$ i.i.d. follow $\mathscr{U}([-\frac{\Delta}{2},\frac{\Delta}{2}])$, and are independent of $\{y_k\}$.

3) We use {\it triangular dither} for quantization of the covariate, i.e., the sample in covariance estimation or the covariate in comprssed sensing. Particularly, when considering the uniform quantizer $\mathcal{Q}_\Delta(.)$  for   the covariate $\bm{x}_k\in \mathbb{R}^d$, we adopt the dither $\bm{\tau}_k\sim \mathscr{U}([-\frac{\Delta}{2},\frac{\Delta}{2}]^d)+\mathscr{U}([-\frac{\Delta}{2},\frac{\Delta}{2}]^d)$,\footnote{\textcolor{black}{An equivalent statement is that entries of $\bm{\tau}_k$ are i.i.d. distributed as $\mathscr{U}\big(\big[-\frac{\Delta}{2},\frac{\Delta}{2}\big]\big)+\mathscr{U}\big(\big[-\frac{\Delta}{2},\frac{\Delta}{2}\big]\big)$. The equivalence can be clearly seen by comparing the joint probability density functions.}} which is  the sum of two independent  $\mathscr{U}([-\frac{\Delta}{2},\frac{\Delta}{2}]^d)$ and referred to as a triangular dither \cite{gray1993dithered}. Simple calculations verify that the triangular dither respects not only the condition in Theorem \ref{lem1}(a), but also the one in Theorem \ref{lem1}(b) with $p=2$; specifically, let $Y=Y_1+Y_2$ where $Y_1$ and $Y_2$ are independent and follow $\mathscr{U}\big(\big[-\frac{\Delta}{2},\frac{\Delta}{2}\big]\big)$, and let $Z\sim \mathscr{U}\big(\big[-\frac{\Delta}{2},\frac{\Delta}{2}\big]\big)$ be independent of $Y$, then based on (\ref{calcu1}), we know $f(u)=\mathbbm{E}(\exp(\mathsf{i}uY))=\big[\frac{2}{\Delta u}\sin \frac{\Delta u}{2}\big]^2$ satisfies  $f(\frac{2\pi l}{\Delta})=0$, and $g(u)=\mathbbm{E}(\exp(\mathsf{i}uY))\mathbbm{E}(\exp(\mathsf{i}uZ))=\big[\frac{2}{\Delta u}\sin\frac{\Delta u}{2}\big]^3$ satisfies $g''(\frac{2\pi l}{\Delta})=0$, where $l$ is any non-zero integer. Thus, at the cost of a   dithering variance larger than uniform dither, the     triangular dither brings the additional nice property of signal-independent variance for the quantization noise   --- $\mathbbm{E}(\xi_{ki}^2)=\frac{1}{4}\Delta^2$, where $\xi_{ki}$ is the $i$-th entry of $\bm{\xi}_k = \mathcal{Q}_\Delta(\bm{x}_k+\bm{\tau}_k)-(\bm{x}_k+\bm{\tau}_k)$.

To the best of our knowledge, the triangular dither is new to the literature of quantized compressed sensing. We will explain its necessity if covariance estimation is involved. This is also complemented by numerical simulation (see Figure \ref{fig5}(a)).


\section{(Near) Minimax Error Rates} \label{sec3}
In this section we derive (near) optimal error rates for several canonical statistical estimation problems. Our novelty is that by using the proposed quantization scheme for heavy-tailed data, (near) optimal error rates could be achieved by computationally feasible estimators.
\subsection{Quantized Covariance Matrix Estimation}\label{sec:covariance}
Given   $\mathscr{X}:=\{\bm{x}_1,...,\bm{x}_n\}$ as i.i.d. copies of a zero-mean  random vector $\bm{x}\in \mathbb{R}^d$, one often encounters the covariance  matrix  estimation problem, i.e., to estimate $\bm{\Sigma^\star} = \mathbbm{E}(\bm{x} \bm{x}^\top)$. This estimation problem is of fundamental importance in multivariate analysis and  has   attracted much research interest (e.g., \cite{cai2012optimal,cai2010optimal,cai2011adaptive,bickel2008covariance}). However, the   practically useful setting (e.g., in a massive MIMO system \cite{yang2023plug}) where the samples undergo certain quantization process remains under-developed, for which we are only aware of the 1-bit quantization results in \cite{dirksen2021covariance,chen2022high}. This setting poses the problem of quantized covariance matrix estimation (QCME), in which    {\it one aims to design quantization scheme for $\bm{x}_k$ that allows for accurate estimation of $\bm{\Sigma^\star}$ only based on the quantized samples}. 
We consider heavy-tailed $\bm{x}_k$ that possesses bounded fourth moments either marginally or jointly, but note that our estimation methods and theoretical results appear to be new even for sub-Gaussian $\bm{x}_k$ (Remark \ref{rem2}). 

As introduced before, we overcome the heavy-tailedness of $\bm{x}_k$  by a data truncation step, i.e., we first truncate $\bm{x}_k$ to $\bm{\widetilde{x}}_k$ in order to make the outliers less influential. Here, we defer the precise definition of $\bm{\widetilde{x}}_k$ to   concrete results because it should be well suited to the error metric. After the truncation, we dither and quantize $\bm{\widetilde{x}}_k$ to $\bm{\dot{x}}_k = \mathcal{Q}_\Delta(\bm{\widetilde{x}}_k+\bm{\tau}_k)$ with the triangular dither $\bm{\tau}_k\sim \mathscr{U}([-\frac{\Delta}{2},\frac{\Delta}{2}]^d)+\mathscr{U}([-\frac{\Delta}{2},\frac{\Delta}{2}]^d)$. Different from the uniform dither adopted in the literature (e.g., \cite{chen2022high,thrampoulidis2020generalized,xu2020quantized,dirksen2021non,jung2021quantized}),  first let us explain our choice of triangular dither. Recall that the  quantization noise and quantization error are respectively defined as $\bm{\xi}_k:=\bm{\dot{x}}_k - \bm{\widetilde{x}}_k$ and $\bm{w}_k:=\bm{\dot{x}}_k - \bm{\widetilde{x}}_k - \bm{\tau}_k$, 
  thus giving $\bm{\xi}_k = \bm{\tau}_k+\bm{w}_k$. Under   uniform dither or triangular dither,  $\bm{w}_k$ is independent of $\bm{\widetilde{x}}_k$ and follows $\mathscr{U}([-\frac{\Delta}{2},\frac{\Delta}{2}]^d)$ (see Section 2.2), thus allowing us to calculate that
\begin{equation}
    \begin{aligned}
    \label{3.1}
    \mathbbm{E}(\bm{\dot{x}}_k\bm{\dot{x}}_k^\top)&= \mathbbm{E}\big((\bm{\widetilde{x}}_k + \bm{\xi}_k)(\bm{\widetilde{x}}_k + \bm{\xi}_k)^\top\big)\\& = \mathbbm{E}(\bm{\widetilde{x}}_k\bm{\widetilde{x}}_k^\top)+ \mathbbm{E}(\bm{\widetilde{x}}_k\bm{\xi}_k^\top) + \mathbbm{E}(\bm{\xi}_k\bm{\widetilde{x}}_k^\top) + \mathbbm{E}(\bm{\xi}_k\bm{\xi}_k^\top)\\&\stackrel{(i)}{=}\mathbbm{E}(\bm{\widetilde{x}}_k\bm{\widetilde{x}}_k^\top)+ \mathbbm{E}(\bm{\xi}_k\bm{\xi}_k^\top).
    \end{aligned}
\end{equation}
Note that $(i)$ is because $\mathbbm{E}(\bm{\xi}_k\bm{\widetilde{x}}_k^\top)=\mathbbm{E}(\bm{\tau}_k\bm{\widetilde{x}}_k^\top)+ \mathbbm{E}(\bm{w}_k\bm{\widetilde{x}}_k^\top)=\mathbbm{E}(\bm{\tau}_k)\mathbbm{E}(\bm{\widetilde{x}}_k^\top)+\mathbbm{E}(\bm{w}_k)\mathbbm{E}(\bm{\widetilde{x}}_k^\top)=0$, due to the previously noted fact that $\bm{\tau}_k$ and $\bm{w}_k$ are independent of $\bm{\widetilde{x}}_k$ and zero-mean. While with suitable choice of the truncation threshold $\mathbbm{E}(\bm{\widetilde{x}}_k\bm{\widetilde{x}}_k^\top)$ is expected to well approximate $\bm{\Sigma^\star}$, the remaining $\mathbb{E}(\bm{\xi}_k\bm{\xi}_k^\top)$  gives rise to constant bias. To address the issue, a straightforward idea is to remove the bias, which requires the full knowledge on $\mathbbm{E}(\bm{\xi}_k\bm{\xi}_k^\top)$, i.e., the covariance matrix of the quantization noise.  For $i\neq j$, because  $\bm{\tau}_k$, $\bm{w}_k\sim \mathscr{U}([-\frac{\Delta}{2},\frac{\Delta}{2}]^d)$ and $\mathbbm{E}(w_{ki}\tau_{kj})=\mathbbm{E}_{\widetilde{x}_{ki}}(\mathbbm{E}[w_{ki}\tau_{kj}|\widetilde{x}_{ki}])=0$ (note that conditionally on $\widetilde{x}_{ki}$, $w_{ki}=\mathcal{Q}_\Delta(\widetilde{x}_{ki}+\tau_{ki})-(\widetilde{x}_{ki}+\tau_{ki})$ and $\tau_{kj}$ are independent),  we have
\begin{equation}
    \begin{aligned}\nonumber
        &\mathbbm{E}(\xi_{ki}\xi_{kj})=\mathbbm{E}\big((w_{ki}+\tau_{ki})(w_{kj}+\tau_{kj})\big)\\
        &=\mathbbm{E}(w_{ki}w_{kj})+\mathbbm{E}(w_{ki}\tau_{kj})+\mathbbm{E}(\tau_{ki}w_{kj})+\mathbbm{E}(\tau_{ki}\tau_{kj})=0,
    \end{aligned}
\end{equation}
showing that $\mathbbm{E}(\bm{\xi}_k\bm{\xi}_k^\top)$ is diagonal. 
Moreover,   under triangular dither the $i$-th diagonal entry is also known as $\mathbbm{E}|\xi_{ki}|^2=\frac{\Delta^2}{4}$, see Section \ref{prequan}. Taken collectively, we arrive at \begin{equation}
    \mathbbm{E}(\bm{\xi}_k\bm{\xi}_k^\top)=\frac{\Delta^2}{4}\bm{I}_d;\label{Delta24}
\end{equation} Based on  (\ref{3.1}) we thus propose the  following estimator
\begin{equation}
\label{3.2}
    \bm{\widehat{\Sigma}} = \frac{1}{n} \sum_{k=1}^n \bm{\dot{x}}_k \bm{\dot{x}}_k^\top - \frac{\Delta^2}{4}\bm{I}_d,
\end{equation}
which is the sample covariance of the quantized sample $\dot{\mathscr{X}}:=\{\bm{\dot{x}}_1,...,\bm{\dot{x}}_n\}$ followed by a correction step. On the other hand, the reason why the standard uniform dither is not suitable for QCME  becomes self-evident --- the diagonal of $\mathbbm{E}(\bm{\xi}_k\bm{\xi}_k^\top)$ remains unknown\footnote{It depends on the input signal, see \cite[Page 3]{gray1993dithered}.} and hence there is no hope to precisely remove the bias.

We are now ready to present error bounds for $\bm{\widehat{\Sigma}}$ under max-norm, operator norm. We will also investigate the high-dimensional setting by assuming sparse structure of $\bm{\Sigma^\star}$, for which we propose a thresholding estimator. More concretely, our first result provides the error rate under $\|\cdot\|_\infty$, in which we assume $\bm{x}_k$ satisfies the marginal fourth moment constraint and utilize an element-wise truncation $\bm{\widetilde{x}}_k=\mathscr{T}_{\zeta}(\bm{x}_k)$.



\begin{theorem}\label{thm1}
\noindent{\rm(Element-Wise Error){\bf \sffamily.}} Given $\Delta>0$ and $\delta > 4$, we consider the problem of QCME described above. We suppose that $\bm{x}_k$s are i.i.d. zero-mean and satisfy the marginal moment constraint $\mathbbm{E}|x_{ki}|^4\leq M$ for any $i\in [d]$, where $x_{ki}$ is the $i$-th entry of $\bm{x}_k$. We truncate $\bm{x}_k$ to $\bm{\widetilde{x}}_k=[\widetilde{x}_{ki}]=\mathscr{T}_{\zeta}(\bm{x}_k)$ with threshold $\zeta \asymp \big(\frac{nM}{\delta \log d}\big)^{1/4}$, then quantize $\bm{\widetilde{x}}_k$ to $\bm{\dot{x}}_k  = \mathcal{Q}_\Delta(\bm{\widetilde{x}}_k+ \bm{\tau}_k)$ with triangular dither $\bm{\tau}_k\sim \mathscr{U}([-\frac{\Delta}{2},\frac{\Delta}{2}]^d)+\mathscr{U}([-\frac{\Delta}{2},\frac{\Delta}{2}]^d)$. If $n\gtrsim \delta \log d$, then the estimator in (\ref{3.2}) satisfies
\begin{equation}\nonumber
    \mathbbm{P}\left(\|\bm{\widehat{\Sigma}}-\bm{\Sigma^\star}\|_{\infty}\geq C\mathscr{L}\sqrt{\frac{\delta \log d}{n}}\right)\leq  2d^{2-\delta},
\end{equation}
where $\mathscr{L}:=\sqrt{M}+\Delta^2$. 
\end{theorem}

\vspace{2mm}

Notably, despite the heavy-tailedness and quantization, the estimator achieves an element-wise rate $O(\sqrt{\frac{\log d}{n}})$ coincident with the one for sub-Gaussian case. One can clearly position quantization level $\Delta$ in the   multiplicative factor $\mathscr{L}=\sqrt{M}+\Delta^2$. Thus, the information loss incurred by quantization is inessential in that it does not affect the key scaling law but only slightly worsens the leading factor. These remarks on the (near) optimality and the information loss incurred by quantization remain  valid in our subsequent theorems. 

Our next result concerns the operator norm estimation error, under which we impose a stronger joint moment constraint on $\bm{x}_k$ and truncate $\bm{x}_k$ regarding $\ell_4$-norm, i.e., $\bm{\check{x}_k}=\frac{\bm{x}_k}{\|\bm{x}_k\|_4}\min\{\|\bm{x}_k\|_4,\zeta\}$ for some threshold $\zeta$. After the dithered uniform quantization, we still define the estimator as (\ref{3.2}).  
\begin{theorem}\label{thm2}
\noindent{\rm(Operator Norm Error){\bf \sffamily.}} Given $\Delta>0$ and $\delta>0$, we consider the problem of QCME described above.   Suppose that the i.i.d. zero-mean $\bm{x}_k$s satisfy $ \mathbbm{E}|\bm{v}^\top \bm{x}_k|^4 \leq M$  for any $\bm{v}\in \mathbb{S}^{d-1}$. We truncate $\bm{x}_k$ to  $\bm{\check{x}}_k=\frac{\bm{x}_k}{\|\bm{x}_k\|_4}\min \{\|\bm{x}_k\|_4,\zeta\}$ with threshold $\zeta \asymp (M^{1/4}+\Delta)\big(\frac{n}{\delta \log d}\big)^{1/4}$, then quantize $\bm{\check{x}}_k$ to $\bm{\dot{x}}_k=\mathcal{Q}_\Delta (\bm{\check{x}}_k+\bm{\tau}_k)$ with triangular dither $\bm{\tau}_k\sim \mathscr{U}([-\frac{\Delta}{2},\frac{\Delta}{2}]^d)+\mathscr{U}([-\frac{\Delta}{2},\frac{\Delta}{2}]^d)$.     If $n\gtrsim \delta d \log d$, then the estimator in (\ref{3.2}) satisfies \begin{equation}\nonumber
    \mathbbm{P}\left(\|\bm{\widehat{\Sigma}}-\bm{\Sigma^\star}\|_{op}\geq C\mathscr{L}\sqrt{\frac{\delta d\log d}{n}}\right)\leq 2d^{-\delta},
\end{equation} 
with $\mathscr{L}:=\sqrt{M}+\Delta^2$.
\end{theorem}

The operator norm error rate in Theorem \ref{thm2} is near minimax optimal, e.g., compared to the lower bound in \cite[Thm. 7]{fan2021shrinkage}, which states that for any estimator $\bm{\widehat{\Sigma}}$ of the positive semi-definite matrix $\bm{\Sigma^\star}$ based on   i.i.d. zero-mean $\{\bm{x}_k\}_{k=1}^n$ with covariance matrix $\bm{\Sigma^\star}$, there exists some $\bm{v}_0\in \mathbb{S}^{d-1}$ such that $\mathbbm{P}\big(\|\bm{\widehat{\Sigma}}-\bm{\Sigma^\star}\|_{op}\geq \frac{1}{48}\sqrt{\frac{6d}{n}}\big)\geq \frac{1}{3}$, where $\bm{\Sigma^\star}=\bm{I}_d+\bm{v}_0\bm{v}_0^\top$. Again, the quantization only affects the multiplicative factor $\mathscr{L}$. Nevertheless, one still needs (at least) $n \gtrsim d$ to achieve small operator norm error. In fact, in a high-dimensional setting where $d$ may exceed $n$, even the sample covariance $\frac{1}{n}\sum_{k=1}^n\bm{x}_k\bm{x}_k^\top$ for sub-Gaussian zero-mean $\bm{x}_k$ may have extremely bad performance. To achieve small operator norm error in a high-dimensional regime, we resort to additional structure on $\bm{\Sigma^\star}$, and specifically we use column-wise sparsity as an example, which corresponds to the situations where dependencies among different coordinates are  weak. Based on the estimator in Theorem \ref{thm1}, we further invoke a  thresholding regularization \cite{bickel2008covariance,cai2012optimal} to promote sparsity.  


\begin{theorem}\label{thm3}
\noindent{\rm(Sparse QCME)\textbf{.}} Under conditions and estimator $\bm{\widehat{\Sigma}}$ in Theorem \ref{thm1}, we additionally assume that all columns of $\bm{\Sigma}^\star = [\sigma^\star_{ij}]$ are $s$-sparse and consider the thresholding estimator $\bm{\widehat{\Sigma}}_s := \mathcal{T}_{\mu}(\bm{\widehat{\Sigma}})$ for some $\mu$ (recall that $\mathcal{T}_\mu(a)=a\cdot \mathbbm{1}(|a|\geq \mu)$ for $a\in \mathbb{R}$). If  $\mu=C_1 (\sqrt{M}+\Delta^2)\sqrt{\frac{\delta \log d}{n}}$ with   sufficiently large $C_1$, then  $\bm{\widehat{\Sigma}}_s$ satisfies \begin{equation}\nonumber
    \mathbbm{P}\left(\|\bm{\widehat{\Sigma}}_s - \bm{\Sigma^\star}\|_{op} \leq C\mathscr{L}s\sqrt{\frac{ \delta \log d}{n}}\right) \geq 1-\exp(-0.25\delta),
\end{equation}
where $\mathscr{L}:=\sqrt{M}+\Delta^2$.
\end{theorem}
Notably,  our estimator $\bm{\widehat{\Sigma}}_s$ achieves minimax rates $O\big(s\sqrt{\frac{\log d}{n}}\big)$ under operator norm, e.g., compared to the minimax lower bound derived in \cite[Thm. 2]{cai2012optimal}, which states that (under some regular scaling) for any covariance estimator $\bm{\Sigma}_{es}$  based on $n$ i.i.d.  samples of $\mathcal{N}(\bm{\mu},\bm{\Sigma}^\star)$ where $\bm{\Sigma}^\star$ is the  true covariance matrix, there exists some covariance matrix $\bm{\Sigma^\star}$ with  $s$-sparse columns such that $\mathbbm{E}\|\bm{\Sigma}_{es}-\bm{\Sigma}^\star\|_{op}^2\gtrsim s^2\frac{\log d}{n}$.

 To analyse the thresholding estimator, our proof resembles the ones developed in prior works
 (e.g., \cite{cai2012optimal}) but requires more efforts like bounding the additional   bias terms arising from the data truncation and quantization. We also point out that the results for the full-data unquantized regime immediately follow by setting $\Delta=0$, thus Theorems \ref{thm1}-\ref{thm2}   represent the strict extension of   \cite[Sect. 4]{fan2021shrinkage}, and Theorem \ref{thm3} complements \cite{fan2021shrinkage} with a high-dimensional sparse setting.  


\begin{rem}\label{rem2}
\noindent{\rm(Sub-Gaussian Case)\textbf{.}} 
While we concentrate on quantization of heavy-tailed data in this work, our results can be readily adjusted to sub-Gaussian $\bm{x}_k$, for which the truncation step is inessential and can be removed (i.e., $\zeta=\infty$). These results are also new to the literature but will not be presented here. 
\end{rem}

\subsection{Quantized Compressed Sensing}\label{sec4}
We consider the linear   model   \begin{equation}\label{csmodel}
    y_k = \bm{x}_k^\top \bm{\theta^\star}+\epsilon_k,~k=1,...,n,
\end{equation} 
where  $\bm{x}_k$s are the covariates, $y_k$s are responses, $\bm{\theta^\star}$ is the sparse signal in compressed sensing or sparse parameter vector in high-dimensional linear regression that we want to estimate.  
 In the quantized compressed sensing (QCS) problem, we are interested in {\it developing quantization scheme for $(\bm{x}_k,y_k)$s $($mainly for $y_k$ in prior works$)$  that enables accurate recovery of $\bm{\theta^\star}$ based on the quantized data}.

\textcolor{black}{In spite of the same mathematical formulation, there are some important differences between compressed sensing and sparse linear regression that we should clarify first. Specifically, different from   sensing vectors in compressed sensing that are generated by some analog measuring device and can oftentimes be designed,   $\bm{x}_k$s in sparse linear regression represent the sample data from certain datasets that are believed to affect the responses $y_k$s through (\ref{csmodel}). While the sparsity of $\bm{\theta^\star}$ is   arguably the most classical signal structure for compressed sensing, due to good interpretability it is also commonly adopted to achieve dimension reduction in high-dimensional statistics. In this work, we are interested in both problem settings. Thus, we do not adopt the  isotropic convention (i.e., $\mathbbm{E}(\bm{x}_k\bm{x}_k^\top)=\bm{I}_d$) from compressed sensing but instead deal with $\bm{x}_k$ having general unknown covariance matrix. While the study of   quantization and heavy-tailed noise is meaningful in both settings, we   note that some of our subsequent results are  mainly of interest to the specific sensing or regression problem. For instance, the  heavy-tailed covariate considered in Theorem  \ref{thm5} is primarily motivated by the regression setting, in which $\bm{x}_k$ may come from a dataset that exhibits much heavier tail than sub-Gaussian data. Moreover, as will be elaborated in Section  \ref{sec:qc-qcs} when appropriate, our subsequent results on covariate quantization (resp., uniform signal recovery guarantee) may prove more useful to the regression problem (resp., compressed sensing problem).}

To fix idea, we assume that $\bm{x}_k$s are i.i.d. drawn from some multi-variate distribution,  $\epsilon_k$s are i.i.d. statistical noise independent of the $\bm{x}_k$s, and we truncate $y_k$ to $\widetilde{y}_k=\mathscr{T}_{\zeta_y}(y_k)$ and then quantize it to $\dot{y}_k=\mathcal{Q}_\Delta(\widetilde{y}_k+\tau_k)$ with uniform dither $\tau_k\sim \mathscr{U}([-\frac{\Delta}{2},\frac{\Delta}{2}])$.  
Under these statistical assumptions and dithered quantization,     near optimal recovery guarantees have been established in \cite{thrampoulidis2020generalized,xu2020quantized} for the  regime where both $\bm{x}_k$ and $\epsilon_k$ are drawn from sub-Gaussian distributions (hence the truncation is not needed). In contrast, our focus is on quantization of heavy-tailed data. Particularly, we always assume that the noise $\epsilon_k$s are i.i.d. drawn from some heavy-tailed distribution, resulting in heavy-tailed responses. We will separately deal with the case of sub-Gaussian covariate and a more challenging situation where $\bm{x}_k$s are also heavy-tailed.


To estimate the sparse $\bm{\theta^\star}$,  a classical approach is via the regularized M-estimator known as Lasso \cite{tibshirani1996regression,negahban2011estimation,negahban2012unified}
\begin{equation}
    \begin{aligned}\nonumber
    \mathop{\arg\min}\limits_{\bm{\theta}}~\frac{1}{2n}\sum_{k=1}^n(y_k-\bm{x}_k^\top\bm{\theta})^2 + \lambda \|\bm{\theta}\|_1,
    \end{aligned}
\end{equation}
whose objective combines the $\ell_2$-loss for data fidelity and $\ell_1$-norm that encourages sparsity. Because  we can only access the quantized data $(\bm{x}_k,\dot{y}_k)$ (or even $(\bm{\dot{x}}_k,\dot{y}_k)$ if covariate quantization is involved, see \textcolor{black}{Section \ref{sec:qc-qcs}}), the main issue lies in the $\ell_2$-loss $\frac{1}{2n}\sum_{k=1}^n(y_k-\bm{x}_k^\top\bm{\theta})^2$ that requires the   unquantized data $(\bm{x}_k,y_k)$. To resolve the issue, we calculate the expected $\ell_2$-loss: 
\begin{equation}
    \begin{aligned}
        \label{expectloss}
        \mathbbm{E}(y_k-\bm{x}_k^\top\bm{\theta})^2&\stackrel{(i)}{=}\bm{\theta}^\top \mathbbm{E}(\bm{x}_k\bm{x}_k^\top)\bm{\theta}-2 \mathbbm{E}(y_k\bm{x}_k)^\top\bm{\theta}:\\&\stackrel{(ii)}{=}\bm{\theta}^\top\bm{\Sigma^\star}\bm{\theta}-2\bm{\Sigma}_{y\bm{x}}^\top\bm{\theta},
    \end{aligned}
\end{equation}
where $(i)$ holds up to an inessential constant $\mathbbm{E}|y_k|^2$, and in $(ii)$ we let $\bm{\Sigma^\star}:= \mathbbm{E}(\bm{x}_k\bm{x}_k^\top)$, $\bm{\Sigma}_{y\bm{x}}=\mathbbm{E}(y_k\bm{x}_k)$. This inspires us to generalize the $\ell_2$ loss to $\frac{1}{2}\bm{\theta}^\top\bm{Q}\bm{\theta}-\bm{b}^\top\bm{\theta}$ and consider the following program 
\begin{equation}\label{4.3}
   \bm{\widehat{\theta}} =  \mathop{\arg\min}\limits_{\bm{\theta}\in \mathcal{S}}~\frac{1}{2}\bm{\theta}^\top \bm{Q}\bm{\theta} - \bm{b}^\top \bm{\theta} +\lambda\|\bm{\theta}\|_1.
\end{equation}
  Compared to (\ref{expectloss}) we will  use $(\bm{Q},\bm{b})$ that well approximates $(\bm{\Sigma^\star},\bm{\Sigma}_{y\bm{x}})$, and we also introduce the constraint $\bm{\theta}\in \mathcal{S}$ to allow more flexibility. It is important to note that this is the general strategy in this work to design estimators in different QCS settings, see more discussions in Remark \ref{rem:modify}.

The next theorem is concerned with QCS under  sub-Gaussian covariate but heavy-tailed response. Note that the heavy-tailedness of $y_k$ stems from the noise distribution assumed to have bounded $2+\nu$ moment ($\nu=2(l-1)>0$ in the theorem statement), but following \cite{fan2021shrinkage,chen2022high,zhu2021taming} we directly  impose the moment constraint on the response. 


 


\begin{theorem}\label{thm4}
\noindent{\rm(Sub-Gaussian Covariate, Heavy-Tailed Response)\textbf{.}} Given some $\delta>0,\Delta>0$, in  (\ref{csmodel}) we suppose that   $\bm{x}_k$s are i.i.d., zero-mean   sub-Gaussian with $\|\bm{x}_k\|_{\psi_2}\leq \sigma$, $\kappa_0\leq \lambda_{\min}(\bm{\Sigma^\star})\leq \lambda_{\max}(\bm{\Sigma^\star})\leq \kappa_1$ for some $\kappa_1>\kappa_0>0$ where $\bm{\Sigma^\star}=\mathbbm{E}(\bm{x}_k\bm{x}_k^\top)$, $\bm{\theta^\star}\in\mathbb{R}^d$ is   $s$-sparse, the noise $\epsilon_k$s are i.i.d. heavy-tailed and independent of $\bm{x}_k$s, and we assume \textcolor{black}{$\mathbbm{E}|y_k|^{2l}\leq M$ for some fixed $l>1$}. In the quantization, we truncate $y_k$ to $\widetilde{y}_k= \mathscr{T}_{\zeta_y}(y_k)$ with threshold $\zeta_y\asymp \big(\frac{nM^{1/l}}{\delta \log d}\big)^{1/2}$, then quantize $\widetilde{y}_k$ to $\dot{y}_k = \mathcal{Q}_\Delta(\widetilde{y}_k+\tau_k)$ with uniform dither $\tau_k\sim \mathscr{U}([-\frac{\Delta}{2},\frac{\Delta}{2}])$. For recovery, we define the estimator $\bm{\widehat{\theta}}$  as (\ref{4.3}) with $\bm{Q} = \frac{1}{n}\sum_{k=1}^n\bm{x}_k\bm{x}_k^\top$, $\bm{b} = \frac{1}{n}\sum_{k=1}^n\dot{y}_k\bm{x}_k$, $\mathcal{S} = \mathbb{R}^d$.  We   set  $\lambda = C_1 \frac{\sigma^2}{\sqrt{\kappa_0}}(\Delta +M^{1/(2l)})\sqrt{\frac{\delta\log d}{n}}$ with sufficiently large $C_1$.  If $n \gtrsim \delta s\log d$ for some hidden constant only depending on $(\kappa_0,\sigma)$, then with probability at least $1-9d^{1-\delta}$, the estimation error $\bm{\widehat{\Upsilon}} = \bm{\widehat{\theta}} - \bm{\theta^\star}$ satisfies\begin{equation}\nonumber
    \|\bm{\widehat{\Upsilon}}\|_2\leq C_3\mathscr{L} \sqrt{\frac{\delta s\log d}{n}}~~~\mathrm{and}~~~\|\bm{\widehat{\Upsilon}}\|_1 \leq C_4\mathscr{L}s\sqrt{\frac{\delta\log d}{n}}~
\end{equation}
where $\mathscr{L} := \frac{\sigma^2(\Delta+M^{1/(2l)})}{\kappa_0^{3/2}}$.
\end{theorem}

 The rate $O\big(\sqrt{\frac{s\log d}{n}}\big)$ for $\ell_2$-norm error is   minimax optimal up to logarithmic factor (e.g., compared to \cite{raskutti2011minimax}). Note that a  random noise 
 bounded by $\Delta$   roughly contributes $\Delta$ to $(\mathbbm{E}|y_k|^{2l})^{1/(2l)}$, and the latter is bounded by $M^{1/(2l)}$; because in the error bound $\Delta$ and $M^{1/(2l)}$ almost play the same role, the effect of uniform quantization  can be  readily interpreted as an additional bounded noise, analogously to the error rate in \cite{sun2022quantized}. 
 
Next, we switch to the more challenging situation where both 
 $\bm{x}_k$ and $y_k$ are heavy-tailed,  assuming that they both possess bounded fourth moments (a marginal moment constraint for $\bm{x}_k$). \textcolor{black}{The consideration of this setting is motivated by the setting of sparse linear regression, where the covariates $\bm{x}_k$s may oftentimes exhibit heavy-tailed behaviour.} Specifically, we element-wisely truncate $\bm{x}_k$ to $\bm{\widetilde{x}}_k$ and   set $\bm{Q}:= \frac{1}{n}\sum_{k=1}^n\bm{\widetilde{x}}_k\bm{\widetilde{x}}_k^\top$ as a robust covariance matrix estimator, whose estimation performance under $\|\cdot\|_{\infty}$ follows immediately from  Theorem \ref{thm1} by setting $\Delta=0$. 
\begin{theorem}
\label{thm5}\noindent{\rm(Heavy-Tailed Covariate, Heavy-Tailed Response)\textbf{.}} Given some $\delta>0$, $\Delta>0$, in (\ref{csmodel}) we suppose that $\bm{x}_k$s are i.i.d. zero-mean satisfying a marginal fourth moment constraint $\sup_{i\in [d]}\mathbbm{E}|x_{ki}|^4\leq M$,  $\kappa_0\leq \lambda_{\min}(\bm{\Sigma^\star})\leq \lambda_{\max}(\bm{\Sigma^\star})\leq \kappa_1$ for some $\kappa_1>\kappa_0>0$  where $\bm{\Sigma^\star}=\mathbbm{E}(\bm{x}_k\bm{x}_k^\top)$, $\bm{\theta^\star}\in\Sigma_s$   satisfies $\|\bm{\theta^\star}\|_1\leq R$, the noise $\epsilon_k$s are i.i.d. heavy-tailed and independent of $\bm{x}_k$s, and we assume $\mathbbm{E}|y_k|^4\leq M$.   In the quantization,   we truncate $\bm{x}_k,y_k$ respectively to $\bm{\widetilde{x}}_k = [\widetilde{x}_{ki}]=\mathscr{T}_{\zeta_x}(\bm{x}_k),~\widetilde{y}_k:= \mathscr{T}_{\zeta_y}(y_k)$ with $\zeta_x, \zeta_y\asymp\big(\frac{nM}{\delta \log d}\big)^{1/4}$, then we quantize $\widetilde{y}_k$ to $\dot{y}_k = \mathcal{Q}_\Delta(\widetilde{y}_k+ \tau_k)$ with uniform dither $\tau_k\sim\mathscr{U}([-\frac{\Delta}{2},\frac{\Delta}{2}])$. For recovery, we define the estimator $\bm{\widehat{\theta}}$ as (\ref{4.3}) with $\bm{Q}=\frac{1}{n}\sum_{k=1}^n\bm{\widetilde{x}}_k\bm{\widetilde{x}}_k^\top$, $\bm{b}=\frac{1}{n}\sum_{k=1}^n\dot{y}_k\bm{\widetilde{x}}_k$, $\mathcal{S}=\mathbb{R}^d$. We set   $\lambda = C_1(R\sqrt{M}+\Delta^2)\sqrt{\frac{\delta \log d}{n}}$ with sufficiently large $C_1$. If $n\gtrsim \delta s^2\log d$ for some hidden constant only depending on $(\kappa_0,M)$, then with probability at least $1-4d^{2-\delta}$, the estimation error $\bm{\widehat{\Upsilon}}:= \bm{\widehat{\theta}} - \bm{\theta^\star}$  satisfies \begin{equation}\nonumber
      \|\bm{\widehat{\Upsilon}}\|_2\leq C_2\mathscr{L}  \sqrt{\frac{\delta s\log d}{n}}~~~\mathrm{and}~~~\|\bm{\widehat{\Upsilon}}\|_1 \leq C_3\mathscr{L}s\sqrt{\frac{\delta \log d}{n}}~
\end{equation} 
where $\mathscr{L}:=\frac{R\sqrt{M}+\Delta^2}{\kappa_0}$.
\end{theorem}

\vspace{1mm}

Theorem \ref{thm5} generalizes   \cite[Thm. 2(b)]{fan2021shrinkage} to the uniform  quantization setting. Clearly,  the obtained rate   remains  near minimax optimal if $R$ is of minor scaling (e.g., bounded or logarithmic factors). Nevertheless,   such near optimality in Theorem \ref{thm5} comes at the cost of more restricted conditions and stronger scaling, as remarked in the following.

\begin{rem}\label{htcova} {\rm  (Comparing Theorems \ref{thm4}-\ref{thm5})\textbf{.}} 
Compared with $n\gtrsim s\log d$ in Theorem \ref{thm4}, the first downside of Theorem \ref{thm5} is the sub-optimal sample complexity $n\gtrsim s^2\log d$, and note that $n\gtrsim s^2\log d$ is also required in \cite[Thm. 2(b)]{fan2021shrinkage}. But indeed, it can be improved to  $n\gtrsim s\log d$   by explicitly adding the constraint $\|\bm{\theta}\|_1\leq R$ to the recovery program, as will be noted as an interesting side finding in Remark \ref{impro}. Secondly, following \cite{fan2021shrinkage} we impose an $\ell_1$-norm constraint $\|\bm{\theta^\star}\|_1\leq R$  that is stronger than $\|\bm{\theta^\star}\|_2\lesssim\frac{M^{1/(2l)}}{\sigma}$ used in the proof of Theorem \ref{thm4}. In fact, when replacing the $\ell_1$ constraint in Theorem \ref{thm5} with an $\ell_2$-norm bound $\|\bm{\theta^\star}\|_2\leq R$, then  our proof technique leads to an error rate   $\|\bm{\widehat{\Upsilon}}\|_2=O\big(\sqrt{\frac{s^2\log d}{n}}\big)$ that exhibits worse dependence on $s$. 
\end{rem}
\begin{rem}\label{rem:modify}
    {\rm  (Modification of $\ell_2$-loss)\textbf{.}} Recall that we generalize the regular $\ell_2$-loss $\frac{1}{2n}\sum_{k=1}^n(y_k-\bm{x}_k^\top\bm{\theta})^2$  to $\frac{1}{2}\bm{\theta}^\top \bm{Q\theta}-\bm{b}^\top\bm{\theta}$ as loss function in (\ref{4.3}). Note that the choice of $(\bm{Q},\bm{b})$ in Theorem \ref{thm4} is tantamount to using the loss function $\frac{1}{2n}\sum_{k=1}^n(\dot{y}_k-\bm{x}_k^\top\bm{\theta})^2$ that replaces $y_k$ with the quantized response $\dot{y}_k$; this idea is analogous to the generalized Lasso investigated for  single index model \cite{plan2016generalized} and dithered quantized model \cite{thrampoulidis2020generalized}, and will be used again in  quantized matrix completion, see (\ref{5.2}) below. However, our generalized $\ell_2$-loss provides more flexibility to deal with heavy-tailedness or quantization of $\bm{x}_k$, e.g., $(\bm{Q},\bm{b})$ in Theorem \ref{thm5} amounts to adopting $\frac{1}{2n}\sum_{k=1}^n({\dot{y}_k}-\bm{\widetilde{x}}_k^\top\bm{\theta})^2$ as loss function, and under quantized covariate  more delicate modifications are required in Theorems \ref{thm6}-\ref{ht1bitqccs}, which is beyond the range of prior works on generalized Lasso. 
\end{rem}

\subsection{Quantized Matrix Completion}\label{sec5}
Completing a low-rank matrix from only a partial observation of its entries is known as the matrix completion problem, which has found many applications including recommendation system, image inpainting, quantum state tomography   \cite{chen2022color,davenport2016overview,bennett2007netflix,nguyen2019low,gross2010quantum}, to name just a few. Mathematically,  let $\bm{\Theta^\star}\in \mathbb{R}^{d\times d}$ be the underlying matrix satisfying $\rank(\bm{\Theta^\star})\leq r$, the matrix completion problem can be formulated as
\begin{equation}\label{mcmodel}
    y_k = \big<\bm{X}_k, \bm{\Theta^\star}\big> + \epsilon_k, ~k=1,2,...,n,
\end{equation}
where  $\bm{X}_k$s are distributed on $\mathcal{X}:=\{\bm{e}_i\bm{e}_j^\top: i,j\in [d]\}$ ($\bm{e}_i$ is the $i$-th column of $\bm{I}_d$), $\epsilon_k$ is observation noise. Note that for $\bm{X}_k=\bm{e}_{i(k)}\bm{e}_{j(k)}^\top$ one has $\big<\bm{X}_k,\bm{\Theta^\star}\big> = \theta^\star_{i(k),j(k)}$, so each observation is a noisy entry. Our main interest is on quantized matrix completion (QMC), where our goal is to {\it design quantizer for the observation $y_k$ that allows for accurate  estimation of $\bm{\Theta^\star}$ from the quantized observations}.

Unlike in compressed sensing, additional condition (besides the low-rankness) on $\bm{\Theta^\star}$ is needed to ensure the well-posedness of the matrix completion problem. More specifically, certain incoherence conditions are required if we pursue exact recovery (e.g., \cite{candes2012exact,recht2011simpler}), whereas a faithful estimation   can be achieved as long as the underlying matrix is not overly spiky and sufficiently diffuse (e.g., \cite{klopp2014noisy,negahban2012restricted}). The latter condition is also known as "low spikiness" and is formulated by $\frac{d\|\bm{\Theta^\star}\|_\infty}{\|\bm{\Theta^\star}\|_F}\leq \alpha$  \cite{fan2021shrinkage,negahban2012restricted}, which has been noted to be necessary for the well-posedness of matrix completion problem \cite{davenport2016overview,negahban2012restricted}. In subsequent works    the low-spikiness condition is often formulated   as the simpler max-norm constraint $\|\bm{\Theta^\star}\|_\infty\leq \alpha$ \cite{klopp2014noisy,chen2022color,klopp2017robust,davenport20141,foucart2020weighted}.

 In this work, we consider the uniform sampling scheme $\bm{X}_k\sim \mathscr{U}(\mathcal{X})$, but with a little bit more work it generalizes to more general sampling scheme  \cite{klopp2014noisy}.     We apply the proposed quantization scheme to possibly heavy-tailed $y_k$ --- we truncate $y_k$ to $\widetilde{y}_k=\mathscr{T}_{\zeta_y}(y_k)$ with some threshold $\zeta_y$, and then quantize $\widetilde{y}_k$ to $\dot{y}_k=\mathcal{Q}_\Delta(\widetilde{y}_k+\tau_k)$ with uniform dither $\tau_k\sim \mathscr{U}([-\frac{\Delta}{2},\frac{\Delta}{2}])$. Because we do not pursue exact recovery (which is impossible under quantization), we do not assume any incoherence condition like \cite{recht2011simpler}. Instead, we only hope to accurately estimate $\bm{\Theta^\star}$, and following \cite{klopp2014noisy,chen2022color,klopp2017robust,davenport20141,foucart2020weighted} we impose a max-norm constraint  $$\|\bm{\Theta^\star}\|_{\infty}\leq \alpha.$$  Overall,  we estimate $\bm{\Theta^\star}$ from $(\bm{X}_k,\dot{y}_k)$ by the regularized M-estimator \cite{negahban2011estimation,negahban2012unified}
\begin{equation}\label{5.2}
    \bm{\widehat{\Theta}} = \mathop{\arg\min}\limits_{\|\bm{\Theta}\|_\infty\leq \alpha}~\frac{1}{2n}\sum_{k=1}^n \big(\dot{y}_k- \big<\bm{X}_k,\bm{\Theta}\big>\big)^2+\lambda\|\bm{\Theta}\|_{nu}
\end{equation} 
that combines an $\ell_2$-loss and nuclear norm regularizer.

In the literature, there has been a line of works on 1-bit or multi-bit matrix completion  related to our results to be presented   \cite{cai2013max,lafond2014probabilistic,klopp2015adaptive,cao2015categorical,bhaskar2016probabilistic}. While the referenced works commonly adopted a likelihood approach, our method  is an essential departure and embraces some advantage, see a precise comparison in Remark \ref{rem8}. Considering such novelty, we include   the result for sub-exponential $\epsilon_k$ in Theorem \ref{thm8}, for which the truncation of $y_k$ becomes unnecessary and  we simply set $\zeta_y=\infty$.

\begin{theorem}
\label{thm8}{\rm (QMC under Sub-Exponential Noise)\textbf{.}} Given some $\Delta>0,\delta>0$, in (\ref{mcmodel}) we suppose that $\bm{X}_k$s are i.i.d. uniformly distributed over $\mathcal{X}=\{\bm{e}_i\bm{e}_j^\top:i,j\in [d]\}$, $\bm{\Theta^\star}\in \mathbb{R}^{d\times d}$   satisfies $\rank(\bm{\Theta^\star})\leq r$ and $\|\bm{\Theta^\star}\|_\infty\leq \alpha$, the     noise $\epsilon_k$s are i.i.d. zero-mean sub-exponential satisfying $\|\epsilon_k\|_{\psi_1}\leq \sigma$, and are independent of $\bm{X}_k$s. In the quantization, we do not truncation $y_k$ but directly quantize it to $\dot{y}_k=\mathcal{Q}_\Delta(y_k+\tau_k)$ with uniform dither $\tau_k\sim \mathscr{U}([-\frac{\Delta}{2},\frac{\Delta}{2}])$. We choose $\lambda =C_1 (\sigma+\Delta) \sqrt{\frac{\delta\log d}{nd}}$ with sufficiently large $C_1$, and define $\bm{\widehat{\Theta}}$ as (\ref{5.2}). If $\delta d\log^3 d\lesssim n\lesssim \delta r^2d^2 \log d$, then with probability at least $1-4d^{-\delta}$, the estimation error $\bm{\widehat{\Upsilon}}:=\bm{\widehat{\Theta}} - \bm{\Theta^\star}$ satisfies\begin{equation}\nonumber
    \frac{\|\bm{\widehat{\Upsilon}}\|_F}{d}\leq C_2 \mathscr{L}\sqrt{\frac{\delta r d \log d}{n}} ~~\mathrm{and}~~ \frac{\|\bm{\widehat{\Upsilon}}\|_{nu}}{d} \leq  C_3\mathscr{L}r\sqrt{\frac{\delta   d \log d}{n}}
\end{equation}
where $\mathscr{L}:= \alpha+\sigma+\Delta $. 
\end{theorem}

By contrast, under heavy-tailed noise only assumed to have bounded variance,  we truncate $y_k$ with a suitable threshold before the dithered quantization to achieve an optimal trade-off between bias and variance.

\begin{theorem}
\label{thm9}{\rm (QMC under Heavy-tailed Noise){\bf \sffamily.}} Given some $\Delta>0,\delta>0$, we consider (\ref{mcmodel})  in the setting of Theorem \ref{thm8} but with the assumption $\|\epsilon_k\|_{\psi_1}\leq \sigma$ replaced by $\mathbbm{E}|\epsilon_k|^2\leq M$. In the quantization, we  truncate $y_k$ to $\widetilde{y}_k=\mathscr{T}_{\zeta_y}(y_k)$ with $\zeta_y\asymp (\sqrt{M}+\alpha)\sqrt{\frac{n}{\delta d \log d}}$, and then quantize $\widetilde{y}_k$ to $\dot{y}_k=\mathcal{Q}_\Delta(\widetilde{y}_k+\tau_k)$ with uniform dither $\tau_k\sim \mathscr{U}([-\frac{\Delta}{2},\frac{\Delta}{2}])$. We choose  $\lambda =C_1 (\alpha+\sqrt{M}+\Delta) \sqrt{\frac{\delta\log d}{nd}}$ with sufficiently large $C_1$, and define $\bm{\widehat{\Theta}}$ as (\ref{5.2}). If $\delta d\log d\lesssim n\lesssim \delta r^2d^2 \log d$, then with probability at least $1-6d^{-\delta}$, the estimation error  $\bm{\widehat{\Upsilon}}:=\bm{\widehat{\Theta}} - \bm{\Theta^\star}$ satisfies  \begin{equation}\nonumber
    \frac{\|\bm{\widehat{\Upsilon}}\|_F}{d}\leq C_2 \mathscr{L}\sqrt{\frac{\delta r d \log d}{n}} ~~\mathrm{and}~~ \frac{\|\bm{\widehat{\Upsilon}}\|_{nu}}{d} \leq  C_3\mathscr{L}r\sqrt{\frac{\delta   d \log d}{n}}
\end{equation}
where $\mathscr{L}:=\alpha+\sqrt{M}+\Delta$. 
\end{theorem}

\vspace{1mm}

Compared to the information-theoretic lower bounds in  \cite{negahban2012restricted,koltchinskii2011nuclear}, the error rates obtained in Theorems \ref{thm8}-\ref{thm9} are   minimax optimal up to logarithmic factors. Specifically, Theorem \ref{thm9} derives near optimal guarantee for QMC with heavy-tailed observations, as the key standpoint of this paper. Note that, the 1-bit quantization counterpart of these two Theorems was derived in our previous work \cite{chen2022high}; in sharp contrast to  Theorem \ref{thm9}, for 1-bit QMC under heavy-tailed noise, the error rate under $\frac{\|\bm{\widehat{\Upsilon}}\|_F}{d}$  in \cite[Thm. 13]{chen2022high} reads as $O\big(\big(\frac{r^2d\log d}{n}\big)^{1/4}\big)$ and 
 is essentially slower; using the 1-bit observations therein, this slow error rate is indeed nearly tight due to the lower bound in \cite[Thm. 14]{chen2022high}. 

To close this section, we  give a remark to illustrate the novelty and advantage of our QMC method by a careful comparison with prior works.  

\begin{rem}
\label{rem8} 
  QMC with  1-bit or multi-bit  quantized observations has received considerable research interest  \cite{davenport20141,cai2013max,lafond2014probabilistic,klopp2015adaptive,cao2015categorical,bhaskar2016probabilistic}. Adapted to our notation, these works studied the model $\dot{y}_k = \mathcal{Q}(\big<\bm{X}_k,\bm{\Theta^\star}\big>+\tau_k)$ under general random dither $\tau_k$ and quantizer $\mathcal{Q}(.)$, and they commonly adopted regularized (or constrained) maximum likelihood estimation  for estimating $\bm{\Theta^\star}$. By contrast,   with   the random dither and quantizer specialized to 
$\tau_k\sim \mathscr{U}([-\frac{\Delta}{2},\frac{\Delta}{2}])$ and $\mathcal{Q}_\Delta(.)$, our model is formulated as $\dot{y}_k=\mathcal{Q}_\Delta(\mathscr{T}_{\zeta_y}(\big<\bm{X}_k,\bm{\Theta^\star}\big>+\epsilon_k)+\tau_k)$. Thus, while suffering from less generality in $(\tau_k,\mathcal{Q})$, our method embraces the   advantage of robustness to pre-quantization noise $\epsilon_k$, whose distribution is unknown and can even be heavy-tailed. Note that such unknown $\epsilon_k$ evidently forbids the likelihood approach. 
\end{rem}
 
\section{Covariate Quantization and Uniform Signal Recovery in Quantized Compressed Sensing}\label{sec:qc-qcs}
By now we have presented near optimal results in the contexts of QCME, QCS and QMC  under heavy-tailed data that further undergo the proposed quantization scheme, which we position as the primary contribution of this work. In this section, we further   provide two additional developments  to enhance our results on heavy-tailed QCS. 
\subsection{Covariate Quantization}\label{sec:covarquan}
In the area of QCS, almost all prior works  merely focused on the quantization of response $y_k$, see the recent survey \cite{dirksen2019quantized}; here, we consider a setting of "complete quantization" --- meaning that the covariate $\bm{x}_k$ is also quantized. To motivate our study of "complete quantization", we interpret compressed sensing as sparse linear regression. Indeed, to reduce the power consumption and computational cost, it is sometimes preferable to work with low-precision data in a machine learning system, e.g.,   the sample quantization scheme      developed in \cite{zhang2017zipml} led to   experimental success in training linear model. Also, it was shown  that direct gradient quantization may not be efficient in certain distributed learning systems where the terminal nodes are connected to the server only through very weak communication fabric and the number of parameters are extremely huge; rather, quantizing and transmitting some important samples could provably reduce communication cost \cite{hanna2021quantization}. 
In fact, the process of data collection may already appeal to quantization  due to certain limit of the data acquisition device (e.g., a low-resolution analog-to-digital module used in distributed signal processing \cite{danaee2022distributed}). Our main goal is to understand how quantization of $(\bm{x}_k,y_k)$s affects the subsequent recovery/learning process, particularly showing that the simple dithered uniform quantization scheme still allows for accurate estimator that may even provide near minimax error rate. To our best knowledge, the only prior rigorous estimation guarantees for QCS with covariate quantization are \cite[Thms. 7-8]{chen2022high}; these two results require a restricted and unnatural assumption, which we will also relax later.




\subsubsection{Multi-bit QCS with Quantized Covariate}
 Since we will also consider the 1-bit quantization, we more precisely refer to the QCS under uniform quantizer  as multi-bit QCS. We will generalize Theorems \ref{thm4}-\ref{thm5} to covariate quantization in the next two theorems.

 Let $(\bm{\dot{x}}_k,\dot{y}_k)$ be the quantized covariate-response pair, we first quickly sketch the idea of our approach. Specifically, we stick to the framework of M-estimator in (\ref{4.3}), which appeals to accurate surrogates for $\bm{\Sigma^\star} =\mathbbm{E}(\bm{x}_k\bm{x}_k^\top)$ and $\bm{\Sigma}_{y\bm{x}}=\mathbbm{E}(y_k\bm{x}_k)$ based on $(\bm{\dot{x}}_k,\dot{y}_k)$, where $\bm{\dot{x}_k}$ represents the quantized covariate. Fortunately, the surrogates can be constructed analogously to our  QCME estimator when triangular dither is used for quantizing $\bm{x}_k$. Let us first state our quantization scheme as follows:  
 \begin{itemize}
 [leftmargin=5ex,topsep=0.15ex]
 \setlength\itemsep{-0.1em}
     \item \textbf{Response Quantization.} This is the same as Theorems \ref{thm4}-\ref{thm5}. We truncate $y_k$  to $\widetilde{y}_k = \mathscr{T}_{\zeta_y}(y_k)$ with threshold $\zeta_y$, and  then quantize $\widetilde{y}_k$ to $\dot{y}_k = \mathcal{Q}_\Delta(\widetilde{y}_k+ \phi_k)$ with uniform dither $\phi_k\sim \mathscr{U}([-\frac{\Delta}{2},\frac{\Delta}{2}])$ and quantization level $\Delta\geq 0$. 

     \item \textbf{Covariate Quantization.} This is the same as Theorem \ref{thm1}. We truncate $\bm{x}_k$   to $\bm{\widetilde{x}}_k= \mathscr{T}_{\zeta_x}(\bm{x}_k)$ with threshold $\zeta_x$, and then quantize $\bm{\widetilde{x}}_k$ to $\bm{\dot{x}}_k = \mathcal{Q}_{\bar{\Delta}}(\bm{\widetilde{x}}_k+\bm{\tau}_k)$ with triangular dither $\bm{\tau}_k\sim \mathscr{U}([-\frac{\bar{\Delta}}{2},\frac{\bar{\Delta}}{2}]^d)+\mathscr{U}([-\frac{\bar{\Delta}}{2},\frac{\bar{\Delta}}{2}]^d)$ and quantization level $\bar{\Delta}\geq 0$.

     \item \textbf{Notation.} We write the quantization noise as $\varphi_k=\dot{y}_k-\widetilde{y}_k$ and $\bm{\xi}_k = \bm{\dot{x}}_k- \bm{\widetilde{x}}_k$, the quantization error as $\vartheta_k= \dot{y}_k-(\widetilde{y}_k+\phi_k)$ and $\bm{w}_k=\bm{\dot{x}}_k-(\bm{\widetilde{x}}_k+ \bm{\tau}_k)$. 
 \end{itemize}
We will adopt the above notation in subsequent developments. Based on the quantized covariate-response pairs $(\bm{\dot{x}}_k,\dot{y}_k)$s, we specify  our estimator by setting $(\bm{Q},\bm{b})$ in (\ref{4.3}) as\begin{equation}\label{quanlasso}
    \bm{Q}=\frac{1}{n}\sum_{k=1}^n \bm{\dot{x}}_k\bm{\dot{x}}_k^\top- \frac{\bar{\Delta}^2}{4}\bm{I}_d~~\mathrm{and}~~\bm{b}=\frac{1}{n}\sum_{k=1}^n \dot{y}_k\bm{\dot{x}}_k.
\end{equation}
Note that  the choice of $\bm{Q}$ is due to the estimator in Theorem \ref{thm1}, while $\bm{b}$ is inspired by the calculation 
\begin{equation}
    \begin{aligned}\nonumber
        &\mathbbm{E}(\dot{y}_k\bm{\dot{x}}_k)=\mathbbm{E}\big((\widetilde{y}_k+\varphi_k)(\bm{\widetilde{x}}_k+\bm{\xi}_k)\big)\\
        &=\mathbbm{E}(\widetilde{y}_k\bm{\widetilde{x}}_k)+\mathbbm{E}(\widetilde{y}_k\bm{\xi}_k)+\mathbbm{E}(\varphi_k\bm{\widetilde{x}}_k)+\mathbbm{E}(\varphi_k\bm{\xi}_k)=\mathbbm{E}(\widetilde{y}_k\bm{\widetilde{x}}_k),
    \end{aligned}
\end{equation}
where the last equality can be seen by conditioning on $\bm{\widetilde{x}}_k$ or $\widetilde{y}_k$.
  However, the issue is  that $\bm{Q}$ is not positive semi-definite, hence the resulting program is non-convex. 
  To explain this,   note that the rank of $\frac{1}{n}\sum_{k=1}^n\bm{\dot{x}}_k\bm{\dot{x}}_k^\top$ does not exceed $n$, so when $d>n$ at least $d-n$ eigenvalues of $\bm{Q}$ are $-\frac{\bar{\Delta}^2}{4}$. Alternatively, the non-convexity can also be seen from the observation that setting $(\bm{Q},\bm{b})$ as in (\ref{quanlasso}) is tantamount to replacing the regular $\ell_2$-loss $\frac{1}{2n}\sum_{k=1}^n (y_k-\bm{x}_k^\top\bm{\theta})^2$  with \begin{equation}
      \nonumber
      \frac{1}{2n}\sum_{k=1}^n (\dot{y}_k-\bm{\dot{x}}_k^\top\bm{\theta})^2-\frac{\bar{\Delta}}{8}\|\bm{\theta}\|_2^2.
  \end{equation} We mention that the lack of positive semi-definiteness of $\bm{Q}$ is problematic in both statistics and optimization aspects: 1) Statistically,   Lemma  \ref{csframework} used to derive the error rates in Theorems \ref{thm4}-\ref{thm5} requires $\bm{Q}$ to be positive semi-definite, and is hence no longer applicable here; 2) From the optimization side, it is in general unknown how to globally optimize a non-convex program. 

Motivated by a   line of previous works on non-convex M-estimator \cite{loh2011high,loh2013regularized,loh2017statistical}, we add an $\ell_1$-norm constraint to (\ref{4.3}) by setting $\mathcal{S}=\{ \bm{\theta}\in \mathbb{R}^d:\|\bm{\theta}\|_1\leq R\}$, where $R$ represents the prior estimation on $\|\bm{\theta^\star}\|_1$. Let $\partial\|\bm{\theta}_1\|_1$ be a subdifferential of $\|\bm{\theta}\|_1$ at $\bm{\theta}=\bm{\theta}_1$,\footnote{Thus, $\partial\|\bm{\widetilde{\theta}}\|_1$ in (\ref{4.17}) below should be understood as "there exists one element in $\partial\|\bm{\widetilde{\theta}}\|_1$ such that (\ref{4.17}) holds."} we consider the local minimizer of the proposed recovery program,\footnote{The existence of local minimizer is guaranteed because of the additional $\ell_1$-constraint.} or more  generally put, $\bm{\widetilde{\theta}}\in \mathcal{S}$ that satisfies\footnote{To distinguish the global minimizer in (\ref{4.3}), we denote  by $\bm{\widetilde{\theta}}$ the estimator in QCS with quantized covariate.} \begin{equation}
    \big<\bm{Q}\bm{\widetilde{\theta}}-\bm{b}+ \lambda \cdot \partial \|\bm{\widetilde{\theta}}\|_1, \bm{\theta} - \bm{\widetilde{\theta}}\big> \geq 0,~~\forall~\bm{\theta}\in \mathcal{S}.\label{4.17}
\end{equation}  
We will prove a fairly strong guarantee stating that all $\bm{\widetilde{\theta}}\in\mathcal{S}$ satisfying (\ref{4.17}) (of course including all local minimizers) enjoy  near minimax error rate. 
While   this guarantee bears resemblance to the ones in \cite{loh2013regularized}, we point out that, \cite{loh2013regularized} only derived concrete results for sub-Gaussian regime; because of the heavy-tailed data and quantization in our setting,   some essentially different ingredients are required  for the technical analysis (see {Remark} \ref{rem4}). 
As before, our results for  sub-Gaussian $\bm{x}_k$ and heavy-tailed $\bm{x}_k$ are presented separately.


\begin{theorem}\label{thm6}
{\rm  (Quantized Sub-Gaussian Covariate)\textbf{.}} Given $\Delta\geq 0$, $\bar{\Delta}\geq 0$, $\delta>0$, we consider (\ref{csmodel}) with the same assumptions on $(\bm{x}_k,y_k,\bm{\theta^\star})$ as  Theorem \ref{thm4},  
and additionally assume that  $\|\bm{\theta^\star}\|_2\leq R$. The quantization of $(\bm{x}_k,y_k)$ is  described above, and we set $\zeta_x=\infty$, $\zeta_y \asymp  \sqrt{\frac{nM^{1/l}}{\delta \log d}}$.
For recovery, we let $\bm{Q} = \frac{1}{n}\sum_{k=1}^n\bm{\dot{x}}_k\bm{\dot{x}}_k^\top -\frac{\bar{\Delta}^2}{4}\bm{I}_d$, $\bm{b}= \frac{1}{n}\sum_{k=1}^n\dot{y}_k\bm{\dot{x}}_k$, $\mathcal{S} = \{\bm{\theta}:\|\bm{\theta}\|_1\leq R\sqrt{s}\}$ and set $\lambda = C_1\frac{(\sigma+\bar{\Delta})^2}{\sqrt{\kappa_0}}(\Delta+M^{1/(2l)})\sqrt{\frac{\delta\log d}{n}}$ with sufficiently large $C_1$. If $n\gtrsim \delta s\log d$ for some hidden constant only depending on $(\kappa_0,\sigma,\Delta,\bar{\Delta},M,R)                $, 
with probability at least $1- 8d^{1-\delta}-C_2\exp(-C_3n)$,  all $\bm{\widetilde{\theta}}\in \mathcal{S}$ satisfying (\ref{4.17}) have  estimation error $\bm{\widetilde{\Upsilon}}:=\bm{\widetilde{\theta}}-\bm{\theta^\star}$ bounded by 
\begin{equation}\nonumber
    \|\bm{\widetilde{\Upsilon}}\|_2\leq C \mathscr{L}\sqrt{\frac{\delta s\log d}{n}} ~~\mathrm{and}~~\|\bm{\widetilde{\Upsilon}}\|_1\leq C' \mathscr{L}s\sqrt{\frac{\delta \log d}{n}}
\end{equation}
where $ \mathscr{L} :=\frac{(\sigma+\bar{\Delta})^2(\bm{\Delta}+M^{1/(2l)})}{\kappa_0^{3/2}}$.
\end{theorem}

\vspace{1mm}


Similarly, the next result extends Theorem \ref{thm5} to a setting involving covariate quantization.

\begin{theorem}\label{thm7}
{\rm  (Quantized Heavy-Tailed Covariate)\textbf{.}} Given $\Delta\geq 0$, $\bar{\Delta}\geq 0$, $\delta>0$, we consider (\ref{csmodel}) with the same assumptions on $(\bm{x}_k,y_k,\bm{\theta^\star})$ as  Theorem  \ref{thm5}. 
The quantization of $(\bm{x}_k,y_k)$ is  described above, and we set $\zeta_x,\zeta_y\asymp \big(\frac{nM}{\delta\log d}\big)^{1/4}$.
For recovery, we let $\bm{Q}=\frac{1}{n}\sum_{k=1}^n \bm{\dot{x}}_k\bm{\dot{x}}_k^\top- \frac{\bar{\Delta}^2}{4}\bm{I}_d$, $\bm{b}= \frac{1}{n}\sum_{k=1}^n\dot{y}_k\bm{\dot{x}}_k$, $\mathcal{S}=\{\bm{\theta}:\|\bm{\theta}\|_1\leq R\}$ 
and set  
$\lambda =C_1 (R\sqrt{M}+\Delta^2+R\bar{\Delta}^2)\sqrt{\frac{\delta \log d}{n}}$ with sufficiently large $C_1$. If $n\gtrsim \delta s\log d$ for some hidden constant only depending on $(\kappa_0,M)$, then with probability at least $1- 8d^{1-\delta}$,  all $\bm{\widetilde{\theta}}\in \mathcal{S}$ satisfying (\ref{4.17}) have   estimation error $\bm{\widetilde{\Upsilon}}:=\bm{\widetilde{\theta}}-\bm{\theta^\star}$ bounded by \begin{equation}\nonumber
    \|\bm{\widetilde{\Upsilon}}\|_2\leq C_3\mathscr{L}\sqrt{\frac{\delta s\log d}{n}} ~~\mathrm{and}~~\|\bm{\widetilde{\Upsilon}}\|_1\leq C_4\mathscr{L}s\sqrt{\frac{\delta \log d}{n}}.
\end{equation}  
where $\mathscr{L}:=\frac{R\sqrt{M}+\Delta^2+R\bar{\Delta}^2}{\kappa_0}$.
\end{theorem}

\begin{rem}\label{rem4}
{\rm  (Comparing Our Analyses with \cite{loh2013regularized}){\bf \sffamily.}} The above results are motivated by a line of works on nonconvex M-estimator \cite{loh2011high,loh2013regularized,loh2017statistical}, and our guarantee for the whole set of stationary points (\ref{4.17}) resembles \cite{loh2013regularized} most. While the main strategy for proving Theorem \ref{thm6} is adjusted from \cite{loh2013regularized},   the proof of Theorem \ref{thm7} does involve an essentially different RSC condition, see our (\ref{4.23}). In particular, compared with  \cite[equation (4)]{loh2013regularized}, the leading factor of $\|\bm{\widetilde{\Upsilon}}\|_1^2$ in (\ref{4.23}) degrades from $O\big(\frac{\log d}{n}\big)$ to $O\big(\sqrt{\frac{\log d}{n}}\big)$. To retain near optimal rate we need to impose a stronger scaling $\|\bm{\theta^\star}\|_1\leq R$  with proper changes in the proof. Although Theorem \ref{thm7} is presented for a   concrete setting, it sheds light on   extension of \cite{loh2013regularized} to a weaker RSC condition that could accommodate covariate with heavier tail. Such extension is formally presented as a deterministic framework in Proposition \ref{framework}.  
\end{rem}
\begin{pro}
\label{framework}
Suppose that the $s$-sparse $\bm{\theta^\star}\in \mathbb{R}^d$  satisfies $\|\bm{\theta^\star}\|_1\leq R$, and the  positive definite matrix  $\bm{\Sigma^\star}\in \mathbb{R}^{d\times d}$ satisfies $\lambda_{\min}(\bm{\Sigma^\star})\geq \kappa_0$. If for some $\bm{Q}\in \mathbb{R}^{d\times d},\bm{b}\in \mathbb{R}^d$ we have \begin{equation}\label{4.26}
    \lambda \geq C_1\max\big\{\|\bm{Q\theta^\star} - \bm{b}\|_\infty,~ R\cdot\|\bm{Q}-\bm{\Sigma^\star}\|_\infty\big\}
\end{equation}
holds for sufficiently large $C_1$, then all $\bm{\widetilde{\theta}}\in \mathcal{S}$ satisfying (\ref{4.17}) with $\mathcal{S} = \{\bm{\theta}\in \mathbb{R}^d:\|\bm{\theta}\|_1\leq R\}$ have estimation error $\bm{\widetilde{\Upsilon}}:=\bm{\widetilde{\theta}}-\bm{\theta^\star}$ bounded by 
\begin{equation}
    \begin{aligned} \nonumber
      \|\bm{\widetilde{\Upsilon}}\|_2\leq C_2 \frac{\sqrt{s}\lambda}{\kappa_0} ~~\mathrm{and}~~\|\bm{\widetilde{\Upsilon}}\|_1\leq C_3 \frac{s\lambda}{\kappa_0}~.
    \end{aligned}
\end{equation} 
\end{pro}

By extracting the ingredients that guarantee (\ref{4.17}) to be accurate, interestingly, Proposition \ref{framework} is now independent of the model assumption (\ref{csmodel}).  Particularly, we could set $\bm{\Sigma^\star}=\mathbbm{E}[\bm{x}_k\bm{x}_k^\top]$ when we apply Proposition \ref{framework} to (\ref{csmodel}). Compared with the framework \cite[Thm. 1]{loh2013regularized}, the key strength of   Proposition \ref{framework} is that it does not explicitly assume the  RSC condition on the loss function that is hard to verify without assuming sub-Gaussian covariate.  
Instead, the role of the RSC assumption in \cite{loh2013regularized} is now played by $\lambda\gtrsim R\|\bm{Q}-\bm{\Sigma^\star}\|_\infty$,  which immediately yields a kind of  RSC condition by simple argument as (\ref{4.24}). Although this RSC condition is often essentially weaker than the conventional one  in terms of the leading factor of $\|\bm{\widetilde{\Upsilon}}\|_1^2$ (see   Remark \ref{rem4}), along this line one can still   derive non-trivial (or even near optimal) error rate. The gain of replacing RSC assumption with $\lambda\gtrsim R\|\bm{Q}-\bm{\Sigma^\star}\|_\infty$ is that the latter amounts  to 
constructing element-wise estimator for $\bm{\Sigma^\star}$, which is often much easier for heavy-tailed covariate (e.g., due to many existing robust covariance estimator).

We conclude this part with a side interesting observation.
\begin{rem}
\label{rem5} \label{impro}By setting $\bar{\Delta}=0$, Theorem \ref{thm7} produces a result (with convex program) for the setting of Theorem \ref{thm5}.  Interestingly, with the additional $\ell_1$-constraint, a notable improvement is that the sub-optimal $n\gtrsim s^2\log d$ in Theorem \ref{thm5} is sharpened to the near optimal one in Theorem \ref{thm7}.  
More concretely, this is because (ii) in (\ref{4.8}) can be  tightened  to $(ii)$ of     (\ref{4.24}). Going back to the full-data unquantized regime,   Theorem \ref{thm7} with $\Delta=\bar{\Delta}=0$ recovers \cite[Theorem 2(b)]{fan2021shrinkage} with improved sample complexity requirement. 
\end{rem}
   
\subsubsection{1-bit QCS with Quantized Covariate} \label{sec4.3}
Our consideration of covariate quantization in QCS seems fairly new to the literature. To the best of our knowledge, the only related results are  \cite[Thms. 7-8]{chen2022high} for QCS with 1-bit quantized covariate and response. The assumption there, however, is quite restrictive. Specifically, it is assumed that $\bm{\Sigma^\star}=\mathbbm{E}(\bm{x}_k\bm{x}_k^\top)$ has sparse columns (see \cite[Assumption 3]{chen2022high}), which is non-standard in both compressed sensing and sparse linear regression.  
Departing momentarily from our focus of dithered uniform quantization, we consider QCS under dithered 1-bit quantization  and   will apply Proposition \ref{framework} to derive results comparable to \cite[Thms. 7-8]{chen2022high} without resorting to the sparsity of $\bm{\Sigma^\star}$.  

We first review the 1-bit quantization scheme developed in \cite{chen2022high}: 
\begin{itemize}
[leftmargin=5ex,topsep=0.15ex]
 \setlength\itemsep{-0.1em}
    \item \textbf{Response Quantization.} We truncate $y_k$ to $\widetilde{y}_k = \mathscr{T}_{\zeta_y}(y_k)$ with some threshold $\zeta_y$, and then quantize $\widetilde{y}_k$ to $\dot{y}_k = \sign(\widetilde{y}_k+\phi_k)$ with uniform dither $\phi_k\sim \mathscr{U}([-\gamma_y,\gamma_y])$.

    \item \textbf{Covariate Quantization.} We truncate $\bm{x}_k$ to $\bm{\widetilde{x}}_k=\mathscr{T}_{\zeta_x}(\bm{x}_k)$ with some threshold $\zeta_x$, and then quantize $\bm{\widetilde{x}}_k$ to $\bm{\dot{x}}_{k1}= \sign(\bm{\widetilde{x}}_k+\bm{\tau}_{k1})$ and $\bm{\dot{x}}_{k2}= \sign(\bm{\widetilde{x}}_k+\bm{\tau}_{k2})$, where $\bm{\tau}_{k1},\bm{\tau}_{k2}\sim \mathscr{U}([-\gamma_x,\gamma_x]^d)$ are independent uniform dithers. (Note that we collect 2 bits per  entry).

   
\end{itemize}

The following two results refine \cite[Thms. 7-8]{chen2022high} by deriving comparable error rates without using sparsity of $\bm{\Sigma^\star}$.

\begin{theorem}\label{sg1bitqccs}
{\rm  (1-bit Quantized Sub-Gaussian Covariate)\textbf{.}} Given $\delta>0$, we consider (\ref{csmodel}) where the $s$-sparse $\bm{\theta^\star}$ satisfies $\|\bm{\theta^\star}\|_1\leq R$, $\bm{x}_k$s are i.i.d. zero-mean sub-Gaussian with $\|\bm{x}_k\|_{\psi_2}\leq \sigma$, and $\bm{\Sigma^\star}=\mathbbm{E}(\bm{x}_k\bm{x}_k^\top)$ satisfies $\lambda_{\min}\big(\bm{\Sigma^\star}\big)\geq \kappa_0$   for some $\kappa_0>0$, the noise ${\epsilon}_k$s are independent of $\bm{x}_k$s and  i.i.d. sub-Gaussian, while for simplicity we assume $\|y_k\|_{\psi_2}\leq \sigma$. In the quantization of $(\bm{x}_k,y_k)$ described above, we set $\zeta_x=\zeta_y=\infty$ and $\gamma_x,\gamma_y\asymp \sigma \sqrt{\log\big(\frac{n}{2\delta \log d}\big)}$.  For   recovery we let $\bm{Q}:=\frac{\gamma_x^2}{2n}\sum_{k=1}^n\big(\bm{\dot{x}}_{k1}\bm{\dot{x}}_{k2}^\top+\bm{\dot{x}}_{k2}\bm{\dot{x}}_{k1}^\top\big)$, $\bm{b}:=\frac{\gamma_x\gamma_y}{n}\sum_{k=1}^n \dot{y}_k \bm{\dot{x}}_{k1}$, $\mathcal{S}=\{\bm{\theta}:\|\bm{\theta}\|_1\leq R\}$ and set   $\lambda = C_1\sigma^2R \sqrt{\frac{\delta \log d(\log n)^2}{n}}$   with sufficiently large $C_1$. If $n\gtrsim \delta s\log d(\log n)^2$, then with probability at least $1-4d^{2-\delta}$, all  $\bm{\widetilde{\theta}}\in\mathcal{S}$ satisfying (\ref{4.17}) have   estimation error $\bm{\widetilde{\Upsilon}}:=\bm{\widetilde{\theta}}-\bm{\theta^\star}$ bounded by \begin{equation}\nonumber
    \|\bm{\widetilde{\Upsilon}}\|_2\leq C_2 \frac{\sigma^2}{\kappa_0}\cdot R\sqrt{\frac{\delta s\log d (\log n)^2}{n}} ~~\mathrm{and}~~\|\bm{\widetilde{\Upsilon}}\|_1\leq  C_3\frac{\sigma^2}{\kappa_0}\cdot Rs\sqrt{\frac{\delta \log d (\log n)^2}{n}}.
\end{equation}
\end{theorem}

\begin{theorem}\label{ht1bitqccs}
{\rm  (1-bit Quantized   Heavy-Tailed Covariate)\textbf{.}} Given $\delta>0$, we consider (\ref{csmodel}) where the $s$-sparse $\bm{\theta^\star}$ satisfies $\|\bm{\theta^\star}\|_1\leq R$, $\bm{x}_k$s are i.i.d. zero-mean heavy-tailed satisfying the joint fourth moment constraint $\sup_{\bm{v}\in \mathbb{S}^{d-1}}\mathbbm{E}|\bm{v}^\top \bm{x}_k|^4\leq M$, and $\bm{\Sigma^\star}=\mathbbm{E}(\bm{x}_k\bm{x}_k^\top)$ satisfies  $\lambda_{\min}\big(\bm{\Sigma^\star}\big)\geq \kappa_0$   for some $\kappa_0>0$, the noise $\epsilon_k$s are independent of $\bm{x}_k$s and i.i.d. heavy-tailed with bounded fourth moment, while for simplicity we assume $\mathbbm{E}|y_k|^4\leq M$.    In the   quantization of $(\bm{x}_k,y_k)$ described above, we set $\zeta_x,\zeta_y,\gamma_x,\gamma_y \asymp \big(\frac{nM^2}{\delta \log d}\big)^{1/8}$ and enforce  $\zeta_x <\gamma_x$, $\zeta_y<\gamma_y$. For   recovery we let $\bm{Q}:=\frac{\gamma_x^2}{2n}\sum_{k=1}^n\big(\bm{\dot{x}}_{k1}\bm{\dot{x}}_{k2}^\top+\bm{\dot{x}}_{k2}\bm{\dot{x}}_{k1}^\top\big)$, $\bm{b}:=\frac{\gamma_x\gamma_y}{n}\sum_{k=1}^n \dot{y}_k \bm{\dot{x}}_{k1}$, $\mathcal{S}=\{\bm{\theta}:\|\bm{\theta}\|_1\leq R\}$ and set $\lambda = C_1\sqrt{M}R\big(\frac{\delta \log d}{n}\big)^{1/4}$ with sufficiently large $C_1$.  If $n\gtrsim \delta s^2\log d$, then with probability at least $1-4d^{2-\delta}$,
all   $\bm{\widetilde{\theta}}\in \mathcal{S}$ satisfying (\ref{4.17}) have estimation error $\bm{\widetilde{\Upsilon}}:=\bm{\widetilde{\theta}}-\bm{\theta^\star}$ bounded by \begin{equation}\nonumber
    \|\bm{\widetilde{\Upsilon}}\|_2\leq C_2\frac{\sqrt{M}}{\kappa_0}\cdot R \Big(\frac{\delta s^2\log d}{n}\Big)^{1/4} ~~\mathrm{and}~~\|\bm{\widetilde{\Upsilon}}\|_1\leq C_3\frac{\sqrt{M}}{\kappa_0}\cdot Rs\Big(\frac{\delta \log d}{n}\Big)^{1/4}.
\end{equation}
\end{theorem}
\subsection{Uniform Recovery Guarantee}

Uniformity is a highly desired property for a compressed sensing guarantee because it allows one to use a fixed (possibly randomly drawn) measurement ensemble  for all sparse signals. Unfortunately, as many other results for nonlinear compressed sensing in the literature, our earlier recovery guarantees are non-uniform and only ensure the accurate recovery of a sparse signal fixed before drawing the random measurement ensemble.

 We provide another additional development to   QCS in this part. Specifically, we establish a uniform recovery guarantee which, despite the heavy-tailed noise and nonlinear quantization scheme, notably retains a near minimax error rate. This is done by upgrading Theorem \ref{thm4} to be uniform over all sparse $\bm{\theta^\star}$ by more in-depth technical tools and a careful covering argument. Part of the techniques is inspired by prior works \cite{genzel2022unified,xu2020quantized}, but certain technical innovations are required:


 1) Like the recent work \cite{genzel2022unified}, one crucial technical tool in our proof is a powerful concentration inequality for product process due to Mendelson \cite{mendelson2016upper}, as adapted in the present Lemma \ref{productpro}. However, \cite{genzel2022unified} only studied sub-Gaussian distribution, and the results produced by their unified approach typically exhibit a decaying rate of $O(n^{-1/4})$ \cite[Sect. 4]{genzel2022unified}. By contrast, our problem involves heavy-tailed noise only having bounded $(2+\nu)$-th moment ($\nu>0$), and we aim to establish a near minimax uniform error bound --- cautiousness and new treatment are thus needed in the application of Lemma \ref{productpro}. More specifically,
  in the proof we need to bound 
$$I_1= \sup_{\bm{\theta} \in \Sigma_{s,R_0}}\sup_{\bm{v}\in \mathscr{V}}\sum_{k=1}^n \big(\widetilde{y}_k\bm{x}_k^\top \bm{v}-\mathbbm{E}[\widetilde{y}_k\bm{x}_k^\top \bm{v}]\big),$$
where $\mathscr{V}=\{\bm{v}:\|\bm{v}\|_2=1,\|\bm{v}\|_1\leq 2\sqrt{s}\}$, and $\Sigma_{s,R_0}=\Sigma_s\cap \{\bm{\theta}\in \mathbb{R}^d:\|\bm{\theta}\|_2\leq R_0\}$ is the signal space of interest, and recall that $\widetilde{y}_k=\mathscr{T}_{\zeta_y}(\bm{x}_k^\top\bm{\theta}+\epsilon_k)$ with sub-Gaussian $\bm{x}_k$. It is natural to invoke Lemma \ref{productpro} to bound $I_1$ straightforwardly, but the issue is on lack of good bound for $\|\widetilde{y}_k\|_{\psi_2}$ due to the heavy-tailedness of $\epsilon_k$; indeed,  one only has   the trivial estimate as
$\|\widetilde{y}_k\|_{\psi_2}=O(\zeta_y)$,  which is much worse than an $O(1)$ bound since $\zeta_y\asymp \sqrt{\frac{n}{\delta \log d}}$, and   using Lemma \ref{productpro} with this estimate leads to a loose bound for $I_1$ and finally a sub-optimal error rate. To address the issue, our main idea is to introduce the truncated heavy-tailed noise $\mathscr{T}_{\zeta_y}(\epsilon_k)$ and define $\widetilde{z}_k=\widetilde{y}_k-\mathscr{T}_{\zeta_y}(\epsilon_k)$, which enables us to decompose $I_1$ as  
\begin{equation}
    \nonumber
    I_1\leq \underbrace{\sup_{\bm{\theta}\in \Sigma_{s,R_0}}\sup_{\bm{v}\in\mathscr{V}}\sum_{k=1}^n \big(\widetilde{z}_k\bm{x}_k^\top\bm{v}-\mathbbm{E}[\widetilde{z}_k\bm{x}_k^\top\bm{v}]\big)}_{:=I_{11}}+\underbrace{\sup_{\bm{v}\in \mathscr{V}} \sum_{k=1}^n \big(\mathscr{T}_{\zeta_y}(\epsilon_k)\bm{x}_k^\top\bm{v}-\mathbbm{E}[\mathscr{T}_{\zeta_y}(\epsilon_k)\bm{x}_k^\top\bm{v}]\big)}_{:=I_{12}}. 
\end{equation}
Now, the benefits of working with $I_{11},I_{12}$ are that: i) We can directly invoke Lemma \ref{productpro} to bound $I_{11}$  since we have a good sub-Gaussian norm estimate   $\|\widetilde{z}_k\|_{\psi_2}\leq \|\bm{x}_k^\top\bm{\theta}\|_{\psi_2}\lesssim\|\bm{x}_k\|_{\psi_2}R_0$, see Step 2.1.1 in the proof; ii) $I_{12}$ becomes the supremum of a process that is independent of $\bm{\theta}$ and only indexed by $\bm{v}$, hence Bernstein's inequality suffices for bounding $I_{12}$ (Step 2.1.2 in the proof), analogously to the proof of the non-uniform guarantee (Theorem \ref{thm4}).

2) Like  \cite[Prop. 6.1]{xu2020quantized}, we invoke a covering argument with similar techniques to bound $I_0=\sup_{\bm{\theta}\in \Sigma_{s,R_0}}\sup_{\bm{v}\in \mathscr{V}}\sum_{k=1}^n\xi_k\bm{x}_k^\top\bm{v}$, where $\xi_k=\mathcal{Q}_\Delta(\widetilde{y}_k+\tau_k)-\widetilde{y}_k$ is the quantization noise. Nevertheless, our Lasso estimator is different from their projected back projection estimator, and it turns out that we need to directly handle "$\sup_{\bm{v}\in \mathscr{V}}$"  by Lemma \ref{talagrand}, unlike    \cite[Prop. 6.2]{xu2020quantized} that again used a covering argument for this purpose. See more discussions in Step 2.4 of the proof.

We are in a position to present our uniform recovery guarantee. We follow most assumptions in Theorem \ref{thm4} but specify the signal space as $\bm{\theta^\star}\in \Sigma_{s,R_0}$ and impose the $(2l)$-th moment constraint on $\epsilon_k$.   Following prior works on QCS (e.g., \cite{genzel2022unified,thrampoulidis2020generalized}),  we consider constrained Lasso that utilizes an $\ell_1$-constraint $\|\bm{\theta}\|_1\leq \|\bm{\theta}^\star\|_1$ (rather than (\ref{4.3})) to pursue uniform recovery. 

\begin{theorem}
    \label{uniformtheorem}
    {\rm  (Uniform Version of Theorem \ref{thm4})\textbf{.}} Given some $\delta>0,\Delta>0$, in (\ref{csmodel}) we 
      suppose  that $\bm{x}_k$s are i.i.d.,  zero-mean sub-Gaussian with $\|\bm{x}_k\|_{\psi_2}\leq \sigma$,      $\kappa_0\leq \lambda_{\min}(\bm{\Sigma^\star})\leq \lambda_{\max}(\bm{\Sigma^\star})\leq \kappa_1$  for some $\kappa_1\geq \kappa_0>0$ where $\bm{\Sigma^\star}=\mathbbm{E}(\bm{x}_k\bm{x}_k^\top)$, $\bm{\theta^\star}\in \Sigma_{s,R_0}:= \Sigma_s \cap \{\bm{\theta}:\|\bm{\theta}\|_2\leq R_0\}$ for some absolute constant $R_0$, $\epsilon_k$s are i.i.d. noise that are independent of $\bm{x}_k$s and satisfy $\mathbbm{E}|\epsilon_k|^{2l}\leq M$ for some fixed $l>1$.     In quantization, we truncate 
$y_k$ to $\widetilde{y}_k=\mathscr{T}_{\zeta_y}(y_k)$ with threshold $\zeta_y\asymp \big(\frac{n(M^{1/l}+\sigma^2)}{\delta \log d}\big)^{1/2}$, then quantize $\widetilde{y}_k$ to $\dot{y}_k = \mathcal{Q}_\Delta(\widetilde{y}_k+\tau_k)$ with uniform dither $\tau_k\sim \mathscr{U}([-\frac{\Delta}{2},\frac{\Delta}{2}])$. For recovery, we define the estimator $\bm{\widehat{\theta}}$ as the solution to constrained Lasso \begin{equation}\nonumber
        \bm{\widehat{\theta}}= \mathop{\arg\min}\limits_{\|\bm{\theta}\|_{1}\leq \|\bm{\theta^\star}\|_1}~\frac{1}{2n}\sum_{k=1}^n(\dot{y}_k-\bm{x}_k^\top\bm{\theta})^2 
    \end{equation}
    If $n \gtrsim \delta s\log \mathscr{W}$ for $\mathscr{W}= \frac{\kappa_1d^2n^3}{\Delta^2s^5\delta^3}$ and some hidden constant depending on $(\kappa_0,\sigma)$,   then  with probability at least $1-Cd^{1-\delta}$ on a single random draw of $(\bm{x}_k,\epsilon_k, \tau_k)_{k=1}^n$, it holds uniformly for all $\bm{\theta^\star}\in\Sigma_{s,R_0}$ that the estimation error  $\bm{\widehat{\Upsilon}}:=\bm{\widehat{\theta}}-\bm{\theta^\star}$ satisfy 
    \begin{gather}
        \|\bm{\widehat{\Upsilon}}\|_2 \leq \frac{C_3\sigma(\sigma+M^{\frac{1}{2l}})}{\kappa_0}\sqrt{\frac{\delta s\log d}{n}}+\frac{C_3\sigma\Delta}{\kappa_0}\sqrt{\frac{\delta s \log \mathscr{W}}{n}},\nonumber\\
        \|\bm{\widehat{\Upsilon}}\|_1\leq \frac{C_4\sigma(\sigma+M^{\frac{1}{2l}})}{\kappa_0}s\sqrt{\frac{\delta \log d}{n}}+\frac{C_4\sigma\Delta}{\kappa_0}s\sqrt{\frac{\delta\log \mathscr{W}}{n}}.\nonumber
    \end{gather}
\end{theorem}

 Notably, our uniform guarantee is still minimax optimal up to some additional logarithmic factors (i.e., $\sqrt{\log\mathscr{W}}$) arising from the covering argument (Step 2.4 of the proof), whose main aim is to show that one uniform dither $\bm{\tau}=[\tau_k]$ suffices for all signals. Thus naturally, $\sqrt{\log \mathscr{W}}$ is associated with a leading factor of the quantization level $\Delta$, meaning that the logarithmic gap between uniform recovery and non-uniform recovery closes when $\Delta\to 0$. In particular, Theorem \ref{uniformtheorem} implies a uniform error rate matching the non-uniform one  in Theorem \ref{thm4} (up to some multiplicative factors) when $\Delta$ is small  enough or in  an unquantized case.  


To the best of our knowledge, the only existing uniform guarantee for heavy-tailed QCS is \cite[Thm. 1.11]{dirksen2021non}, but the following  distinctions make it impossible to closely compare their result with our Theorem \ref{uniformtheorem}: 1) \cite[Thm. 1.11]{dirksen2021non} is for dithered 1-bit quantization, but ours is for dithered uniform quantizer; 2) We handle heavy-tailedness by truncation, while 
\cite[Thm. 1.11]{dirksen2021non} does not involve this kind of special treatment; 3) \cite[Thm. 1.11]{dirksen2021non} considers a highly intractable program with hamming distance as objective and $\bm{\theta}\in \Sigma_{s}$ as constraint (when specialized to sparse signal), while our Theorem \ref{uniformtheorem} is for the convex program Lasso; 4) Their analysis is based on an in-depth result on random hyperplane tessellations (see also \cite{dirksen2022sharp,plan2014dimension}), while our proof follows the more standard strategy (i.e., to upgrade each piece in a non-uniform proof to be uniform) and requires certain technical innovations (e.g., the treatment to deal with the truncation step). It is possible to use such a standard strategy to upgrade Theorem \ref{thm5} to a uniform result whose error rate may exhibit worse dependence on $s$ due to covering argument.




\section{Numerical Simulations}\label{sec6}
In this section we provide two sets of experimental results to support and demonstrate our theoretical developments. The first set of our simulations is devoted to validate our major standpoint that near minimax rates are achievable in  quantized heavy-tailed settings. Then, the second set of results are presented to illustrate the crucial role played by  the appropriate dither (i.e., triangular dither for covariate, uniform dither for response) before uniform quantization. For the importance of data truncation we refer to   in \cite[Sect. 5]{fan2021shrinkage}, which includes three estimation problems in this work and contrasts the estimations with or without the data truncation. 

\subsection{(Near) Minimax Error Rates}
Each data point in our results is set to be the mean value of $50$ or $100$ independent trials. 



\subsubsection{Quantized Covariance Matrix Estimation}
We start from covariance matrix estimation, specifically we verify the element-wise rate $\mathscr{B}_1:=O\big(\mathscr{L}\sqrt{\frac{\log d}{n}}\big)$ and operator norm rate $\mathscr{B}_2:=O\big(\mathscr{L}\sqrt{\frac{d\log d}{n}}\big)$ in Theorems \ref{thm1}-\ref{thm2}.

For estimator in Theorem \ref{thm1}, we draw $\bm{x}_k = (x_{ki})$ such that the first two coordinates are independently drawn from $\mathsf{t}(4.5)$, $(x_{ki})_{i=3,4}$ are from $\mathsf{t}(6)$ with covariance $\mathbbm{E}(x_{k3}x_{k4})=1.2$, and the remaining $d-4$ coordinates are i.i.d. following  $\mathsf{t}(6)$. We test different choices of $(d,\Delta)$  under $n=80:20:220$, and the log-log plots are shown in Figure \ref{fig1}(a). Clearly, for each $(d,\Delta)$ the experimental points roughly exhibit a straight line that is well aligned with the dashed line representing the $n^{-1/2}$ rate. As predicted by the factor $\mathscr{L}=\sqrt{M}+\Delta^2$, the curves with larger $\Delta$ are higher, but note that the error decreasing rates remain unchanged. In addition, the curves of $(d,\Delta)=(100,1),(120,1)$ are extremely close, which is consistent with the logarithmic dependence of $\mathscr{B}_1$  on  $d$.

For the error bound $\mathscr{B}_2$, the coordinates of $\bm{x}_k$ are independently drawn from a scaled version of $\mathsf{t}(4.5)$ such that $\bm{\Sigma^\star}=\mathrm{diag}(2,2,1,...,1)$, and we test different settings of $(d,\Delta)$ under $n= 200:100:1000$. As shown in Figure \ref{fig1}(b), the operator norm error decreases with $n$ in the optimal rate $n^{-1/2}$, and using a coarser dithered quantizer (i.e., larger $\Delta$) only slightly lifts the curves. Indeed, the effect seems consistent with $\mathscr{L}$'s quadratic dependence on $\Delta$. To validate the relative scaling of $n$ and $d$, in addition to the setting $(d,\Delta)=(100,1)$ under $n=200:100:1000$,
we try $(d,\Delta)= (150,1)$ under $1.5$ times  the original sample size $n=1.5\times(200:100:1000)$ (but in Figure \ref{fig1}(b) we still plot the curve according to the sample size of $200:100:1000$ without the multiplicative factor of $1.5$), and surprisingly the obtained curve coincides with the one for $(d,\Delta)=(100,1)$. Thus, ignoring the logarithmic factor $\log d$, the operator norm error can be characterized by $\mathscr{B}_2$ fairly well.

Additionally, we want to compare $\mathscr{B}_1$ and $\mathscr{B}_2$ regarding the dependence on $d$ more clearly. Specifically, we generate the samples $\bm{x}_k$s as in Figure \ref{fig1}(a) and test the fixed sample size  $n=180$ and varying dimension $d=80:20:260$. The max norm estimation errors of $\bm{\widehat{\Sigma}}$ in Theorem \ref{thm1} and the operator norm errors (under $d=80:20:180$ to ensure $n\geq d$) of the estimator in Theorem \ref{thm2} are reported in Figure \ref{fig1}(c). It is clear that the max norm error increases with $d$ rather slowly, while the operator norm error increases much more significantly under larger $d$. This is consistent with the logarithmic dependence of $\mathscr{B}_1$ on $d$ and the more essential dependence of $\mathscr{B}_2$ on $d$.    

\begin{figure}[ht!]
    \centering
    \includegraphics[scale = 0.6]{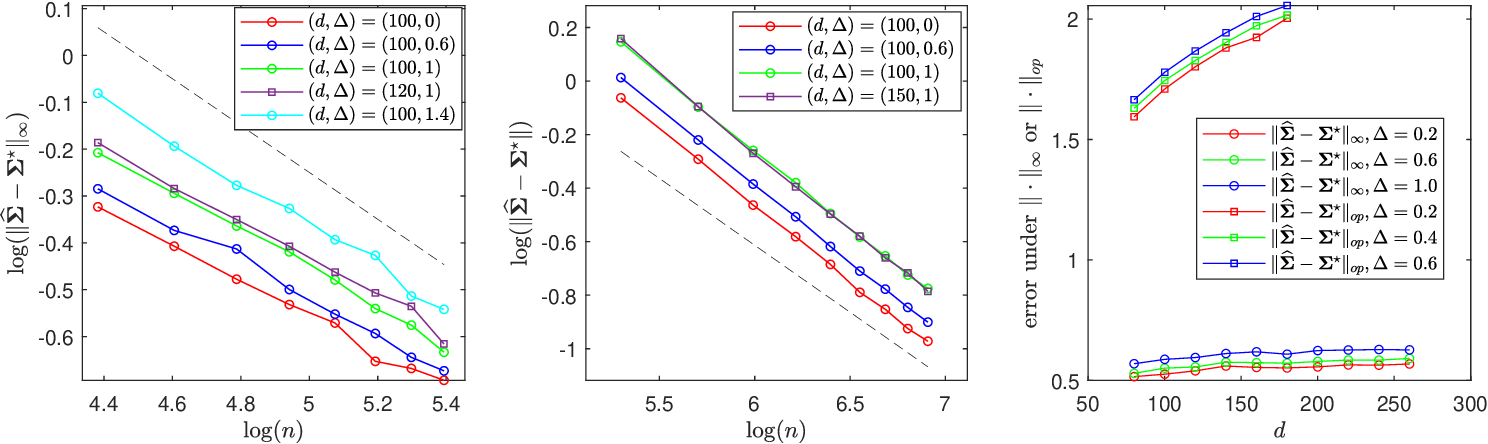}
    
 ~~~   (a) \hspace{4.3cm} (b)\hspace{4.3cm}(c) 
    \caption{(a): Element-wise error (Theorem \ref{thm1}); (b): operator norm error (Theorem \ref{thm2}); (c): the dependence on $d$ of both error metrics.}
    \label{fig1}
\end{figure}

\subsubsection{Quantized Compressed Sensing}

We now switch to QCS with unquantized covariate and aim to verify the $\ell_2$-norm error rate $\mathscr{B}_3=O\big(\mathscr{L}\sqrt{\frac{s\log d}{n}}\big)$ obtained in Theorems \ref{thm4}-\ref{thm5}. We let the support of the $s$-sparse   $\bm{\theta^\star}\in \mathbb{R}^d$ be $[s]$, and then draw the non-zero entries from a uniform distribution over $\mathbb{S}^{s-1}$ (hence $\|\bm{\theta^\star}\|_2=1$). For the setting of Theorem \ref{thm4} we adopt $\bm{x}_k\sim \mathcal{N}(0,\bm{I}_d)$ and $\epsilon_k\sim \frac{1}{\sqrt{6}}\mathsf{t}(3)$, while $\bm{x}_{ki}\stackrel{iid}{\sim}\frac{\sqrt{5}}{3}\mathsf{t}(4.5)$ and $\epsilon_k\sim \frac{1}{\sqrt{3}}\mathsf{t}(4.5)$ for Theorem \ref{thm5}. We simulate different choices of $(d,s,\Delta)$ under $n=100:100:1000$, and the proposed convex program (\ref{4.3}) is solved with the framework of ADMM (we refer to the review \cite{boyd2011distributed}). Experimental results are shown as   log-log plots in Figure \ref{fig2}. Consistent with the theoretical bound $\mathscr{B}_3$, the errors in both cases decrease in a rate of $n^{-1/2}$, whereas the effect of uniform quantization is merely on the multiplicative factor $\mathscr{L}$. Interestingly, it seems that the gaps between $\Delta=0,0.5$ and $\Delta = 0.5,1$ are in agreement with the explicit form of $\mathscr{L}$, i.e., $\mathscr{L}\asymp M^{1/(2l)}+\Delta$ for Theorem \ref{thm4}, and $\mathscr{L}\asymp \sqrt{M}+\Delta^2$ for Theorem \ref{thm5}. In addition, note that the curves of $(d,s)=(150,5),(180,5)$ are close, whereas increasing $s=8$ suffers from significantly larger error. This is consistent with the scaling law of  $(n,d,s)$ in $\mathscr{B}_3$.

\begin{figure}[ht!]
    \centering
    \includegraphics[scale = 0.62]{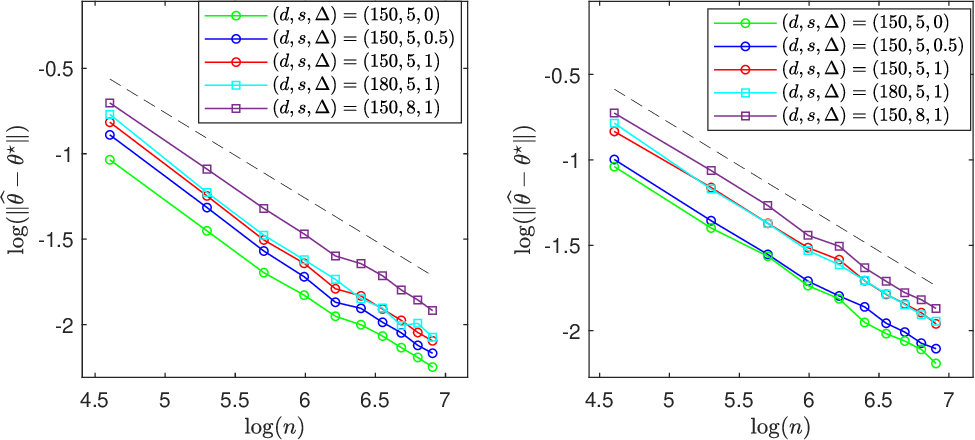}
    
 ~~~   (a) \hspace{4.4cm} (b) 
    \caption{(a): QCS in Theorem \ref{thm4}; (b): QCS in Theorem \ref{thm5}.}
    \label{fig2}
\end{figure}

Then, we simulate the complete quantization setting where both covariate and response are quantized (Theorems \ref{thm6}-\ref{thm7}). The simulation details are the same as before except that $\bm{x}_k$ is also quantized with quantization level same as $y_k$. We provide  the best $\ell_1$-norm constraint for recovery, i.e., $\mathcal{S}:=\{\bm{\theta}:\|\bm{\theta}\|_1\leq\|\bm{\theta^\star}\|_1 \}$. Then,   composite gradient descent \cite{loh2011high,loh2013regularized} is invoked to handle the non-convex estimation program. We show the log-log plots in Figure \ref{fig3}. Note that these results have implications similar to Figure \ref{fig2}, in terms of the $n^{-1/2}$ rate, the effect of quantization, and the relative scaling of $(n,d,s)$.

\begin{figure}[ht!]
    \centering
    \includegraphics[scale = 0.58]{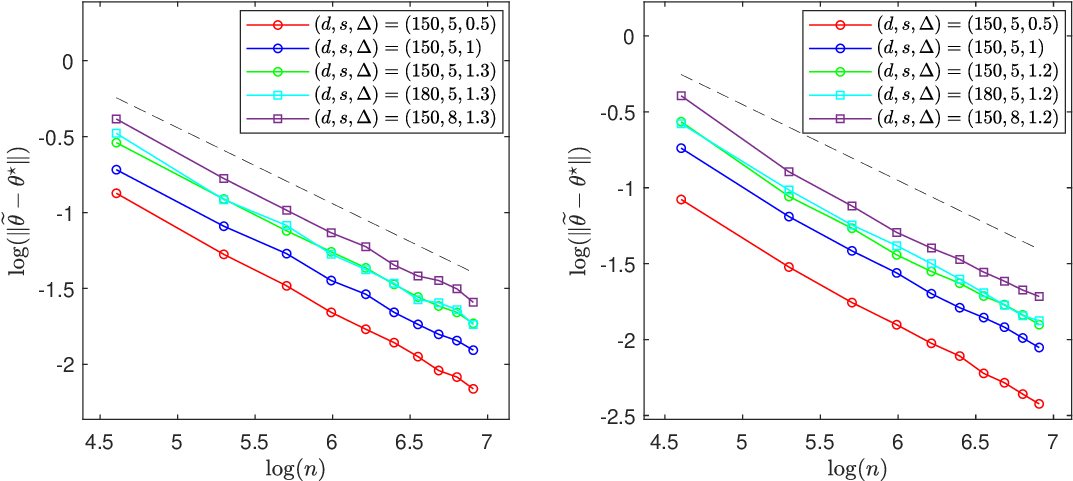}
    
 ~~~   (a) \hspace{4.4cm} (b) 
    \caption{(a):  QCS in Theorem \ref{thm6}; (b):  QCS in Theorem \ref{thm7}.}
    \label{fig3}
\end{figure}



\subsubsection{Quantized Matrix Completion}
Finally, we simulate QMC and demonstrate the   error bound $\mathscr{B}_4=O\big(\mathscr{L}\sqrt{\frac{rd\log d}{n}}\big)$ for $\|\bm{\widehat{\Upsilon}} \|_F/d$  in Theorems \ref{thm8}-\ref{thm9}. We generate the rank-$r$ $\bm{\Theta^\star}\in \mathbb{R}^{d\times d}$ as follows: we first generate $\bm{\Theta}_0\in \mathbb{R}^{d\times r}$ with i.i.d. standard Gaussian entries to obtain the rank-$r$ $\bm{\Theta}_1:=\bm{\Theta}_0\bm{\Theta}_0^\top$, then we rescale it to $\bm{\Theta^\star}:=k_1\bm{\Theta}_1$ such that $\|\bm{\Theta^\star}\|_F=d$. 
We use $\epsilon_k\sim \mathcal{N}(0,\frac{1}{4})$ to simulate the sub-exponential noise in Theorem \ref{thm8}, while $\epsilon_k\sim \frac{1}{\sqrt{6}}\mathsf{t}(3)$ for Theorem \ref{thm9}. The convex program (\ref{5.2}) is fed with  $\alpha=\|\bm{\Theta^\star}\|_\infty$ and optimized by the ADMM algorithm. We test different choices of $(d,r,\Delta)$ under $n=2000:1000:8000$, with the log-log error plots   displayed in Figure \ref{fig4}. Firstly, the experimental curves are well aligned with the dashed line that represents the optimal $n^{-1/2}$ rate. Then, comparing the results for $\Delta = 0,0.5,1$, we conclude that quantization only affects the multiplicative factor $\mathscr{L}$ in the estimation error. It should also be noted that, increasing either $d$ or $r$ leads to significantly larger error, which is consistent with the $\mathscr{B}_4$'s essential  dependence on $d$ and $r$.

\begin{figure}[ht!]
    \centering
    \includegraphics[scale = 0.72]{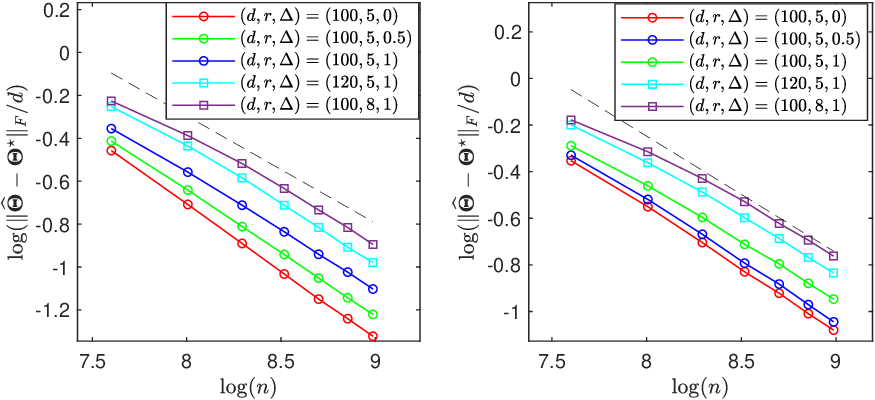}
    
 ~~~   (a) \hspace{4.4cm} (b) 
    \caption{(a): QMC in Theorem \ref{thm8}; (b): QMC in Theorem \ref{thm9}.}
    \label{fig4}
\end{figure}

\subsection{Importance of Appropriate Dithering}
To demonstrate the crucial role played by the suitable dither, we   provide the second set of simulations. In order to observe more significant phenomena and then conclude evidently, we may test huge sample size but a rather simple estimation problem under coarse quantization (i.e., large $\Delta$).

Specifically, for covariance matrix estimation we set $d=1$ and i.i.d. draw $X_1,...,X_n$ from $\mathcal{N}(0,1)$. Thus, the problem boils down to estimating $\mathbbm{E}|X_k|^2$, for which the estimators in Theorems \ref{thm1}-\ref{thm2} coincide. Since $X_k$ is sub-Gaussian, we do not perform data truncation before dithered quantization. Besides our estimator $\bm{\widehat{\Sigma}}=\frac{1}{n}\sum_{k=1}^n\dot{X}_k^2-\frac{\Delta^2}{4}$ where $\dot{X}_k=\mathcal{Q}_\Delta(X_k+\tau_k)$ and $\tau_k$ is triangular dither, we invite the following competitors: 
\begin{itemize}
[leftmargin=5ex,topsep=0.15ex]
 \setlength\itemsep{-0.1em}
    \item $\bm{\widehat{\Sigma}}_{no}=\frac{1}{n}\sum_{k=1}^n(\dot{X}'_k)^2$, where $\dot{X}'_k=\mathcal{Q}_\Delta(X_k)$ is the direct quantization without dithering;
    \item $\bm{\widehat{\Sigma}}_u-\frac{\Delta^2}{6}$ and $\bm{\widehat{\Sigma}}_u$, where $\bm{\widehat{\Sigma}}_u=\frac{1}{n}\sum_{k=1}^n(\dot{X}''_k)^2$, and $\dot{X}''_k=\mathcal{Q}_\Delta(X_k+\tau''_k)$ is quantized under uniform dither $\tau''_k\sim \mathscr{U}([-\frac{\Delta}{2},\frac{\Delta}{2}])$.
\end{itemize}   To illustrate the choice of $\bm{\widehat{\Sigma}}_u-\frac{\Delta^2}{6}$ and $\bm{\widehat{\Sigma}}_u$, we write $\dot{X}''_k=X_k+\tau''_k+w_k=X_k+\xi_k$ with quantization error $w_k\sim \mathscr{U}([-\frac{\Delta}{2},\frac{\Delta}{2}])$ (due to Theorem \ref{lem1}(a)) and quantization noise $\xi_k=\tau''_k+w_k$, then (\ref{3.1})  gives $\mathbbm{E}(\dot{X}''_k)^2=\mathbbm{E}|X_k|^2 +\mathbbm{E}|\xi_k|^2$, while $\mathbbm{E}|\xi_k|^2$ remains unknown. Thus, we consider $\bm{\widehat{\Sigma}}_u-\frac{\Delta^2}{6}$    because of an unjustified  guess $\mathbbm{E}|\xi_k|^2\approx \mathbbm{E}|\tau''_k|^2+\mathbbm{E}|w_k|^2=\frac{\Delta^2}{6}$, while   $\bm{\widehat{\Sigma}}_u$ simply gives up the correction of $\mathbbm{E}|\xi_k|^2$. We test $\Delta =3$ under $n=(2:2:20)\cdot10^3$. From the results shown in Figure \ref{fig5}(a), the proposed estimator based on quantized data under triangular dither embraces the lowest estimation errors and the optimal rate of $n^{-1/2}$, whereas other competitors are not consistent, i.e., they all reach some error floors under a large sample size.

For the two remaining signal recovery problems, we simply focus on the quantization of the response $y_k$. In particular, we simulate QCS in the setting of Theorem \ref{thm4}, with $(d,s,\Delta)=(50,3,2)$ under $n:=(2:2:20)\cdot 10^3$. Other experimental details are as previously stated. We compare our estimator $\bm{\widehat{\theta}}$ with its counterpart  $\bm{\widehat{\theta}}'$ defined by (\ref{4.3}) with the same $\bm{Q},\mathcal{S}$ but $\bm{b}'=\frac{1}{n}\sum_{k=1}^n\dot{y}'_k\bm{x}_k$, where $\dot{y}_k' = \mathcal{Q}_\Delta(\widetilde{y}_k)$ is a direct uniform quantization with no dither. Evidently, the simulation results   in Figure \ref{fig5}(b) confirm that the application of a uniform dither significantly lessens the recovery errors. Without dithering, although  our results under Gaussian covariate still exhibit $n^{-1/2}$ decreasing rate,  identifiability issue unavoidably arises under Bernoulli covariate. In that case, the simulation without dithering will evidently deviate from the $n^{-1/2}$ rate, see  \cite[Figure 1]{sun2022quantized} for instance.

In analogy, we simulate QMC   (Theorem \ref{thm8}) with data generated as previous experiments, and specifically we try $(d,r,\Delta)= (30,5,1.5)$ under $n=(5:5:25)\cdot 10^3$. While our estimator $\bm{\widehat{\Theta}}$ is defined in (\ref{5.2}) involving $\dot{y}_k$ from a dithered quantizer, we simulate the performance of its counterpart without dithering, i.e., $\bm{\widehat{\Theta}}'$ defined in (\ref{5.2}) with $\dot{y}_k$ substituted by $\dot{y}_k'=\mathcal{Q}_\Delta(y_k)$. From the experimental results displayed in Figure \ref{fig5}(c), one shall clearly see that $\bm{\widehat{\Theta}}$ performs much better in terms of the   decreasing rate of $n^{-1/2}$ and the estimation error; while the curve without dithering even does not decrease.

\begin{figure}[ht!]
    \centering
    \includegraphics[scale = 0.72]{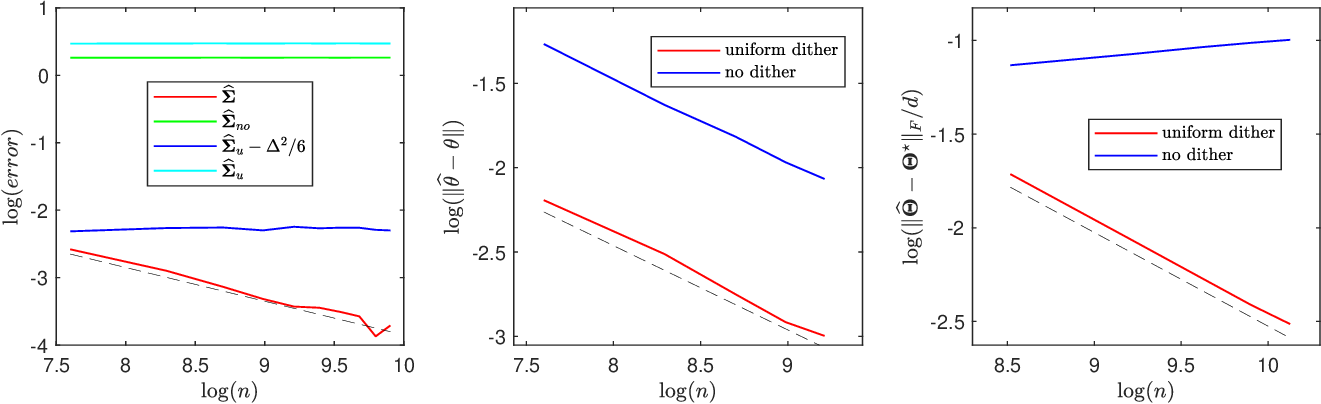}
    
 ~~~   (a) \hspace{4.6cm} (b) \hspace{4.6cm} (c) 
    \caption{(a): covariance matrix estimation; (b): QCS in Theorem \ref{thm4}; (c): QMC in Theorem \ref{thm8}.}
    \label{fig5}
\end{figure}
  
\section{Concluding Remarks}\label{sec7}
In digital signal processing and many distributed   machine learning systems, data quantization is an indispensable process. On the other hand,  many modern datasets exhibit   heavy-tailedness, and the past decade has witnessed an increasing interest in statistical estimation methods robust to heavy-tailed data. 
In this work we try to bridge these two developments by studying quantization of heavy-tailed data.
We propose to truncate the heavy-tailed data prior to  a   uniform quantizer with random dither well suited to the problem at hand. Applying our quantization scheme to covariance  matrix estimation, compressed sensing, and matrix completion, we have proposed (near) optimal estimators based on quantized data, and they are computationally feasible. These results suggest a unified conclusion that the dithered quantization does not affect the key scaling law in the error rate but only slightly worsens the multiplicative factor, which was complemented by numerical simulations. 
Further, from two respects, we presented additional developments for quantized compressed sensing. Firstly, we study a novel setting   that involves covariate quantization.   Because our quantized covariance matrix estimator is not positive semi-definite, the proposed recovery program is non-convex, but we proved that all local minimizers enjoy near minimax   rate. At a higher level, this development extends a line of works on non-convex M-estimator \cite{loh2011high,loh2013regularized,loh2017statistical} to accommodate heavy-tailed covariate, see the deterministic framework Proposition \ref{framework}. As application, we   derive results for (dithered) 1-bit compressed sensing as byproducts. Secondly,  we established a near minimax uniform recovery guarantee for QCS under heavy-tailed noise, which states that all sparse signals within an $\ell_2$-ball can be uniformly recovered up to near optimal $\ell_2$-norm error, using a single realization of  the measurement ensemble. We believe the developments presented in this work will prove useful in many other estimation problems, for instance, the triangular dither and the quantization scheme apply to multi-task learning, as shown by subsequent works \cite{chen2023quantized,li2023two}.



\bibliographystyle{plain}
\bibliography{libr}
\begin{appendix}
 
\section{Proofs in Section \ref{sec3}}
\subsection{Quantized Covariance Matrix Estimation}
We first provide Bernstein's inequality that is recurring in our proofs.  In application we will choose the more convenient one from (\ref{bernform1}) and (\ref{bernform2}).
\begin{lem}\label{bernstein}
    {\rm (Bernstein's inequality, \cite[Thm. 2.10, Coro. 2.11]{boucheron2013concentration})\textbf{.}} Let $X_1,...,X_n$ be independent random variables, and assume that there exist positive numbers $v$ and $c$ such that $\sum_{i=1}^n \mathbbm{E}[X_i^2]\leq v$ and $$\sum_{i=1}^n \mathbbm{E}|X_i|^q\leq \frac{q!}{2}vc^{q-2}\text{ for all integers }q\geq 3 ,$$
    then for any $t>0$ we have  \begin{gather}
        \label{bernform1}\mathbbm{P}\left(\Big|\sum_{i=1}^n(X_i-\mathbbm{E}X_i)\Big|\geq \sqrt{2vt}+ct\right)\leq 2\exp\big(-t\big)
        \\
        \label{bernform2}
        \mathbbm{P}\left(\Big|\sum_{i=1}^n(X_i-\mathbbm{E}X_i)\Big|\geq t\right)\leq  2\exp \left(-\frac{t^2}{2(v+ct)}\right) 
    \end{gather}
\end{lem}
We will also use the Matrix Bernstein's inequality. 
\begin{lem}
    \label{matrixbern}
    {\rm (Matrix Bernstein, \cite[Thm. 6.1.1]{tropp2015introduction})\textbf{.}} Let $\bm{S}_1,...,\bm{S}_n$ be independent zero-mean random variables with common dimension $d_1\times d_2$. We assume that $\|\bm{S}_k\|_{op}\leq L$ for $k\in [n]$ and introduce the matrix variance statistic $$\nu = \max\left\{\Big\|\sum_{k=1}^n \mathbbm{E}(\bm{S}_k\bm{S}_k^\top)\Big\|_{op},\Big\|\sum_{k=1}^n \mathbbm{E}(\bm{S}_k^\top \bm{S}_k)\Big\|_{op}\right\}.$$
    Then for any $t\geq 0$, we have $$\mathbbm{P}\left(\Big\|\sum_{k=1}^n\bm{S}_k\Big\|_{op}\geq t\right)\leq (d_1+d_2)\exp\left(\frac{-\frac{1}{2}t^2}{ \nu  + \frac{ Lt}{3}}\right).$$
\end{lem}
\subsubsection{Proof of Theorem \ref{thm1}}
\begin{proof} Recall that $\bm{\xi}_k=\bm{\dot{x}}_k-\bm{\widetilde{x}}_k$ is the quantization noise, and $\mathbbm{E}(\bm{\xi}_k\bm{\xi}_k^\top) = \frac{\Delta^2}{4}\bm{I}_d$,  which implies $\mathbbm{E}\bm{\widehat{\Sigma}}=\mathbbm{E}(\bm{\widetilde{x}}_k\bm{\widetilde{x}}_k^\top)$. Thus, by using triangle inequality we obtain $$\|\bm{\widehat{\Sigma}} - \bm{\Sigma^\star}\|_{\infty}\leq \|\bm{\widehat{\Sigma}} - \mathbbm{E}\bm{\widehat{\Sigma}}\|_{\infty} + \|\mathbbm{E}(\bm{\widetilde{x}}_k\bm{\widetilde{x}}_k^\top - \bm{x}_k\bm{x}_k^\top)\|_{\infty}:=I_1+I_2.$$ 

\noindent{\bf\sffamily Step 1. Bounding $I_1$.}

\vspace{1mm}

Note that $\|\bm{\widehat{\Sigma}} - \mathbbm{E}\bm{\widehat{\Sigma}}\|_\infty = \|\frac{1}{n}\sum_{k=1}^n\bm{\dot{x}}_k\bm{\dot{x}}_k^\top-\mathbbm{E}(\bm{\dot{x}}_k\bm{\dot{x}}_k^\top)\|_\infty$, so for any $(i,j)\in [d]\times [d]$ we aim to bound the $(i,j)$-th entry error $$ |\widehat{\sigma}_{ij}- \mathbbm{E}\widehat{\sigma}_{ij}| =\left| \sum_{k=1}^n\frac{1}{n} \dot{x}_{ki}\dot{x}_{kj}-\mathbbm{E}(\dot{x}_{ki}\dot{x}_{kj})\right|$$

Observe that the quantization noise is bounded as follows  $$\|\bm{\xi}_k\|_\infty\leq \|\mathcal{Q}_\Delta(\bm{\widetilde{x}}_k+\bm{\tau}_k)-(\bm{\widetilde{x}}_k+\bm{\tau}_k)\|_\infty + \|\bm{\tau}_k\|_\infty \leq \frac{3\Delta}{2},$$ which implies  $\mathbbm{E}|\xi_{ki}|^4\leq (\frac{3\Delta}{2})^4$ and $\|\bm{\dot{x}}_k\|_\infty \leq \|\bm{\widetilde{x}}_k\|_\infty + \|\bm{\xi}_k\|_\infty \leq \zeta + \frac{3\Delta}{2}$. By the moment constraint on $x_{ki}$ we have  $\mathbbm{E}|\widetilde{x}_{ki}|^4\leq \mathbbm{E}|x_{ki}|^4\leq M$. Thus, for any positive integer $p\geq 2$ we have the following bound \begin{equation}
    \begin{aligned}\label{prepabernstein}
      \sum_{k=1}^n\mathbbm{E} \Big|\frac{\dot{x}_{ki}\dot{x}_{kj}}{n}\Big|^q &\leq \frac{(\zeta+\frac{3}{2}\Delta)^{2(q-2)}}{n^q}\sum_{k=1}^n  \mathbbm{E}(\dot{x}_{ki}\dot{x}_{kj})^2\\&\leq\frac{(\zeta+\frac{3}{2}\Delta)^{2(q-2)}}{2n^q} \sum_{k=1}^n\big(\mathbbm{E}|\dot{x}_{ki}|^4+\mathbbm{E}|\dot{x}_{kj}|^4\big)\\
      &\leq \frac{4(\zeta+\frac{3}{2}\Delta)^{2(q-2)}}{n^q}\sum_{k=1}^n \big(\mathbbm{E}|\widetilde{x}_{ki}|^4+\mathbbm{E}|\widetilde{x}_{kj}|^4+ \mathbbm{E}|\xi_{ki}|^4+\mathbbm{E}|\xi_{kj}|^4\big)\\&\leq \frac{q!}{2} v_0c_0^{q-2},
    \end{aligned}
\end{equation}
for some $v_0 = O\big(\frac{{M}+\Delta^4}{n}\big)$, $c_0 =O\big(\frac{(\zeta+\Delta)^2}{n}\big)$. With these preparations, we can invoke Bernstein's inequality  (Lemma \ref{bernstein}) to obtain that, for any $t\geq 0$, with probability at least $1-2\exp(-t)$,
\begin{equation}\nonumber
   |\widehat{\sigma}_{ij} - \mathbbm{E}\widehat{\sigma}_{ij}| \leq C_1 \Big(\sqrt{\frac{(M+\Delta^4)t}{n}} + \frac{(\zeta^2 + \Delta^2)t}{n}\Big).
\end{equation}
Taking $t = \delta \log d$ and using the choice $\zeta\asymp \big(\frac{nM}{\delta \log d}\big)^{1/4}$, then   applying a union bound over $(i,j)\in [d]\times [d]$, under the scaling $n\gtrsim \delta \log d$,  we obtain that $I_1 \lesssim (\sqrt{M}+\Delta^2) \sqrt{\frac{\delta \log d}{n}}$ holds with probability at least $1-2d^{2-\delta}$.

\vspace{1mm}

\noindent{\bf \sffamily Step 2. Bounding $I_2$.}

We aim to  bound  $|\mathbbm{E}(\widetilde{x}_{ki}\widetilde{x}_{kj} - x_{ki}x_{kj})|$ for any $(i,j)\in [d]\times [d]$. First by the definition of truncation we have $$\big|\mathbbm{E}(\widetilde{x}_{ki}\widetilde{x}_{kj} - x_{ki}x_{kj})\big|\leq \mathbbm{E}\big[|x_{ki}x_{kj}|(\mathbbm{1}(|x_{ki}|\geq \zeta)+\mathbbm{1}(|x_{kj}|\geq \zeta))\big];$$  then applying  Cauchy-Schwarz to $\mathbbm{E}\big[|x_{ki}x_{kj}|\mathbbm{1}(|x_{ki}\geq \zeta|)\big]$, we obtain $$\mathbbm{E}\big[|x_{ki}x_{kj}|\mathbbm{1}(|x_{ki}|\geq \zeta) \big] \leq \big[\mathbbm{E}|x_{ki}x_{kj}|^2\big]^{1/2}\big[\mathbbm{P}(|x_{ki}|\geq \zeta)\big]^{1/2}\leq \sqrt{M}\sqrt{\frac{M}{\zeta^4}} = \frac{M}{\zeta^2},$$
where the second inequality is due to Markov's inequality.
Note that this bound remains valid for $\mathbbm{E}\big[|x_{ki}x_{kj}|\mathbbm{1}(|x_{kj}|\geq \zeta)\big]$. Since this holds for any $(i,j)\in [d]\times [d]$, combining with $\zeta\asymp \big(\frac{nM}{\delta \log d}\big)^{1/4}$, we obtain $\|\mathbbm{E}(\bm{\widetilde{x}}_k\bm{\widetilde{x}}_k^\top - \bm{x}_k\bm{x}_k^\top)\|_\infty\leq \frac{2M}{\zeta^2}\lesssim \sqrt{M}\sqrt{\frac{\delta \log d}{n}}$.

By  putting pieces together,  we have $\|\bm{\widehat{\Sigma}} - \bm{\Sigma^\star}\|_{\infty}\lesssim (\sqrt{M}+\Delta^2)\sqrt{\frac{\delta \log d}{n}}$ with probability at least $1-2d^{2-\delta}$, as claimed. \end{proof}

\subsubsection{Proof of Theorem \ref{thm2}}
\begin{proof} Note that the calculations in (\ref{3.1}) and (\ref{Delta24}) remain valid (but the truncated samples are denoted by $\bm{\check{x}}_k$ rather than $\bm{\widetilde{x}}_k$), so we have $\mathbbm{E}\bm{\widehat{\Sigma}}=\mathbbm{E}(\bm{\check{x}}_k\bm{\check{x}}_k^\top)$. Using triangle inequality we first decompose the error as $$\|\bm{\widehat{\Sigma}} - \bm{\Sigma^\star}\|_{op}\leq \|\bm{\widehat{\Sigma}} -\mathbbm{E}\bm{\widehat{\Sigma}}\|_{op}+\|\mathbbm{E}(\bm{\check{x}}_k\bm{\check{x}}_k^\top-\bm{x}_k\bm{x}_k^\top)\| _{op}:=I_1+I_2.$$

\noindent{\bf \sffamily Step 1. Bounding $I_1$.}

We first write that $$\bm{\widehat{\Sigma}}-\mathbbm{E}\bm{\widehat{\Sigma}} = \frac{1}{n}\sum_{k=1}^n\bm{S}_k \text{ where }\bm{S}_k = \bm{\dot{x}}_k\bm{\dot{x}}_k^\top-\mathbbm{E}(\bm{\dot{x}}_k\bm{\dot{x}}_k^\top).$$ Recall  that we define quantization error as $\bm{w}_k=\bm{\dot{x}}_k- \bm{\check{x}}_k-\bm{\tau}_k$ and quantization noise as  $\bm{\xi}_k=\bm{\dot{x}}_k- \bm{\check{x}}_k$, and observe that the quantization noise is bounded $\|\bm{\xi}_k\|_\infty= \|\bm{\dot{x}}_k - \bm{\check{x}}_k\|_{\infty}=\|\bm{\tau}_k+\bm{w}_k\|_\infty\leq \frac{3}{2}\Delta$. Thus, by $\|\bm{a}\|^2 _2 \leq \sqrt{d}\|\bm{a}\|_4^2$ that holds for any $\bm{a}\in \mathbb{R}^d$, we obtain \begin{equation}
   \begin{aligned}\nonumber
     \|\bm{\dot{x}}_k\bm{\dot{x}}_k^\top\| _{op}&= \|\bm{\dot{x}}_k\|_2^2=\|\bm{\check{x}}_k + \bm{\xi}_k\|^2 _2\leq 2\|\bm{\check{x}}_k\|^2_2+2\|\bm{\xi}_k\|^2_2 \\
    &\leq 2 \sqrt{d}\cdot\|\bm{\check{x}}_k\|_4^2+2d\cdot\Big(\frac{3\Delta}{2}\Big)^2\leq 2\sqrt{d}\zeta^2+\frac{9}{2}d\Delta^2,
   \end{aligned}
\end{equation}
which implies $\|\bm{S}_k\|_{op} \leq \|\bm{\dot{x}}_k\bm{\dot{x}}_k^\top\|_{op} + \mathbbm{E}\|\bm{\dot{x}}_k\bm{\dot{x}}_k^\top\|_{op}\leq 4\sqrt{d}\zeta^2 +9d\Delta^2$. Moreover, we estimate the matrix variance statistic. Since $\bm{S}_k$ is symmetric,   we simply deal with $\|\mathbbm{E}\bm{S}_k^2\|_{op}$ and  some algebra gives $\mathbbm{E}\bm{S}_k^2=\mathbbm{E}\big[\|\bm{\dot{x}}_k\|_2^2\bm{\dot{x}}_k\bm{\dot{x}}_k^\top\big]-\big(\mathbbm{E}\big[\bm{\dot{x}}_k\bm{\dot{x}}_k^\top\big]\big)^2$. First let us note that \begin{equation}\nonumber
    \begin{aligned}\Big\|\big(\mathbbm{E}\big[\bm{\dot{x}}_k\bm{\dot{x}}_k^\top\big]\big)^2\Big\|_{op}&=\Big\|\mathbbm{E}\big[\bm{\dot{x}}_k\bm{\dot{x}}_k^\top\big]\Big\|^2_{op}=\Big\|\mathbbm{E}\big[\bm{\check{x}}_k\bm{\check{x}}_k^\top\big] +\frac{\Delta^2}{4}\bm{I}_d\Big\|^2_{op}\\&\leq \Big(\Big\|\mathbbm{E}\big[\bm{\check{x}}_k\bm{\check{x}}_k^\top\big]\Big\|_{op} + \frac{\Delta^2}{4}\Big)^2\leq 2\Big\|\mathbbm{E}\big[\bm{\check{x}}_k\bm{\check{x}}_k^\top\big]\Big\|_{op}^2 + \frac{\Delta^4}{8}.\end{aligned}
\end{equation} Combining with the observation that  $$\Big\|\mathbbm{E}\big[\bm{\check{x}}_k\bm{\check{x}}_k^\top\big]\Big\|_{op}=\sup_{\bm{v}\in \mathbb{S}^{d-1}}\mathbbm{E}(\bm{v}^\top \bm{\check{x}}_k)^2\leq \sup_{\bm{v}\in \mathbb{S}^{d-1}}\sqrt{\mathbbm{E}(\bm{v}^\top \bm{x}_k)^4}\leq \sqrt{M},$$ we obtain $\big\|(\mathbbm{E}[\bm{\dot{x}}_k\bm{\dot{x}}_k^\top])^2\big\|_{op} = O(M+\Delta^4)$. Then we turn to the operator norm of $\mathbbm{E}[\|\bm{\dot{x}}_k\|_2^2\bm{\dot{x}}_k\bm{\dot{x}}_k^\top]$. We apply Cauchy-Schwarz to estimate \begin{equation}
    \begin{aligned}\label{3.7}
     \big\|\mathbbm{E}\big(\|\bm{\dot{x}}_k\|_2^2\bm{\dot{x}}_k\bm{\dot{x}}_k^\top\big)\big\|_{op} =\sup_{\bm{v}\in \mathbb{S}^{d-1}} \mathbbm{E}\big( \|\bm{\dot{x}}_k\|_2^2(\bm{v}^\top \bm{\dot{x}}_k)^2  \big){\leq} \sqrt{\mathbbm{E}\|\bm{\dot{x}}_k\|_2^4}\sup_{\bm{v}\in \mathbb{S}^{d-1}} \sqrt{\mathbbm{E}(\bm{v}^\top \bm{\dot{x}}_k)^4}.
    \end{aligned}
\end{equation}
By   $\|\bm{a}\|^2_2\leq \sqrt{d}\|\bm{a}\|_4^2$ that holds for any $\bm{a}\in \mathbb{R}^d$,   $\mathbbm{E}|\check{x}_{ki}|^4\leq \mathbbm{E}|x_{ki}|^4\leq M$, $\bm{\dot{x}}_k=\bm{\check{x}}_k+\bm{\xi}_k$ and $\|\bm{\xi}_k\|_\infty\leq \frac{3\Delta}{2}$, we obtain \begin{equation}
    \begin{aligned}
        \label{627add1}\mathbbm{E}\|\bm{\dot{x}}_k\|^4_2&\leq \mathbbm{E}(\|\bm{\check{x}}_k\|_2+\|\bm{\xi}_k\|_2)^4\lesssim \mathbbm{E}( \|\bm{\check{x}}_k\|_2^4+  \|\bm{\xi}_k\|^4_2)\\&\leq d\mathbbm{E}(\|\bm{\check{x}}_k\|^4_4 + \|\bm{\xi}_k\|^4_4)\lesssim d^2( M+\Delta^4).
    \end{aligned}
\end{equation} For any $\bm{v}\in \textcolor{black}{\mathbb{S}^{d-1}}$, we write 
$\bm{\dot{x}}_k= \bm{\check{x}}_k+\bm{\tau}_k+\bm{w}_k$ and then have the bound \begin{equation}
    \begin{aligned}\label{627add2}
      &\mathbbm{E}(\bm{v}^\top \bm{\dot{x}}_k)^4  \lesssim \mathbbm{E}(\bm{v}^\top \bm{\check{x}}_k)^4+ \mathbbm{E}(\bm{v}^\top \bm{\tau}_k)^4+ \mathbbm{E}(\bm{v}^\top \bm{w}_k)^4\stackrel{(i)}{\lesssim} M + \Delta^4, 
    \end{aligned}
\end{equation}
where $(i)$ is because $\mathbbm{E}(\bm{v}^\top \bm{\check{x}}_k)^4\leq\mathbbm{E}(\bm{v}^\top \bm{x}_k)^4\leq M$,    $\bm{\tau}_k\sim\mathscr{U}([-\frac{\Delta}{2},\frac{\Delta}{2}]^d)+\mathscr{U}([-\frac{\Delta}{2},\frac{\Delta}{2}]^d) $, and the quantization error $\bm{w}_k$ follows $ \mathscr{U}([-\frac{\Delta}{2},\frac{\Delta}{2}]^d)$; in more detail, $\|\bm{v}^\top \bm{\tau}_k\|_{\psi_2},\|\bm{v}^\top \bm{w}_k\|_{\psi_2}=O(\Delta)$ and  then the moment constraint of sub-Gaussian random variable implies     $\mathbbm{E}(\bm{v}^\top\bm{\tau}_k)^4=O(\Delta^4)$ and $\mathbbm{E}(\bm{v}^\top\bm{w}_k)=O(\Delta^4)$. From (\ref{3.7}), (\ref{627add1}) and (\ref{627add2}), we obtain $\big\|\mathbbm{E}\big(\|\bm{\dot{x}}_k\|_2^2\bm{\dot{x}}_k\bm{\dot{x}}_k^\top\big)\big\|_{op}=O\big(d(\Delta^4+M)\big)$. Further combining with $\mathbbm{E}\bm{S}_k^2=\mathbbm{E}\big[\|\bm{\dot{x}}_k\|_2^2\bm{\dot{x}}_k\bm{\dot{x}}_k^\top\big]-\big(\mathbbm{E}\big[\bm{\dot{x}}_k\bm{\dot{x}}_k^\top\big]\big)^2$ and $\big\|(\mathbbm{E}[\bm{\dot{x}}_k\bm{\dot{x}}_k^\top])^2\big\|_{op} = O(M+\Delta^4)$, we arrive at $\|\mathbbm{E}\bm{S}_k^2\|_{op}\lesssim d(\Delta^4+M)$ and hence $\big\|\sum_{k=1}^n \mathbbm{E}\bm{S}_k^2\big\|_{op}\lesssim nd(\Delta^4+M)$.  With these preparations,   Matrix Bernstein's inequality (Lemma \ref{matrixbern}) yields the following inequality that holds for any $t\geq 0$\begin{equation}\nonumber
     \mathbbm{P}\Big(\|\bm{\widehat{\Sigma}} - \mathbbm{E}\bm{\widehat{\Sigma}}\|_{op}\geq t\Big) \leq 2d\exp\left(-\frac{C_1nt^2}{(M+\Delta^4)d + (\sqrt{d}\zeta^2+d\Delta^2)t}\right).
\end{equation}    
  Setting $t=C_2 (\sqrt{M}+\Delta^2)\sqrt{\frac{\delta d\log d}{n}}$ with sufficiently large $C_2$, under the scaling of $n\gtrsim \delta d\log d$ and the threshold $\zeta \asymp (M^{1/4}+\Delta)\big(\frac{n}{\delta \log d}\big)^{1/4}$, we obtain that  $I_1=\|\bm{\widehat{\Sigma}}-\mathbbm{E}\bm{\widehat{\Sigma}}\|_{op} \leq C_2(\sqrt{M}+\Delta^2)\sqrt{\frac{\delta d\log d}{n}}$ holds with probability at least $1-2d^{1-\delta}$.

\noindent{\bf \sffamily Step 2. Bounding $I_2$.}

Having bounded the concentration term $I_1$, we now switch to  the bias term $$I_2=\sup_{\bm{v}\in \mathbb{S}^{d-1}}\big|\bm{v}^\top \mathbbm{E}(\bm{\check{x}}_k\bm{\check{x}}_k^\top - \bm{x}_k\bm{x}_k^\top) \bm{v}\big|.$$ For any $\bm{v}\in \mathbb{S}^{d-1}$,   because $\bm{\check{x}}_k$ is obtained from truncating $\bm{x}_4$ regarding $\ell_4$-norm, we have \begin{equation}
    \begin{aligned}\nonumber
     \big|\bm{v}^\top \mathbbm{E}(\bm{\check{x}}_k\bm{\check{x}}_k^\top - \bm{x}_k\bm{x}_k^\top) \bm{v}\big| &=\Big| \mathbbm{E}\Big[\big((\bm{v}^\top\bm{\check{x}}_k)^2-(\bm{v}^\top \bm{x}_k)^2\big)\mathbbm{1}(\|\bm{x}_k\|_4\geq \zeta)\Big]\Big| \\
    &\leq \mathbbm{E}\big[(\bm{v}^\top \bm{x}_k)^2\mathbbm{1}(\|\bm{x}_k\|_4\geq\zeta)\big]\\&\stackrel{(i)}{\leq} \sqrt{\mathbbm{E}(\bm{v}^\top \bm{x}_k)^4}\sqrt{\mathbbm{P}(\|\bm{x}_k\|^4_4\geq \zeta^4)}\\
    &\stackrel{(ii)}{\leq} \sqrt{M \frac{\mathbbm{E}\|\bm{x}_k\|_4^4}{\zeta^4}}  \stackrel{(iii)}{\lesssim} \sqrt{\frac{M\delta d\log d}{n}},
    \end{aligned}
\end{equation}
where $(i)$ and $(ii)$ are respectively by Cauchy-Schwarz and Markov's, and in $(iii)$ we use $\zeta \asymp (M^{1/4}+\Delta)\big(\frac{\delta d\log d}{n}\big)^{1/4}$. This leads to the bound $I_2\lesssim \sqrt{\frac{M\delta d \log d}{n}}.$ Combining the bounds of $I_1,I_2$ completes the proof.\end{proof}

\subsubsection{Proof of Theorem \ref{thm3}}
This small appendix is devoted to the proof of Theorem \ref{thm3}, for which we need a Lemma concerning the element-wise error rate of $\bm{\widehat{\Sigma}}_s$, i.e., $|\breve{\sigma}_{ij}-\sigma^\star_{ij}|$ where we write   $\bm{\widehat{\Sigma}}_s=[\breve{\sigma}_{ij}]$, $\bm{\Sigma^\star}= \mathbbm{E}(\bm{x}_k\bm{x}_k^\top) =[\sigma_{ij}^\star]$. Recalling that $\bm{\widehat{\Sigma}}_s= \mathcal{T}_\mu(\bm{\widehat{\Sigma}})$, the key message from Lemma \ref{lemmm} is that due to the thresholding operator $\mathcal{T}_{\mu}(\cdot)$,   $\bm{\widehat{\Sigma}}_s$ respects an element-wise bound tighter than $O\big(\sqrt{\frac{\delta \log d}{n}}\big)$ in Theorem \ref{thm1}, as can be seen from the additional branch $|\sigma^\star_{ij}|$ in (\ref{627add3}).

\begin{lem}\label{lemmm}
{\rm  (Element-wise Error Rate of $\bm{\widehat{\Sigma}}_s$)\textbf{.}}  For any $i,j\in [d]$, the thresholding estimator $\bm{\widehat{\Sigma}}_s = [\breve{\sigma}_{ij}]$ in Theorem \ref{thm3} satisfies   for some $C$ that \begin{equation}\label{627add3}
    \mathbbm{P}\left(|\breve{\sigma}_{ij}-\sigma^\star_{ij}|\leq C\min\Big\{|\sigma^\star_{ij}|,\mathscr{L}\sqrt{\frac{\delta \log d}{n}}\Big\}\right)\geq 1-2d^{-\delta}
\end{equation}
where $\mathscr{L}:=\sqrt{M}+\Delta^2$. 
\end{lem}
\begin{proof} Recall that $\bm{\widehat{\Sigma}}_s =[\breve{\sigma}_{ij}]= \mathcal{T}_\mu(\bm{\widehat{\Sigma}})=\mathcal{T}_\mu\big([\widehat{\sigma}_{ij}]\big)$ and hence $\breve{\sigma}_{ij}=\mathcal{T}_\mu(\widehat{\sigma}_{ij})$. Given $(i,j)$, the proof of  Theorem \ref{thm1} delivers $|\widehat{\sigma}_{ij}-\sigma_{ij}^\star|\leq C_1\mathscr{L} \sqrt{\frac{\delta \log d}{n}}$ with probability at least $1-2d^{-\delta}$. Assume that we are on this event in the following analyses. As stated in Theorem \ref{thm3}, we set $\mu = C_2\mathscr{L}\sqrt{\frac{\delta \log d}{n}}$ with $C_2>C_1$, $\mathscr{L}=\sqrt{M}+\Delta^2$. Since $\breve{\sigma}_{ij} = \mathcal{T}_\mu(\widehat{\sigma}_{ij})$, we   discuss whether $|\widehat{\sigma}|\geq \mu$ holds.

\noindent{\bf \sffamily Case 1. when $|\widehat{\sigma}_{ij}|<\mu$ holds.}

In this case we have $\breve{\sigma}_{ij}=0$, thus $|\breve{\sigma}_{ij}-\sigma^\star_{ij}| \leq |\sigma^\star_{ij}|$. Further, $|\sigma^\star_{ij}|\leq |\sigma^\star_{ij}-\widehat{\sigma}_{ij}|+|\widehat{\sigma}_{ij}|\leq C_1\mathscr{L}\sqrt{\frac{\delta \log d}{n}}+\mu\lesssim \mathscr{L}\sqrt{\frac{\delta \log d}{n}}$, so we also have $|\breve{\sigma}_{ij}-\sigma^\star_{ij}|\lesssim \mathscr{L}\sqrt{\frac{\delta \log d}{n}}$.

\vspace{1mm}

\noindent{\bf\sffamily Case 2. when $|\widehat{\sigma}_{ij}|\geq \mu$ holds.}

Then, we consider $|\widehat{\sigma}_{ij}|\geq \mu$ that implies $\breve{\sigma}_{ij}=\widehat{\sigma}_{ij}$. Then $|\breve{\sigma}_{ij}-\sigma^\star_{ij}|=|\widehat{\sigma}_{ij}-\sigma^\star_{ij}|\leq C_1\mathscr{L}\sqrt{\frac{\delta \log d}{n}}$. Moreover, $|\sigma^\star_{ij}|\geq |\widehat{\sigma}_{ij}|-|\widehat{\sigma}_{ij}-\sigma^\star_{ij}|\geq \mu -|\widehat{\sigma}_{ij}-\sigma^\star_{ij}|\geq (C_2-C_1)\mathscr{L}\sqrt{\frac{\delta \log d}{n}}$, so we also have $|\breve{\sigma}_{ij}-\sigma^\star_{ij}|=O(|\sigma^\star_{ij}|)$.

Therefore, in both cases we have proved that $|\breve{\sigma}_{ij}-\sigma^\star_{ij}|\lesssim\min\big\{|\sigma^\star_{ij}|,\mathscr{L}\sqrt{\frac{\delta \log d}{n}}\big\}$, which completes the proof. \end{proof}

We are now in a position to present the proof.

\vspace{2mm}
\noindent{\it Proof of Theorem \ref{thm3}.} We let $p=\frac{\delta}{4}\geq 1$ (just  assume $\delta\geq 4$) and use $B_0:= \mathscr{L}\sqrt{\frac{\delta \log d}{n}}$ as shorthand. For $(i,j)\in [d]\times [d]$ we define the event $\mathscr{A}_{ij}$ as $$\mathscr{A}_{ij}=\Big\{|\breve{\sigma}_{ij}-\sigma^\star_{ij}|\leq C_1\min\big\{|\sigma^\star_{ij}|,B_0\big\}\Big\}.$$
By Lemma \ref{lemmm} we can choose $C_1$ to be sufficiently large such that $C_1B_0>3\mu$ and $\mathbbm{P}(\mathscr{A}_{ij}^\complement)\leq 2d^{-\delta}$; here, by convention we let $\mathscr{A}_{ij}^\complement$ be the complement of $\mathscr{A}_{ij}$. Our proof strategy is to first bound  the $p$-th order moment $\mathbbm{E}\|\bm{\widehat{\Sigma}}_s-\bm{\Sigma^\star}\|_{op}^p$, and then invoke Markov's inequality to derive a high probability bound. We start with a simple estimate \begin{equation}
    \begin{aligned}\nonumber
      &\mathbbm{E}\|\bm{\widehat{\Sigma}}_s-\bm{\Sigma^\star}\|^p_{op} \stackrel{(i)}{\leq} \mathbbm{E}\Big(\sup_{j\in [d]}\sum_{i=1}^d|\breve{\sigma}_{ij}-\sigma^\star_{ij}|\mathbbm{1}(\mathscr{A}_{ij})+\sup_{j\in [d]}\sum_{i=1}^d|\breve{\sigma}_{ij}-\sigma^\star_{ij}|\mathbbm{1}(\mathscr{A}_{ij}^\complement)\Big)^p\\
      &\stackrel{(ii)}{\leq} 2^p \mathbbm{E}\sup_{j\in [d]}\Big(\sum_{i=1}^d|\breve{\sigma}_{ij}-\sigma^\star_{ij}|\mathbbm{1}(\mathscr{A}_{ij}) \Big)^p + 2^p \mathbbm{E}\sup_{j\in [d]}\Big(\sum_{i=1}^d|\breve{\sigma}_{ij}-\sigma^\star_{ij}|\mathbbm{1}(\mathscr{A}^\complement_{ij}) \Big)^p:=I_1+I_2
    \end{aligned}
\end{equation}
where $(i)$ and $(ii)$ are due to $\|\bm{A}\|_{op}\leq \sup_{j\in [d]}\sum_{i\in [d]}|a_{ij}|$ for symmetric $\bm{A}$ and $(a+b)^p\leq (2a)^p+(2b)^p$. In this proof, the ranges of indices in summation or supremum, if omitted, are $[d]$.

\noindent{\bf \sffamily Step 1. Bounding $I_1$.}

 By the definition of $\mathscr{A}_{ij}$, $|\breve{\sigma}_{ij}-\sigma^\star_{ij}|=0$  if $|\sigma^\star_{ij}|=0$.  Because the   columns of $\bm{\Sigma^\star}$ are $s$-sparse, we can straightforwardly bound $I_1$ as follows:
\begin{equation}
    \begin{aligned}
    \label{A.3}
      I_1 =2^p \mathbbm{E}\sup_j \Big(\sum_{i:|\sigma^\star_{ij}|>0}|\breve{\sigma}_{ij}-\sigma^\star_{ij}|\mathbbm{1}(\mathscr{A}_{ij})\Big)^p \leq \big(2C_1sB_0\big)^p.
    \end{aligned}
\end{equation}

\noindent{\bf \sffamily Step 2. Bounding $I_2$.}

We first write $I_2=2^p\mathbbm{E}\sup_j W_j$ with $W_j:=\big(\sum_i|\breve{\sigma}_{ij}-\sigma^\star_{ij}|\mathbbm{1}(\mathscr{A}_{ij}^\complement)\big)^p$,  then start from \begin{equation}
    \begin{aligned}
      &W_j\stackrel{(i)}{\leq} \Big(\sum_{i=1}^d|\sigma^\star_{ij}|\mathbbm{1}(\mathscr{A}_{ij}^\complement)\mathbbm{1}(|\widehat{\sigma}_{ij}|<\mu)+ \sum_{i=1}^d |\widehat{\sigma}_{ij}-\mathbbm{E}\widehat{\sigma}_{ij}|\mathbbm{1}(\mathscr{A}_{ij}^\complement)+ \sum_{i=1}^d|\widetilde{\sigma}_{ij}-\sigma^\star_{ij}|\mathbbm{1}(\mathscr{A}_{ij}^\complement)\Big)^p\\
      &\leq (3d)^{p-1} \Big(\sum_{i=1}^d |\sigma^\star_{ij}|^p\mathbbm{1}(\mathscr{A}_{ij}^\complement)\mathbbm{1}(|\widehat{\sigma}_{ij}|<\mu)+ \sum_{i=1}^d|\widehat{\sigma}_{ij}-\mathbbm{E}\widehat{\sigma}_{ij}|^p\mathbbm{1}(\mathscr{A}_{ij}^\complement)+\sum_{i=1}^d|\widetilde{\sigma}_{ij}-\sigma^\star_{ij}|^p \mathbbm{1}(\mathscr{A}_{ij}^\complement)\Big),\nonumber
    \end{aligned}
\end{equation}
 note that in $(i)$ we define $$\mathbbm{E}\widehat{\sigma}_{ij}=\mathbbm{E}(\widetilde{x}_{ki}\widetilde{x}_{kj}):=\widetilde{\sigma}_{ij}.$$ By replacing $\sup_j$ with $\sum_j$, this  further gives 
\begin{equation}
   \begin{aligned}\label{A.4}
      &I_2 \leq 6^pd^{p-1}\Big({\sum_{i,j}|\sigma^\star_{ij}|^p\mathbbm{E}\big[\mathbbm{1}(\mathscr{A}_{ij}^\complement)\mathbbm{1}(|\widehat{\sigma}_{ij}|<\mu)\big]} + {\sum_{i,j}\mathbbm{E}\big[|\widehat{\sigma}_{ij}-\mathbbm{E}\widehat{\sigma}_{ij}|^p\mathbbm{1}(\mathscr{A}_{ij}^\complement)\big]} \\&+ 
     \sum_{i,j}|\widetilde{\sigma}_{ij}-\sigma^\star_{ij}|^p \mathbbm{P}(\mathscr{A}_{ij}^\complement)\Big) := 6^pd^{p-1}\big(I_{21}+I_{22}+I_{23}\big).
   \end{aligned}
\end{equation}

\noindent{\bf \sffamily Step 2.1. Bounding $I_{21}$.}

Note that $\mathscr{A}_{ij}^\complement$ means $|\breve{\sigma}_{ij}-\sigma^\star_{ij}|>C_1\min\{|\sigma^\star_{ij}|,B_0\}$, and $|\widehat{\sigma}_{ij}|<\mu$ implies $\breve{\sigma}_{ij}=0$, their combination thus allows us to proceed as the following $(i)$ and $(iii)$:   $$|\sigma^\star_{ij}|\stackrel{(i)}{>}C_1B_0\stackrel{(ii)}{>}3\mu \stackrel{(iii)}{>}3|\widehat{\sigma}_{ij}|\geq 3|\sigma^\star_{ij}|-3|\widehat{\sigma}_{ij}-\sigma^\star_{ij}|,$$
where $(ii)$ is due to our choice of $C_1$. Thus, $\mathscr{A}_{ij}^\complement\cap\{|\widehat{\sigma}_{ij}|<\mu\}$ implies $|\widehat{\sigma}_{ij}-\sigma^\star_{ij}|>\frac{2}{3}|\sigma^\star_{ij}|$ and $|\sigma^\star_{ij}|>3\mu$. Note that Step 2 in the proof of Theorem \ref{thm1} gives $|\widetilde{\sigma}_{ij}-\sigma^\star_{ij}|=O\big(B_0\big)$, and hence we can assume $\mu>|\widetilde{\sigma}_{ij}-\sigma^\star_{ij}|$ and so $|\sigma^\star_{ij}|>3|\widetilde{\sigma}_{ij}-\sigma^\star_{ij}|$. Using these relations and triangle inequality, we obtain $$\frac{2}{3}|\sigma^\star_{ij}|<|\widehat{\sigma}_{ij}-\sigma^\star_{ij}|\leq |\widehat{\sigma}_{ij}-\mathbbm{E}\widehat{\sigma}_{ij}|+|\widetilde{\sigma}_{ij}-\sigma^\star_{ij}|<|\widehat{\sigma}_{ij}-\mathbbm{E}\widehat{\sigma}_{ij}|+\frac{1}{3}|\sigma^\star_{ij}|,$$
which implies $|\widehat{\sigma}_{ij}-\mathbbm{E}\widehat{\sigma}_{ij}|>\frac{1}{3}|\sigma^\star_{ij}|$. Now we conclude that, $\mathscr{A}_{ij}^\complement\cap\{|\widehat{\sigma}_{ij}|<\mu\}$ implies $|\widehat{\sigma}_{ij}-\mathbbm{E}\widehat{\sigma}_{ij}|>\frac{1}{3}|\sigma^\star_{ij}|$ and $|\sigma^\star_{ij}|>3\mu$, which allows us to bound $I_{21}$ as  \begin{equation}
    \begin{aligned}\label{A.5}
      &I_{21} = \sum_{i,j}|\sigma^\star_{ij}|^p \mathbbm{1}(|\sigma_{ij}^\star|>3\mu)\mathbbm{P}\Big(|\widehat{\sigma}_{ij}-\mathbbm{E}\widehat{\sigma}_{ij}|>\frac{1}{3}|\sigma^\star_{ij}|\Big).
    \end{aligned}
\end{equation}
Analogously to the proof of Theorem \ref{thm1}, we can apply Bernstein's inequality to $\mathbbm{P}\big(|\widehat{\sigma}_{ij}-\mathbbm{E}\widehat{\sigma}_{ij}|>\frac{1}{3}|\sigma^\star_{ij}|\big)$. More specifically, by preparations as in (\ref{prepabernstein}), we can use (\ref{bernform2}) in Lemma \ref{bernstein} with $v=O\big(\frac{M+\Delta^4}{n}\big)$, $c=O\big(\frac{\zeta^2+\Delta^2}{n})=O(\frac{\Delta^2}{n}+\sqrt{\frac{M}{n\delta \log d}}\big)$ (recall that $\zeta \asymp \big(\frac{nM}{\Delta\log d}\big)^{1/4}$).  For some absolute constants $C_2,C_3$, it gives \begin{equation}
    \begin{aligned}\label{627add4}
      &\mathbbm{P}\Big(|\widehat{\sigma}_{ij}-\mathbbm{E}\widehat{\sigma}_{ij}|>\frac{1}{3}|\sigma^\star_{ij}|\Big) \leq 2\exp\left(-\frac{|\sigma^\star_{ij}|^2}{C_2\big\{\frac{M+\Delta^4}{n}+\frac{\Delta^2|\sigma^\star_{ij}|}{n}+\sqrt{\frac{M}{n\delta \log d}}|\sigma^\star_{ij}|\big\}}\right)\\
      & {\leq} 2\exp\left(-\frac{3n|\sigma^\star_{ij}|}{C_2}\min\Big\{\frac{|\sigma^\star_{ij}|}{M+\Delta^4},\frac{1}{\Delta^2},\sqrt{\frac{\delta \log d}{nM}}\Big\}\right)\stackrel{(i)}\leq 2\exp\left(-  \frac{C_3|\sigma^\star_{ij}|\sqrt{n\delta \log d}}{\sqrt{M}+\Delta^2}\right),
    \end{aligned}
\end{equation}
and in $(i)$ we use $\min\big\{\frac{|\sigma^\star_{ij}|}{M+\Delta^4},\Delta^{-2},\sqrt{\frac{\delta \log d}{nM}}\big\}\gtrsim \frac{1}{\sqrt{M}+\Delta^2}\sqrt{\frac{\delta \log d}{n}}$ that holds because $|\sigma^\star_{ij}|>3\mu$ and $n\gtrsim \delta \log d$. We substitute  (\ref{627add4}) into (\ref{A.5}) and perform some estimates   \begin{equation}
    \begin{aligned}\nonumber
      &I_{21}\leq 2\sum_{i,j}|\sigma^\star_{ij}|^p\mathbbm{1}\big(|\sigma^\star_{ij}|>3\mu\big)\exp\left(-\frac{C_3|\sigma^\star_{ij}|\sqrt{n\delta \log d}}{\sqrt{M}+\Delta^2}\right)\\
      &=2\sum_{i,j}\left(\frac{\sqrt{M}+\Delta^2}{C_3\sqrt{n\delta \log d}}\right)^p \cdot \left(\frac{C_3|\sigma^\star_{ij}|\sqrt{n\delta \log d
     }}{\sqrt{M}+\Delta^2}\right)^p\exp\left(-\frac{0.5C_3|\sigma^\star_{ij}|\sqrt{n\delta \log d}}{\sqrt{M}+\Delta^2}\right)\\&~~~~~~~~~~~~~~~~~~~~~~~~~~~~~~~~~~~~~\cdot  \exp\left(-\frac{0.5C_3|\sigma^\star_{ij}|\sqrt{n\delta \log d}}{\sqrt{M}+\Delta^2}\right)\mathbbm{1}\big(|\sigma^\star_{ij}|>3\mu\big)\\
     & \stackrel{(i)}{\leq}2\sum_{i,j}\left(\frac{\sqrt{M}+\Delta^2}{C_3\sqrt{n\delta \log d}}\right)^p \cdot \left(\sup_{t\geq 0}~t^p\exp\Big(-\frac{t}{2}\Big)\right)\cdot \exp\left(-\frac{3C_3}{2}\frac{\sqrt{n\delta \log d}\cdot \mu}{\sqrt{M}+\Delta^2}\right)\\& \stackrel{(ii)}{\leq}2d^{2-10\delta} \left(\frac{\sqrt{M}+\Delta^2}{C_3}\sqrt{\frac{\delta }{n\log d}}\right)^p\leq 2d^{2-10\delta}(C_3^{-1}B_0)^p,
    \end{aligned}
\end{equation}
where in $(i)$ we substitute $|\sigma^\star_{ij}|>3\mu$ from the indicator function into the exponent, $(ii)$ is because   $\sup_{t\geq 0}t^p\exp\big(-\frac{t}{2}\big)\leq p^p$, $p=\frac{\delta}{4}$, and we consider $\mu= C_4(\sqrt{M}+\Delta^2)\sqrt{\frac{\delta \log d}{n}}$ with $C_4$ large enough.

\noindent{\bf \sffamily Step 2.2. Bounding $I_{22}$.}

Then, we deal with $I_{22}$ by Cauchy-Schwarz \begin{equation}
    \nonumber I_{22}\leq\sum_{i,j}\sqrt{\mathbbm{E}|\widehat{\sigma}_{ij}-\mathbbm{E}\widehat{\sigma}_{ij}|^{2p}}\sqrt{\mathbbm{P}(\mathscr{A}_{ij}^\complement)}.
\end{equation}
As in (\ref{627add4}), we can use (\ref{bernform2}) in Lemma \ref{bernstein} with $v= O\big(\frac{M+\Delta^4}{n}\big)$ and  $c=O\big(\frac{\Delta^2}{n}+\sqrt{\frac{M}{n\delta \log d}}\big)$, yielding that for any  $t\geq 0$, 
$\mathbbm{P}\big(|\widehat{\sigma}_{ij}-\mathbbm{E}\widehat{\sigma}_{ij}|\geq t\big) \leq2 \exp\big(-\frac{t^2}{2(v+ct)}\big) \leq 2\exp\big(-\frac{t^2}{4v}\big)+2\exp\big(-\frac{t}{4c}\big).
$
Based on this probability tail bound, we can bound the moment  via integral  as follows
\begin{equation}
    \begin{aligned}\nonumber
      \mathbbm{E}|\widehat{\sigma}_{ij}-\mathbbm{E}\widehat{\sigma}_{ij}|^{2p}&=2p \int_{0}^\infty t^{2p-1} \mathbbm{P}(|\widehat{\sigma}_{ij}-\mathbbm{E}\widehat{\sigma}_{ij}|>t)~\mathrm{d}t\\
      &\leq 4p\int_{0}^\infty t^{2p-1}\big(\exp(-\frac{t^2}{4v})+\exp(-\frac{t}{4c})\big)~\mathrm{d}t\\&=   2\big[(4v)^p\Gamma(p+1)+ (4c)^{2p}\Gamma(2p+1)\big]\stackrel{(i)}{\leq} 2\big[(4vp)^p+(8cp)^{2p}\big],
    \end{aligned}
\end{equation}
where we use $\Gamma(p+1)\leq p^p$, $\Gamma(2p+1)\leq (2p)^{2p}$ in $(i)$ under suitably large $p$. Thus, it follows that \begin{equation}\nonumber
    I_{22}\leq \sum_{i,j}2d^{-\frac{\delta}{2}}\sqrt{(4vp)^p+(8cp)^{2p}} \leq 2d^{2-\frac{\delta}{2}}\big[(2\sqrt{pv})^p+(8cp)^p\big]\stackrel{(i)}{\leq}2d^{2-\frac{\delta}{2}}(C_4B_0)^p,
\end{equation}
where $(i)$ is due to $2\sqrt{pv}\leq(\sqrt{M}+\Delta^2)\sqrt{\frac{\delta}{n}}$ and $8cp=\frac{2\Delta^2\delta}{n}+2\sqrt{\frac{\delta M}{n\log d}}$ (recall that $p=\frac{\delta}{4}$).

\noindent{\bf \sffamily Step 2.3. Bounding $I_{23}$.}

From Step 2 in the proof of Theorem \ref{thm1} we have $|\widetilde{\sigma}_{ij}-\sigma^\star_{ij}|\leq C_5B_0$. This directly leads to $$I_{23}\leq d^2\cdot 2d^{-\delta}\cdot (C_5B_0)^p = 2d^{2-\delta}(C_5B_0)^p.$$


We are in a position to combine everything and conclude the proof.
Putting all pieces into (\ref{A.4}), it follows that $I_2\leq d^{1-\frac{\delta}{4}}(C_6B_0)^p$. Assuming $\delta \geq 4$, such upper bound is dominated by (\ref{A.3}) for $I_1$, we can hence conclude that $\mathbbm{E}\|\bm{\widehat{\Sigma}}_s-\bm{\Sigma^\star}\|_{op}^p\leq (C_6sB_0)^p$. Therefore, by Markov's inequality, $$\mathbbm{P}(\|\bm{\widehat{\Sigma}}_s-\bm{\Sigma^\star}\|_{op}\geq C_6esB_0)\leq \frac{\mathbbm{E}\|\bm{\widehat{\Sigma}}_s-\bm{\Sigma^\star}\|_{op}^p}{(C_6esB_0)^p}\leq \exp(-p) = \exp\Big(-\frac{\delta}{4}\Big),$$
which completes the proof. \hfill $\square$

\subsection{Quantized Compressed Sensing}
Note that our estimation procedure in QCS, QMC falls in the   framework of regularized M-estimator, see \cite{negahban2011estimation,fan2021shrinkage,chen2022high} for instance. Particularly, we introduce the following deterministic result for analysing the estimator (\ref{4.3}). 
\begin{lem}
    \label{csframework}
    {\rm(Adapted from}{\rm  \cite[Coro. 2]{chen2022high}){\textbf{.}}} Consider (\ref{csmodel}) and the estimator $\bm{\widehat{\theta}}$ defined in (\ref{4.3}), let $\bm{\widehat{\Upsilon}}:=\bm{\widehat{\theta}} - \bm{\theta^\star}$ be the estimation error. If $\bm{Q}$ is positive semi-definite, and $\lambda\geq 2\|\bm{Q\theta^\star}-\bm{b}\|_\infty$, then it holds that $\|\bm{\widehat{\Upsilon}}\|_1\leq 10\sqrt{s}\|\bm{\widehat{\Upsilon}}\|_2$. Moreover, if for some $\kappa>0$ we have the restricted strong convexity (RSC) $\bm{\widehat{\Upsilon}}^\top\bm{Q}\bm{\widehat{\Upsilon}}\geq \kappa \|\bm{\widehat{\Upsilon}}\|_2^2$, then we have  the error bounds  $\|\bm{\widehat{\Upsilon}}\|_2\leq 30\sqrt{s}\big(\frac{\lambda}{\kappa}\big)$ and $\|\bm{\widehat{\Upsilon}}\|_1\leq 300s \big(\frac{\lambda}{\kappa}\big)$.\footnote{We do not optimize the constants in Lemmas \ref{csframework}, \ref{mcframework} for easy reference.}
\end{lem}
To establish the RSC condition, a convenient way is to use the matrix deviation inequality. The following Lemma is adapted from \cite{liaw2017simple}, by combining Theorem 3 and Remark 1 therein.\footnote{The dependence on $K$   can be further refined \cite{jeong2022sub}, while this is not pursued in the present paper.} 

\begin{lem}
    \label{deviation}
    {\rm(Adapted from}{\rm  \cite[Thm. 3]{liaw2017simple}){\bf \sffamily.}} Assume $\bm{A}\in \mathbb{R}^{n\times d}$ has independent zero-mean sub-Gaussian rows $\bm{\alpha}_k^\top$s satisfying  $\|\bm{\alpha}_k\|_{\psi_2}\leq K$, and the eigenvalues of $\bm{\Sigma}:=\mathbbm{E}(\bm{\alpha}_k\bm{\alpha}_k^\top)$ are between $[\kappa_0,\kappa_1]$ for some $\kappa_1\geq \kappa_0>0$. For   $\mathcal{T}\subset \mathbb{R}^d$ we let $\mathrm{rad}(\mathcal{T})=\sup_{\bm{x}\in \mathcal{T}}\|\bm{x}\|_2$ be its radius. Then with probability at least $1-\exp(-u^2)$, it holds that \begin{equation}\nonumber
        \sup_{\bm{x}\in \mathcal{T}}\Big|\|\bm{Ax}\|_2- \sqrt{n}\|\sqrt{\bm{\Sigma}}\bm{x}\|_2\Big| \leq \frac{C\sqrt{\kappa_1}K^2}{\kappa_0}\Big(\omega(\mathcal{T})+u\cdot \mathrm{rad}(\mathcal{T})\Big),
    \end{equation} 
    where $\omega(\mathcal{T}) = \mathbbm{E}\sup_{\bm{v}\in \mathcal{T}}[\bm{g}^\top\bm{v}]$ with $\bm{g}\sim \mathcal{N}(0,\bm{I}_d)$ is the Gaussian width of $\mathcal{T}$.
\end{lem}

Based on Lemma \ref{csframework}, the proofs of Theorems \ref{thm4}-\ref{thm5} are divided into two steps, i.e., showing $\lambda\geq 2\|\bm{Q\theta^\star}-\bm{b}\|_\infty$ and verifying the RSC. While we still have   full $\bm{x}_k$ in Theorems \ref{thm4}-\ref{thm5},  we will study the more challenging settings where the covariates $\bm{x}_k$s are also quantized via $\mathcal{Q}_{\bar{\Delta}}(\cdot)$ in Theorems \ref{thm6}-\ref{thm7}, in which we can take $\bar{\Delta}=0$  to return the settings of Theorems \ref{thm4}-\ref{thm5}. \textcolor{black}{Using such perspective, for most technical ingredients (e.g., the verification of $\lambda\geq 2\|\bm{Q\theta^\star}-\bm{b}\|_\infty$)  in the   proofs of Theorems \ref{thm4}-\ref{thm5} we can simply refer to the counterparts established in the proofs of Theorems \ref{thm6}-\ref{thm7}. This avoids repetition and will be   explained in the proofs more clearly.}  
\subsubsection{Proof of Theorem \ref{thm4}}
\noindent{\it Proof.}  We divide the proofs into two steps.

\noindent{\bf \sffamily  Step 1. Proving $\lambda \geq 2\|\bm{Q\theta^\star}-\bm{b}\|_\infty$} 

   Recall that we choose $\bm{Q}=\frac{1}{n}\sum_{k=1}^n\bm{x}_k\bm{x}_k^\top$ and $\bm{b}=\frac{1}{n}\sum_{k=1}^n\dot{y}_k\bm{x}_k$. In the setting of Theorem \ref{thm6}, the process of obtaining $\dot{y}_k$ remains the same, while the covariates $\bm{x}_k$s are further quantized to $\bm{\dot{x}}_k=\mathcal{Q}_{\bar{\Delta}}(\bm{x}_k+\bm{\tau_k})$ for some $\bar{\Delta}>0$ under triangular dither $\bm{\tau}_k\sim \mathscr{U}([-\frac{\bar{\Delta}}{2},\frac{\bar{\Delta}}{2}]^d)+\mathscr{U}([-\frac{\bar{\Delta}}{2},\frac{\bar{\Delta}}{2}]^d)$, and we choose $\bm{Q}=\frac{1}{n}\sum_{k=1}^n\bm{\dot{x}}_k\bm{\dot{x}}_k^\top-\frac{\bar{\Delta}^2}{4}\bm{I}_d$ and $\bm{b}=\frac{1}{n}\sum_{k=1}^n\dot{y}_k\bm{\dot{x}}_k$ there. As a result, by considering $\bar{\Delta}=0$,
   it can be implied by   Step 1  in the proof of Theorem \ref{thm6}   that  under the choice $\lambda = C_1\frac{\sigma^2}{\sqrt{\kappa_0}}(\Delta+M^{1/(2l)})\sqrt{\frac{\delta \log d}{n}}$ with sufficiently large $C_1$, $\lambda \geq 2\|\frac{1}{n}\sum_{k=1}^n\bm{x}_k\bm{x}_k^\top \bm{\theta^\star} - \frac{1}{n}\sum_{k=1}^n\dot{y}_k\bm{x}_k\|_{\infty}$ holds with probability at least $1-8d^{1-\delta}$.
 Then, by using Lemma \ref{csframework} we obtain $\|\bm{\widehat{\Upsilon}}\|_1\leq 10\sqrt{s}\|\bm{\widehat{\Upsilon}}\|_2$.

\noindent{\bf \sffamily  Step 2.   Verifying the RSC   $\bm{\widehat{\Upsilon}}^\top\bm{Q}\bm{\widehat{\Upsilon}}\geq\kappa\|\bm{\widehat{\Upsilon}}\|_2^2$}

We refer to Step 2  in the proof of Theorem \ref{thm6}. In particular, with the choices $\bar{\Delta}=0$ and $\bm{v}=\bm{\widehat{\Upsilon}}$ in (\ref{nB.7}), combined with $\|\bm{\widehat{\Upsilon}}\|_1 \leq 10\sqrt{s}\|\bm{\widehat{\Upsilon}}\|_2$, we obtain
\begin{equation}
    \begin{aligned}\nonumber
        &\frac{1}{\sqrt{n}}\|\bm{X\widehat{\Delta}}\| _2\geq \sqrt{\kappa_0}\|\bm{\widehat{\Upsilon}}\|_2 - \frac{C_2\sqrt{\kappa_1} \sigma^2}{\kappa_0}\sqrt{\frac{\delta s\log d}{n}}\|\bm{\widehat{\Upsilon}}\|_2 \geq \frac{1}{2}\sqrt{\kappa_0}\|\bm{\widehat{\Upsilon}}\|_2,
    \end{aligned}
\end{equation}
where the last inequality is due to the assumed scaling $n\gtrsim \delta s\log d$. With these preparations, a direct application of   Lemma \ref{csframework} completes the proof. \hfill $\square$
\subsubsection{Proof of Theorem \ref{thm5}}

\noindent
{\it Proof.} The proof is similarly based on   Lemma \ref{csframework}. 

\noindent{\bf \sffamily  Step 1. Proving $\lambda \geq 2\|\bm{Q\theta^\star}-\bm{b}\|_{\infty}$}

Recall that we choose $\bm{Q}=\frac{1}{n}\sum_{k=1}^n\bm{\widetilde{x}}_k\bm{\widetilde{x}}_k^\top$ and $\bm{b}=\frac{1}{n}\sum_{k=1}^n\dot{y}_k\bm{\widetilde{x}}_k$.
In the setting of Theorem \ref{thm7}, the process of obtaining $\dot{y}_k$ remains the same, while the truncated covariates $\bm{\widetilde{x}}_k$s are further quantized to $\bm{\dot{x}}_k=\mathcal{Q}_{\bar{\Delta}}(\bm{\widetilde{x}}_k+\bm{\tau}_k)$ for some $\bar{\Delta}\geq 0$ under triangular dither $\bm{\tau}_k\sim \mathscr{U}([-\frac{\bar{\Delta}}{2},\frac{\bar{\Delta}}{2}]^d)+\mathscr{U}([-\frac{\bar{\Delta}}{2},\frac{\bar{\Delta}}{2}]^d)$, and we choose $\bm{Q}=\frac{1}{n}\sum_{k=1}^n \bm{\dot{x}}_k\bm{\dot{x}}_k^\top-\frac{\bar{\Delta}^2}{4}\bm{I}_d$ and $\bm{b}=\frac{1}{n}\sum_{k=1}^n \dot{y}_k\bm{\dot{x}}_k$ there. As a result, by considering $\bar{\Delta}=0$,
it can be implied by  step 1  in the proof of Theorem \ref{thm7} that,   our choice $\lambda = C_1(R\sqrt{M}+\Delta^2)\sqrt{\frac{\delta \log d}{n}}$ with sufficiently large $C_1$ ensures $\lambda\geq 2\|\bm{Q\theta^\star}-\bm{b}\|_\infty$ with the promised probability.    
  By Lemma \ref{csframework} we obtain $\|\bm{\widehat{\Upsilon}}\|_1\leq 10\sqrt{s}\|\bm{\widehat{\Upsilon}}\|_2$. 

\vspace{1mm}

\noindent{\bf \sffamily  Step 2.   Verifying the RSC   $\bm{\widehat{\Upsilon}}^\top\bm{Q}\bm{\widehat{\Upsilon}}\geq\kappa\|\bm{\widehat{\Upsilon}}\|_2^2$}

Unlike the case of sub-Gaussian covariate that is based on matrix deviation inequality (Lemma \ref{deviation}), here we establish a lower bound for $\bm{\widehat{\Upsilon}}^\top\bm{Q}\bm{\widehat{\Upsilon}}$ using the bound on $\|\bm{Q}-\bm{\Sigma^\star}\|_\infty$ (Theorem \ref{thm1}). Specifically,   setting $\Delta=0$ in Theorem \ref{thm1} yields that,  $\|\bm{Q}-\bm{\Sigma^\star}\|_\infty \lesssim  \sqrt{\frac{\delta M \log d}{n}}$ holds with probability at least $1-2d^{2-\delta}$, which allows us to proceed as follows: \begin{equation}
    \begin{aligned}\label{4.8}
\bm{\widehat{\Upsilon}}^\top \bm{Q}\bm{\widehat{\Upsilon}}&=  \bm{\widehat{\Upsilon}}^\top\bm{\Sigma^\star}\bm{\widehat{\Upsilon}}-  \bm{\widehat{\Upsilon}}^\top (\bm{\Sigma^\star}-\bm{Q})\bm{\widehat{\Upsilon}}\\&\stackrel{(i)}{\geq}\kappa_0\|\bm{\widehat{\Upsilon}}\|_2^2 - \sqrt{\frac{\delta M\log d}{n}}\|\bm{\widehat{\Upsilon}}\|_1^2\\
    &\stackrel{(ii)}{\geq} \Big(\kappa_0 - C_6s\sqrt{\frac{\delta M \log d}{n}}\Big) \|\bm{\widehat{\Upsilon}}\|^2_2\stackrel{(iii)}{\geq} \frac{\kappa_0}{2}\|\bm{\widehat{\Upsilon}}\|^2_2,
    \end{aligned}
\end{equation}
where $(i)$ is because $\bm{\widehat{\Upsilon}}^\top (\bm{\Sigma^\star}-\bm{Q})\bm{\widehat{\Upsilon}}\leq \|\bm{\widehat{\Upsilon}}\|_1^2\|\bm{Q}-\bm{\Sigma^\star}\|_\infty$, $(ii)$ is due to $\|\bm{\widehat{\Upsilon}}\|_1\leq 10\sqrt{s}\|\bm{\widehat{\Upsilon}}\|_2$, $(iii)$ is due to the the assumed scaling $n\gtrsim \delta s^2\log d$.
Now the desired results follow immediately from Lemma \ref{csframework}. \hfill $\square$

\subsection{Quantized Matrix Completion}
Under the observation model (\ref{mcmodel}),  we first provide a deterministic framework for analysing the estimator (\ref{5.2}). \begin{lem}
    \label{mcframework}
    {\rm(Adapted from}{\rm  \cite[{Coro. 3}]{chen2022high}){\bf .}} Let $\bm{\widehat{\Upsilon}}:=\bm{\widehat{\Theta}}-\bm{\Theta^\star}$. If \begin{equation}
        \label{biglammc}
        \lambda\geq 2\left\|\frac{1}{n} \sum_{k=1}^n(\big<\bm{X}_k,\bm{\Theta^\star}\big> - \dot{y}_k)\bm{X}_k\right\|_{op},
    \end{equation} then it holds that $\|\bm{\widehat{\Upsilon}}\|_{nu} \leq 10\sqrt{r}\|\bm{\widehat{\Upsilon}}\|_F$. Moreover, if for some $\kappa>0$ we have the restricted strong convexity (RSC) $\frac{1}{n}\sum_{k=1}^n\big|\big<\bm{X}_k,\bm{\widehat{\Upsilon}}\big>\big|^2\geq \kappa \|\bm{\widehat{\Upsilon}}\|_F^2$, then we have the error bounds $\|\bm{\widehat{\Upsilon}}\|_F\leq 30\sqrt{r}\big(\frac{\lambda}{\kappa}\big)$ and $\|\bm{\widehat{\Upsilon}}\|_{nu}\leq 300r\big(\frac{\lambda}{\kappa}\big)$.
\end{lem}
Clearly, to derive statistical error rate of $\bm{\widehat{\Theta}}$ from Lemma \ref{mcframework}, the key ingredients are (\ref{biglammc}) and the RSC. Specialized to the covariate $\bm{X}_k\sim \mathscr{U}\big(\{\bm{e}_i\bm{e}_j^\top:i,j\in [d]\}\big)$ in matrix completion, we will use the following lemma to establish RSC.

\begin{lem}\label{mcrsc}
    {\rm (Adapted from \cite[Lem. 4]{chen2022high} with $q=0$)\textbf{.}} Given some $\alpha>0,\delta>0$, we define the constraint set $\mathcal{\psi}$ with sufficiently large $\psi$ as \begin{equation}
        \begin{aligned}\label{mcconstra}
       \mathcal{C}(\psi)=\Big\{\bm{\Theta}\in\mathbb{R}^{d\times d}: \|\bm{\Theta}&\|_\infty\leq 2\alpha,\|\bm{\Theta}\|_{nu}\leq 10\sqrt{r}\|\bm{\Theta}\|_F,\\&\|\bm{\Theta}\|_F^2 \geq (\alpha d)^2\sqrt{\frac{\psi\delta \log d}{n}}\Big\}.
       \end{aligned}
    \end{equation}  
    Let $\bm{X}_1,...,\bm{X}_n$ be i.i.d. uniformly distributed on $\{\bm{e}_i\bm{e}_j^\top:i,j\in [d]\}$, then there exist  absolute constants $\kappa\in (0,1)$ and $C$, such that with probability at least $1-d^{-\delta}$ we have \begin{equation}
        \label{mcrsceq}
        \frac{1}{n}\sum_{k=1}^n \big|\big<\bm{X}_k,\bm{\Theta}\big>\big|^2 \geq \frac{\kappa\|\bm{\Theta}\|_F^2}{d^2}-\frac{C\alpha^2rd\log d}{n},~\forall~ \bm{\Theta}\in \mathcal{C}(\psi).
    \end{equation}
\end{lem}
Matrix completion with sub-exponential noise was studied in \cite{klopp2014noisy}, and we make use of the following Lemma in the  sub-exponential case. 
\begin{lem}
    \label{kloppbound}
    {\rm (Adapted from \cite[Lem. 5]{klopp2014noisy})\textbf{.}} Given some $\delta>0$. Let $\bm{X}_1,...,\bm{X}_n$ be i.i.d. uniformly distributed on $\{\bm{e}_i\bm{e}_j^\top:i,j\in [d]\}$, independent of $\bm{X}_k$s, $\epsilon_1,...,{\epsilon_n}$ are i.i.d. zero-mean and satisfy $\|\epsilon_k\|_{\psi_1}\leq \sigma$. If $n\gtrsim \delta d \log ^3 d$, with probability at least $1-d^{-\delta}$ we have $$\Big\|\frac{1}{n}\sum_{k=1}^n \epsilon_k\bm{X}_k\Big\|_{op}\leq \sigma \sqrt{\frac{\delta \log d}{nd}}.$$  
\end{lem}

\subsubsection{Proof of Theorem \ref{thm8}}
\noindent{\it Proof.} We divide the proof into two steps.

\noindent{\bf \sffamily  Step 1.   Proving (\ref{biglammc})}

Defining $w_k:=\dot{y}_k - y_k-\tau_k$ as the quantization error, from Theorem \ref{lem1}(a) we know that $w_k$s are independent of $\bm{X}_k$ and i.i.d. uniformly distributed on $[-\frac{\Delta}{2},\frac{\Delta}{2}]$. Thus, we can further write that $$\dot{y}_k = y_k+\tau_k+w_k=\big<\bm{X}_k,\bm{\Theta^\star}\big>+\epsilon_k+\tau_k+w_k,$$  which allows us to decompose $I$ into $$I\leq \Big\|\frac{1}{n}\sum_{k=1}^n \epsilon_k\bm{X}_k\Big\|_{op}+\Big\|\frac{1}{n}\sum_{k=1}^n \tau_k\bm{X}_k\Big\|_{op}+ \Big\|\frac{1}{n}\sum_{k=1}^n w_k\bm{X}_k\Big\|_{op}=I_1+I_2+I_3.$$ Because
$\epsilon_k$s are independent of $\bm{X}_k$s and i.i.d. sub-exponential noise satisfying $\|\epsilon_k\|_{\psi_1}\leq \sigma$,
under the scaling $n\gtrsim \delta d\log^3d$,  Lemma \ref{kloppbound} implies that $I_1\lesssim \sigma \sqrt{\frac{\delta \log d}{nd}}$ holds with probability at least $1-d^{-\delta}$. Analogously, $\tau_k$s and $w_k$s are independent of $\{\bm{X}_k:k\in[n]\}$ and are  i.i.d. uniformly distributed on $[-\frac{\Delta}{2},\frac{\Delta}{2}]$,
 Lemma \ref{kloppbound} also applies to $I_2$ and $I_3$, yielding that with the promised probability $I_2+I_3\lesssim \Delta\sqrt{\frac{\delta \log d}{nd}}$. Taken collectively, $I\lesssim (\sigma+\Delta)\sqrt{\frac{\delta \log d}{nd}}$, so setting $\lambda = C_1(\sigma+\Delta)\sqrt{\frac{\delta \log d}{nd}}$ with sufficiently large $C_1$ ensures $\lambda\geq 2I$, with probability at least $1-3d^{-\delta}$. Further, Lemma \ref{mcframework} gives  $\|\bm{\widehat{\Upsilon}}\|_{nu}\leq 10\sqrt{r}\|\bm{\widehat{\Upsilon}}\|_F$.

\vspace{1mm}

\noindent{\bf \sffamily  Step 2. Verifying   RSC  }

First note that $\|\bm{\widehat{\Upsilon}}\|_\infty \leq \|\bm{\widehat{\Theta}}\|_\infty+\|\bm{\Theta^\star}\|_\infty\leq 2\alpha$;  and as proved before, $\|\bm{\widehat{\Upsilon}}\|_{nu}\leq 10\sqrt{r}\|\bm{\widehat{\Upsilon}}\|_F$. To proceed we  define the constraint set $\mathcal{C}(\psi)$ as in (\ref{mcconstra}) with some properly chosen constant $\psi$. Then using Lemma \ref{mcrsc}, for some absolute constants $\kappa,C$, (\ref{mcrsceq}) holds with probability at least $1-d^{-\delta}$. Then we discuss several cases.

1) If $\bm{\widehat{\Upsilon}}\notin \mathcal{C}(\psi)$, because $\bm{\widehat{\Upsilon}}$ satisfies the first two constraints in the definition of $\mathcal{C}(\psi)$, it must violate the third constraint and satisfy $\|\bm{\widehat{\Upsilon}}\|_F^2 \leq (\alpha d)^2\sqrt{\frac{\psi\delta \log d}{n}}$, which gives $\|\bm{\widehat{\Upsilon}}\|_F\lesssim \alpha d\big(\frac{\delta \log d}{n}\big)^{1/4}\stackrel{(i)}{\lesssim} \alpha d\sqrt{\frac{\delta r d \log d}{n}}$, as desired. Note that $(i)$ is due to the scaling $n\lesssim \delta r^2d^2 \log d$.

2) If $\bm{\widehat{\Upsilon}}\in\mathcal{C}(\psi)$, (\ref{mcrsceq}) implies that  $\frac{1}{n}\sum_{k=1}^n \big|\big<\bm{X}_k,\bm{\widehat{\Upsilon}}\big>\big|^2 \geq   \frac{\kappa\|\bm{\widehat{\Upsilon}}\|_F^2}{d^2} - C  \frac{\alpha^2 rd\log d}{n}$, and we further consider the  following two cases.

2.1) If $C    \frac{\alpha^2 rd\log d}{n} \geq \frac{\kappa\|\bm{\widehat{\Upsilon}}\|_F^2}{2d^2}$, we have $\|\bm{\widehat{\Upsilon}}\|_F\lesssim \alpha d \sqrt{\frac{rd\log d}{n}}$, as desired.

2.2) If $C    \frac{\alpha^2rd\log d}{n} < \frac{\kappa\|\bm{\widehat{\Upsilon}}\|_F^2}{2d^2}$, then     the RSC condition holds: $\frac{1}{n}\sum_{k=1}^n \big|\big<\bm{X}_k,\bm{\widehat{\Upsilon}}\big>\big|^2\geq \frac{\kappa\|\bm{\widehat{\Upsilon}}\|^2_F}{2d^2}$. This allows us to apply Lemma \ref{mcframework} to obtain $\|\bm{\widehat{\Upsilon}}\|_F \lesssim (\sigma+\Delta)d\sqrt{\frac{\delta r d\log d}{n}}$.

Thus, in any case, we have shown $\|\bm{\widehat{\Upsilon}}\|_F =O\big((\alpha+\sigma+\Delta)d\sqrt{\frac{\delta rd \log d}{n}}\big)$. The bound on $\|\bm{\widehat{\Upsilon}}\|_{nu}$ follows immediately from $\|\bm{\widehat{\Upsilon}}\|_{nu}\leq 10\sqrt{r}\|\bm{\widehat{\Upsilon}}\|_F$. The proof is complete. \hfill $\square$ 
\subsubsection{Proof of Theorem \ref{thm9}}

\noindent{\it Proof.} The proof is based on Lemma \ref{mcframework} and divided into two steps.

\noindent{\bf \sffamily  Step 1.   Proving (\ref{biglammc})}

  Recall that the quantization error $w_k:=\dot{y}_k-\widetilde{y}_k-\tau_k$ is zero-mean and independent of $\bm{X}_k$ (Theorem \ref{lem1}(a)),  thus we have $ \mathbbm{E}(\dot{y}_k\bm{X}_k)=\mathbbm{E}(\widetilde{y}_k\bm{X}_k)+ \mathbbm{E}(\tau_k\bm{X}_k)+\mathbbm{E}(w_k\bm{X}_k) = \mathbbm{E}(\widetilde{y}_k\bm{X}_k).$  Combining with $\mathbbm{E}\big(\big<\bm{X}_k,\bm{\Theta^\star}\big>\bm{X}_k\big)= \mathbbm{E}(y_k\bm{X}_k)$, 
triangle inequality can first decompose the target term into \begin{equation}
    \begin{aligned}\nonumber
  \Big\|\frac{1}{n}\sum_{k=1}^n \big(\dot{y}_k - \big<\bm{X}_k,\bm{\Theta^\star}\big>\big)\bm{X}_k\Big\|_{op}&\leq   \Big\| \frac{1}{n}\sum_{k=1}^n \dot{y}_k\bm{X}_k - \mathbbm{E}(\dot{y}_k\bm{X}_k)\Big\|_{op}\\
    + \Big\|\mathbbm{E}\big(y_k\bm{X}_k-\widetilde{y}_k\bm{X}_k\big)\Big\| _{op} + \Big\|& \frac{1}{n}\sum_{k=1}^n\big<\bm{X}_k,\bm{\Theta^\star}\big>\bm{X}_k-\mathbbm{E}\big(\big<\bm{X}_k,\bm{\Theta^\star}\big>\bm{X}_k\big)\Big\| _{op}\\&:=I_1+I_2+I_3.
    \end{aligned}
\end{equation}

\noindent{\bf \sffamily Step 1.1. Bounding $I_1$ and $I_3$}


We write  $ I_1 = \|\sum_{k=1}^n \bm{S}_k\|_{op}$ and $I_3=\|\sum_{k=1}^n \bm{W}_k\|_{op}$ by defining \begin{equation}
    \begin{aligned}\nonumber
        \bm{S}_k = \frac{1}{n}\big(\dot{y}_k\bm{X}_k-\mathbbm{E}(\dot{y}_k\bm{X}_k)\big),\bm{W}_k= \frac{1}{n}\big(\big<\bm{X}_k,\bm{\Theta^\star}\big>\bm{X}_k-\mathbbm{E}\big[\big<\bm{X}_k,\bm{\Theta^\star}\big>\bm{X}_k\big]\big),
    \end{aligned}
\end{equation}
 By  $|\dot{y}_k| \leq |\widetilde{y}_k|+|\tau_k| +|w_k|\leq\zeta_y+\Delta$ we have  $$\|\bm{S}_k\|_{op}\leq \frac{1}{n}\|\dot{y}_k\bm{X}_k\|_{op}+\frac{1}{n}\|\mathbbm{E}(\dot{y}_k\bm{X}_k)\|_{op}\leq \frac{1}{n}\|\dot{y}_k\bm{X}_k\|_{op}+\frac{1}{n}\mathbbm{E}\|\dot{y}_k\bm{X}_k\|_{op} \leq \frac{2(\zeta_y+\Delta)}{n} .$$ Analogously, we have $\|\bm{W}_k\|_{op}\leq \frac{2\alpha}{n}$ since $\big|\big<\bm{X}_k,\bm{\Theta^\star}\big>\big|\leq \|\bm{\Theta^\star}\|_\infty\leq \alpha$. In addition, by $\|\mathbbm{E}\{(\bm{A}-\mathbbm{E}\bm{A})^\top(\bm{A}-\mathbbm{E}\bm{A})\}\|_{op}\leq \|\mathbbm{E}(\bm{A}^\top\bm{A})\|_{op}$ ($\forall \bm{A}$) and the simple fact $\mathbbm{E}(\bm{X}_k\bm{X}_k^\top) = \mathbbm{E}(\bm{X}_k^\top\bm{X}_k) = \bm{I}_d/d$, we estimate the matrix variance statistic   as follows\begin{equation}
    \begin{aligned}\nonumber
      &\Big\|\sum_{k=1}^n\mathbbm{E}(\bm{S}_k\bm{S}_k^\top)\Big\|_{op}=n\big\|\mathbbm{E}(\bm{S}_k\bm{S}_k^\top)\big\|_{op} \leq \frac{1}{n} \big\|\mathbbm{E}\big(\dot{y}_k^2\bm{X}_k\bm{X}_k^\top\big)\big\|_{op}\\&=\frac{1}{n}\sup_{\bm{v}\in\mathbb{S}^{d-1}}\mathbbm{E} \big(\dot{y}_k^2\cdot \|\bm{X}_k^\top \bm{v}\|_2^2\big) =\frac{1}{n}\sup_{\bm{v}\in\mathbb{S}^{d-1}} \mathbbm{E}_{\bm{X}_k} \Big(\big[\mathbbm{E}_{\dot{y}_k|\bm{X}_k}(\dot{y}_k^2)\big]\|\bm{X}_k^\top\bm{v}\|_2^2\Big)\\&\stackrel{(i)}{\leq} \frac{4}{n}(\alpha^2+M+\Delta^2)\sup_{\bm{v}\in \mathbb{S}^{d-1}}\mathbbm{E}_{\bm{X}_k}\|\bm{X}_k^\top \bm{v}\|^2_2 \leq \frac{4(\alpha^2+M+\Delta^2)}{nd},
    \end{aligned}
\end{equation}
where $(i)$ is because given $\bm{X}_k$ we can estimate $\mathbbm{E}_{\dot{y}_k|\bm{X}_k}(\dot{y}_k^2)\leq 2\big(\mathbbm{E}_{\dot{y}_k|\bm{X}_k}(\widetilde{y}_k^2)+\Delta^2\big)$ since $|\dot{y}_k-\widetilde{y}_k|\leq \Delta$, and moreover $ \mathbbm{E}_{\dot{y}_k|\bm{X}_k}(\widetilde{y}_k^2)\leq\mathbbm{E}_{\dot{y}_k|\bm{X}_k}(y_k^2)\leq 2\big(\mathbbm{E}_{\dot{y}_k|\bm{X}_k}(\big<\bm{X}_k,\bm{\Theta^\star}\big>^2)+\mathbbm{E}_{\dot{y}_k|\bm{X}_k}(\epsilon_k^2)\big)\leq 2(\alpha^2+ M).$  It is not hard to see that   this bound remains valid for $\|\sum_{k=1}^n\mathbbm{E}(\bm{S}_k^\top \bm{S}_k)\|_{op}$. Also, by similar arguments one can prove $$\max\left\{\Big\|\sum_{k=1}^n\mathbbm{E}(\bm{W}_k^\top \bm{W}_k)\Big\|_{op},\Big\|\sum_{k=1}^n\mathbbm{E}(\bm{W}_k\bm{W}_k^\top)\Big\|_{op}\right\}\leq \frac{\alpha^2}{nd}.$$ Thus, Matrix Bernstein's inequality   (Lemma \ref{matrixbern})  gives 
\begin{gather}
    \mathbbm{P}\big(I_1 \geq t\big) \leq 2d\cdot \exp\Big(-\frac{C_4ndt^2}{(\alpha^2+M+\Delta^2)+(\zeta_y+\Delta)dt}\Big)\nonumber\\\mathbbm{P} \big(I_3\geq t\big)\leq 2d\cdot\exp \Big(-\frac{C_5ndt^2}{\alpha^2+ \alpha dt}\Big)\nonumber
\end{gather} 
Thus, setting $t = C_6 (\alpha+\sqrt{M}+\Delta)\sqrt{\frac{\delta\log d}{nd}}$ in the two inequalities above with sufficiently large $C_6$, combining with the {scaling} that $\sqrt{\frac{\delta d \log d}{n}}= O(1)$, we obtain $I_1+I_3\lesssim (\alpha+\sqrt{M}+\Delta)\sqrt{\frac{\delta \log d}{nd}}$ with probability at least $1-4d^{1-\delta}$. 

\noindent{\bf \sffamily Step 1.2. Bounding   $I_2$}

Let us bound $\|\mathbbm{E}\big((y_k-\widetilde{y}_k)\bm{X}_k\big)\|_\infty$ first. Write $(i,j)$-th entry of $\bm{X}_k$ as $x_{k,ij}$, then for given $(i,j)$, $\mathbbm{P}(x_{k,ij}=1) = d^{-2}$, $x_{k,ij}=0$ otherwise. We can thus proceed by the following estimations:\begin{equation}
      \begin{aligned}\nonumber
        \big|\mathbbm{E}\big((y_k-\widetilde{y}_k)x_{k,ij}\big)\big|&=\big|\mathbbm{E}\big((y_k-\widetilde{y}_k)x_{k,ij}\mathbbm{1}(|y_k|\geq \zeta_y)\big)\big| \\&\leq \mathbbm{E}\big(|y_k|x_{k,ij}\mathbbm{1}(|y_k|\geq \zeta_y)\big)\\
        &= \mathbbm{E}_{x_{k,ij}}\big(\big\{\mathbbm{E}_{y_k|x_{k,ij}}|y_k|\mathbbm{1}(|y_k|\geq \zeta_y)\big\}x_{k,ij}\big)  \\& = d^{-2} \mathbbm{E}_{y_k|x_{k,ij}=1}\big(|y_k|\mathbbm{1}(|y_k|\geq \zeta_y)\big)\\
        &\stackrel{(i)}{\leq}d^{-2}\sqrt{\mathbbm{E}_{y_k\sim \theta^\star_{ij}+\epsilon_k}(y_k^2)}\sqrt{\mathbbm{P}_{y_k\sim \theta^\star_{ij}+\epsilon_k}(y_k^2\geq \zeta_y^2)}\\&\stackrel{(ii)}{\leq}d^{-2}\frac{\alpha^2+M}{\zeta_y}\lesssim \frac{\alpha+\sqrt{M}}{d^2}\sqrt{\frac{\delta d\log d}{n}},
      \end{aligned}
  \end{equation}
where $(i)$, $(ii)$ is by Cauchy-Schwarz and Markov's, respectively. Since this holds for any $(i,j)$, we obtain $\|\mathbbm{E}\big((y_k-\widetilde{y}_k)\bm{X}_k\big)\|_\infty =O \big((\alpha+\sqrt{M})d^{-2}\sqrt{\frac{\delta d\log d}{n}}\big)$, which   further gives $I_2 = O \big((\alpha+\sqrt{M})\sqrt{\frac{\delta \log d}{nd}}\big)$ by using $\|\bm{A}\|_{op}\leq d\|\bm{A}\|_\infty$ ($\forall\bm{A}\in \mathbb{R}^{d\times d}$). Putting pieces together, with probability at least $1-4d^{1-\delta}$ we have $\|\frac{1}{n}\sum_k \big(\dot{y}_k - \big<\bm{X}_k,\bm{\Theta^\star}\big>\big)\bm{X}_k\|_{op} \lesssim (\alpha+\sqrt{M}+\Delta)\sqrt{\frac{\delta \log d}{nd}}$, hence $\lambda =C_1 (\alpha+\sqrt{M}+\Delta) \sqrt{\frac{\delta\log d}{nd}}$ ensures (\ref{biglammc}) under the same probability. Further, Lemma \ref{mcframework} gives  $\|\bm{\widehat{\Upsilon}}\|_{nu}\leq 10\sqrt{r}\|\bm{\widehat{\Upsilon}}\|_F$.

\vspace{1mm}

\noindent{\bf \sffamily  Step 2. Verifying  RSC}

The remaining part is almost the same as  { Step 2} in the proof of Theorem \ref{thm8} --- defining the constraint set $\mathcal{C}(\psi)$ as (\ref{mcconstra}) and then discussing several cases based on whether $\bm{\widehat{\Upsilon}}\in \mathcal{C}(\psi)$ holds.  Thus, we conclude the proof without providing the details. \hfill $\square$

\section{Proofs in Section \ref{sec:qc-qcs}}

This appendix collects the proofs in Section \ref{sec:qc-qcs} concerning covariate quantization and uniform signal recovery in QCS. 
\subsection{Covariate Quantization}
Because of the non-convexity, the proofs in this part can no longer be based on Lemma \ref{csframework}. Indeed, bounding the estimation errors of $\bm{\widetilde{\theta}}$s satisfying (\ref{4.17}) require more tedious manipulations essentially due to the additional $\ell_1$ constraint (induced by the constraint $\mathcal{S}$ in (\ref{4.17})).  
\subsubsection{Proof of Theorem \ref{thm6}}

\noindent{\it Proof.}  The proof is divided into three steps --- the first two steps resemble the previous proofs that are based on Lemma \ref{csframework}, while we bound the estimation errors in the last step.

\noindent{\bf \sffamily  Step 1. Proving $\lambda \geq \beta\|\bm{Q\theta^\star} - \bm{b}\|_{\infty}$   for some pre-specified $\beta>2$}

Recall that $(\bm{Q},\bm{b})$ are constructed from the quantized data as $\bm{Q}=\frac{1}{n}\sum_{k=1}^n \bm{\dot{x}}_k\bm{\dot{x}}_k^\top - \frac{\bar{{\Delta}}^2}{4}\bm{I}_d$ and $\bm{b}=\frac{1}{n}\sum_{k=1}^n \dot{y}_k\bm{\dot{x}}_k$. We will show that, $\lambda = C_1 \frac{(\sigma+\bar{\Delta})^2}{\sqrt{\kappa_0}}(\Delta +M^{1/(2l)})\sqrt{\frac{\log d}{n}}$ guarantees $\lambda \geq  \beta\|\frac{1}{n}\sum_{k=1}^n (\bm{\dot{x}}_k\bm{\dot{x}}_k^\top - \frac{\bar{\Delta}^2}{4}\bm{I}_d)\bm{\theta^\star} - \frac{1}{n}\sum_{k=1}^n \dot{y}_k\bm{\dot{x}}_k\|_\infty$ holds with the promised probability, where $\beta>2$ is any pre-specified constant.  Recall the notation we  introduced: $\dot{y}_k = \widetilde{y}_k + \phi_k + \vartheta_k$ with the   quantization error $\vartheta_k\sim \mathscr{U}([-\frac{\Delta}{2},\frac{\Delta}{2}])$ being independent of $\widetilde{y}_k$, $\bm{\dot{x}}_k = \bm{x}_k + \bm{\tau}_k+\bm{w}_k$ with  the   quantization error $\bm{w}_k \sim \mathscr{U}([-\frac{\bar{\Delta}}{2},\frac{\bar{\Delta}}{2}]^d)$ being independent of $\bm{x}_k$. Combining with the assumptions that the dithers are independent of $(\bm{x}_k,y_k)$ and  that $\phi_k$s and $\bm{\tau}_k$s are independent, we have \begin{equation}
\begin{aligned}\label{628add1}
    \mathbbm{E}(\dot{y}_k&\bm{\dot{x}}_k)=\mathbbm{E}\big((\widetilde{y}_k+\phi_k+\vartheta_k)(\bm{x}_k+\bm{\tau}_k+\bm{w}_k)\big)= \mathbbm{E}(\widetilde{y}_k\bm{x}_k),\\
    &\mathbbm{E}\Big(\Big[\bm{\dot{x}}_k\bm{\dot{x}}_k^\top - \frac{\bar{\Delta}^2}{4}\bm{I}_d\Big]\bm{\theta^\star}\Big)=\mathbbm{E}(\bm{x}_k\bm{x}_k^\top \bm{\theta^\star}) = \mathbbm{E}(y_k\bm{x}_k),
\end{aligned}
\end{equation}  which allows us to decompose the target term as two concentration terms ($I_1,I_3$) and a bias term $(I_2)$\begin{equation}
    \begin{aligned}\nonumber
  &  \Big\|\frac{1}{n}\sum_{k=1}^n \Big[\bm{\dot{x}}_k\bm{\dot{x}}_k^\top - \frac{\bar{\Delta}^2}{4}\bm{I}_d\Big]\bm{\theta^\star} - \frac{1}{n}\sum_{k=1}^n \dot{y}_k\bm{\dot{x}}_k\Big\|_\infty\leq  {\Big\|\frac{1}{n}\sum_{k=1}^n\dot{y}_k\bm{\dot{x}}_k-\mathbbm{E}(\dot{y}_k\bm{\dot{x}}_k)\Big\|_{\infty}}  \\&+ {\Big\|\mathbbm{E}\big(y_k\bm{x}_k- \widetilde{y}_k \bm{x}_k\big)\Big\|_\infty}   +  {\Big\|\frac{1}{n}\sum_{k=1}^n\bm{\dot{x}}_k\bm{\dot{x}}_k^\top \bm{\theta^\star}-\mathbbm{E}(\bm{\dot{x}}_k\bm{\dot{x}}_k^\top \bm{\theta^\star})\Big\|_\infty}:=I_1+I_2+I_3.
    \end{aligned}
\end{equation}

\noindent{\bf \sffamily Step 1.1. Bounding $I_1$}

Denote the $i$-th entry of $\bm{x}_k,\bm{\dot{x}}_k,\bm{\tau}_k,\bm{w}_k$ by $x_{ki},\dot{x}_{ki},\tau_{ki},w_{ki}$, respectively.
For $I_1$, the $i$-th entry reads $\frac{1}{n}\sum_{k=1}^n\dot{y}_k\dot{x}_{ki}-\mathbbm{E}(\dot{y}_k\dot{x}_{ki})$.  By using the relations $
        |\dot{y}_k|\leq |\widetilde{y}_k|+|\phi_k|+|\vartheta_k|\leq \zeta_y+\Delta$, $\|\bm{\dot{x}}_{k}\|_{\psi_2}\leq \|\bm{x}_{k}\|_{\psi_2}+\|\bm{\tau}_{k}\|_{\psi_2}+\|\bm{w}_{k}\|_{\psi_2}\lesssim \sigma+\bar{\Delta}$ and $\mathbbm{E}|\dot{y}_k|^{2l}\lesssim \mathbbm{E}|\widetilde{y}_k|^{2l} + \mathbbm{E}|\phi_k+\vartheta_k|^{2l}\lesssim M+\Delta^{2l}$,  for any integer $q\geq 2$ we can bound that \begin{equation}
    \begin{aligned}\label{nB.2}
    &\sum_{k=1}^n \mathbbm{E}\Big|\frac{\dot{y}_k\dot{x}_{ki}}{n}\Big|^q \leq \frac{(\zeta_y+\Delta)^{q-2} }{n^q}\sum_{k=1}^n \mathbbm{E}|\dot{y}_k^2\dot{x}_{ki}^q|\\
    & \stackrel{(i)}{\leq}\frac{(\zeta_y+\Delta)^{q-2} }{n^q}\sum_{k=1}^n \big\{\mathbbm{E}|\dot{y}_k|^{2l}\big\}^{\frac{1}{l}} \big\{\mathbbm{E}|\dot{x}_{ki}|^{\frac{lq}{l-1}}\big\}^{1-\frac{1}{l}}\\&\stackrel{(ii)}{\lesssim} \Big(\frac{(\sigma+\bar{\Delta})(\zeta_y+\Delta)}{n}\Big)^{q-2} \Big(\frac{(\sigma+\bar{\Delta})^2(M^{\frac{1}{l}}+\Delta^2)}{n}\Big)\Big(\sqrt{\frac{lq}{l-1}}\Big)^q;
    \end{aligned}
\end{equation}
combining with Stirling's approximation and treating $l$ as absolute constant, this provides$$  \sum_{k=1}^n \mathbbm{E}\Big|\frac{\dot{y}_k\dot{x}_{ki}}{n}\Big|^q\leq  \frac{q!}{2}v_0c_0^{q-2} \text{ where }v_0=O\Big(\frac{(\sigma+\bar{\Delta})^2(M^{{1}/{l}}+\Delta^2)}{n}\Big),~ c_0=O\Big(\frac{(\sigma+\bar{\Delta})(\zeta_y+\Delta)}{n}\Big).$$In (\ref{nB.2}), 
 $(i)$ is due to Holder's, and in $(ii)$ we use the moment constraint of sub-Gaussian variable (\ref{2.2b}). With these preparations, we invoke Bernstein's inequality  (Lemma \ref{bernstein})   and then a union bound over $i\in [d]$ to obtain
$$\mathbbm{P}\Big(I_1\lesssim (\sigma+\bar{\Delta}) (M^{\frac{1}{2l}}+\Delta)\sqrt{\frac{t}{n}}+ \frac{(\sigma+\bar{\Delta})(\zeta_y+\Delta)t}{n}\Big)\geq 1-2d\cdot \exp(-t),$$
  Thus, taking $t =\delta \log d$ and plug in $\zeta_y\asymp\sqrt{\frac{nM^{1/l}}{\delta \log d}}$, we obtain 
  $$\mathbbm{P}\Big(I_1\lesssim (\sigma+\bar{\Delta}) (M^{1/(2l)}+\Delta)\sqrt{\frac{\delta \log d}{n}}\Big)\geq 1-2d^{1-\delta}.$$

\noindent{\bf \sffamily Step 1.2. Bounding $I_2$}

Moreover,  we   estimate the $i$-th entry of $I_2$ by \begin{equation}
    \begin{aligned}\label{nB.3}
      &|\mathbbm{E}\big((y_k-\widetilde{y}_k)x_{ki}\big)|\leq \mathbbm{E}|y_kx_{ki}\mathbbm{1}(|y_k|\geq \zeta_y)|\\&\stackrel{(i)}{\leq} \big(\mathbbm{E}|y_k|^{\frac{2l}{2l-1}}|x_{ki}|^{\frac{2l}{2l-1}}\big)^{1-\frac{1}{2l}}\big(\mathbbm{P}(|y_k|\geq \zeta_y)\big)^{\frac{1}{2l}} \\
      &\stackrel{(ii)}{\leq}\Big(\big[\mathbbm{E}|y_k|^{2l}\big]^{\frac{1}{2l-1}}\big[\mathbbm{E}|x_{ki}|^{\frac{l}{l-1}}\big]^{\frac{2l-2}{2l-1}}\Big)^{1-\frac{1}{2l}}\Big(\mathbbm{P}(|y_k|^{2l}\geq \zeta_y^{2l})\Big)^{\frac{1}{2l}}\\&\stackrel{(iii)}{\leq} \frac{\sigma M^{1/l}}{\zeta_y}\lesssim \sigma M^{\frac{1}{2l}}\sqrt{\frac{\delta \log d}{n}},
    \end{aligned}
\end{equation}
where $(i)$, $(ii)$ are due to Holder's, $(iii)$ is due to Markov's. Since this holds for $i\in [d]$, it gives $I_2 \lesssim \sigma M^{1/(2l)}\sqrt{\frac{\delta \log d}{n}}$.

\vspace{1mm}
\noindent{\bf \sffamily Step 1.3. Bounding $I_3$}
\vspace{1mm}

We first derive a bound for $\|\bm{\theta^\star}\|_2$ that is implicitly implied by other conditions:  $$M^{1/l}\geq \mathbbm{E}|y_k|^2 \geq \mathbbm{E}(\bm{x}_k^\top \bm{\theta^\star})^2 = (\bm{\theta^\star})^\top\bm{\Sigma^\star}\bm{\theta^\star}\geq  \kappa_0\|\bm{\theta^\star}\|_2^2 \Longrightarrow\|\bm{\theta^\star}\|_2 =O\Big(\frac{M^{1/(2l)}}{\sqrt{\kappa_0}}\Big).$$
Hence, we can estimate $\|(\bm{\dot{x}}_k^\top\bm{\theta^\star})\dot{x}_{ki}\|_{\psi_1}\leq \|\bm{\dot{x}}_k^\top\bm{\theta^\star}\|_{\psi_2}\|\dot{x}_{ki}\|_{\psi_2}\leq \|\bm{\dot{x}}_k\|_{\psi_2}^2 \|\bm{\theta^\star}\|_2\lesssim {(\sigma+\bar{\Delta})^2} \frac{M^{1/(2l)}}{\sqrt{\kappa_0}}$. Then, we invoke Bernstein's inequality regarding the independent sum of sub-exponential random variables (e.g.,  \cite[Thm. 2.8.1]{vershynin2018high})\footnote{The application of Bernstein's inequality leads to the $\sigma^2$ ($\sigma$ is the upper bound on $\|\bm{x}_k\|_{\psi_2}$) dependence in the multiplicative factor $\mathscr{L}$. It is possible to refine this quadratic dependence by using a new Bernstein's inequality developed in \cite[Thm. 1.3]{jeong2022sub}, but we do not pursue this in the present paper.} to obtain  that for any $t\geq 0$ we have \begin{equation}
    \begin{aligned}\nonumber
      &\mathbbm{P}\left(\Big|\frac{1}{n}\sum_{k=1}^n (\bm{\dot{x}}_k^\top \bm{\theta^\star})\dot{x}_{ki} -\mathbbm{E}\{(\bm{\dot{x}}_k^\top \bm{\theta^\star})\dot{x}_{ki}\}\Big|\geq t\right)\\&~~~~~~~~~~~~~~~~~~\leq 2\exp\left(-C_5n\min\left\{\frac{\sqrt{\kappa_0} t}{(\sigma+\bar{\Delta})^2 M^{\frac{1}{2l}}},\Big(\frac{\sqrt{\kappa_0} t}{(\sigma+\bar{\Delta})^2 M^{\frac{1}{2l}}}\Big)^2\right\}\right) 
    \end{aligned}
\end{equation}
Hence, we can set $t = C_6\frac{(\sigma+\bar{\Delta})^2}{\sqrt{\kappa_0}} M^{1/(2l)}\sqrt{\frac{\delta\log d}{n}}$ with sufficiently large $C_6$, and further apply union bound over $i\in [d]$, under the scaling that $\frac{\delta \log d}{n}$ is small enough, we obtain $I_3\lesssim \frac{(\sigma+\bar{\Delta})^2}{\sqrt{\kappa_0}} M^{1/(2l)}\sqrt{\frac{\delta \log d}{n}}$ holds with probability at least $1-2d^{1-\delta}$.

Putting pieces together, since $\kappa_0 \lesssim \sigma^2$, it is immediate that $I\lesssim\frac{(\sigma+\bar{\Delta})^2}{\sqrt{\kappa_0}} \big(\Delta + M^{1/(2l)}\big)\sqrt{\frac{\delta \log d}{n}}$ holds with probability at least $1-8d^{1-\delta}$. Since $\lambda = C_1 \frac{(\sigma+\bar{\Delta})^2}{\sqrt{\kappa_0}}(\Delta +M^{1/(2l)})\sqrt{\frac{\delta\log d}{n}}$ with sufficiently large $C_1$, $\lambda \geq \beta \cdot\|\bm{Q\theta^\star}- \bm{b}\|_\infty$ holds w.h.p.

\vspace{1mm}

\noindent{\bf \sffamily Step 2. Verifying   RSC}

We provide a lower bound for $\bm{v}^\top \bm{Q}\bm{v}= \frac{1}{n}\|\bm{\dot{X}v}\|_2^2- \frac{\bar{\Delta}^2}{4}\|\bm{v}\|_2^2$ by using the matrix deviation inequality (Lemma \ref{deviation}). First note that the rows of  $\bm{\dot{X}}$ are sub-Gaussian $\|\bm{\dot{x}}_k\|_{\psi_2} \lesssim \sigma + \bar{\Delta}$. Since $\mathbbm{E}(\bm{\dot{x}}_k\bm{\dot{x}}_k^\top)=\bm{\Sigma^\star}+\frac{\bar{\Delta}^2}{4}\bm{I}_d$, all eigenvalues of $\bm{\dot{\Sigma}}:=\mathbbm{E}(\bm{\dot{x}}_k\bm{\dot{x}}_k^\top)$ are between $[\kappa_0+\frac{1}{4}\bar{\Delta}^2,\kappa_1+\frac{1}{4}\bar{\Delta}^2]$. Thus, we invoke Lemma \ref{deviation} for $\mathcal{T}:=\{\bm{v}\in \mathbb{R}^d:\|\bm{v}\|_1= 1\}$ with $u=\sqrt{\delta \log d}$; due to the well-known Gaussian width estimate  $\omega(\mathcal{T})\lesssim \sqrt{\log d}$ \cite[Example 7.5.9]{vershynin2018high}, with probability at least $1-d^{-\delta}$ the following event holds
\begin{equation}\nonumber
\begin{aligned}
\sup_{\|\bm{v}\|_1=1}\Big|\|\bm{\dot{X}v}\|_2-\sqrt{n}\|\bm{\dot{\Sigma}}^{1/2}\bm{v}\|_2\Big|&\leq \frac{c_1\sqrt{\kappa_1+\frac{1}{4}\bar{\Delta}^2}(\sigma+ \bar{\Delta})^2}{\kappa_0+\frac{1}{4}\bar{\Delta}^2}\sqrt{\delta \log d}\\&:= c_1 \mathscr{L}_1 \sqrt{\delta \log d}.
\end{aligned}
\end{equation}
Under the same probability, a simple  rescaling then provides 
\begin{equation}
    \label{nB.6}
 \Big|\frac{1}{\sqrt{n}}\|\bm{\dot{X}v}\|_2-\|\bm{\dot{\Sigma}}^{1/2}\bm{v}\|_2\Big| \leq c_1\mathscr{L}_1\sqrt{\frac{\delta \log d}{n}}\|\bm{v}\|_1,~\forall~\bm{v}\in \mathbb{R}^d,
\end{equation}
which implies  \begin{equation}
    \begin{aligned}\label{nB.7}
        &\frac{1}{\sqrt{n}}\|\bm{\dot{X}v}\|_2 \geq \|\bm{\dot{\Sigma}}^{1/2}\bm{v}\|_2 - c_1 \mathscr{L}_1\Big({\frac{\delta \log d}{n}}\Big)^{1/2}\|\bm{v}\|_1\\&\geq \Big({\kappa_0+\frac{1}{4}\bar{\Delta}^2}\Big)^{1/2} \|\bm{v}\|_2-  c_1 \mathscr{L}_1\Big({\frac{\delta \log d}{n}}\Big)^{1/2}\|\bm{v}\|_1,~\forall~\bm{v}\in \mathbb{R}^d.
    \end{aligned}
\end{equation}
 Based on (\ref{nB.7}), we let $\hat{c}:=\frac{2\kappa_0+\bar{\Delta}^2}{4\kappa_0+\bar{\Delta}^2}$ and use  the inequality $(a-b)^2\geq \hat{c}a^2-\frac{\hat{c}}{1-\hat{c}}b^2$ to obtain  
 \begin{equation}
     \begin{aligned}\nonumber
         \bm{v}^\top \bm{Q}\bm{v} &= \frac{1}{n}\|\bm{\dot{X}v}\|_2^2 - \frac{\bar{\Delta}^2}{4}\|\bm{v}\|_2^2\\&\geq \hat{c}\big(\kappa_0+\frac{1}{4}\bar{\Delta}^2\big)\|\bm{v}\|_2^2- c_1^2\mathscr{L}_1^2\frac{\hat{c}}{1-\hat{c}}\frac{\delta \log d}{n}\|\bm{v}\|_1^2-\frac{\bar{\Delta}^2}{4}\|\bm{v}\|_2^2\\
         &\geq \frac{\kappa_0}{2}\|\bm{v}\|_2^2- c_1^2\mathscr{L}_1^2 \big(1+\frac{\bar{\Delta}^2}{2\kappa_0}\big)\frac{\delta \log d}{n}\|\bm{v}\|_1^2\\&:=\frac{\kappa_0}{2}\|\bm{v}\|_2^2-c_2(\kappa_0,\sigma,\bar{\Delta})\cdot \frac{\delta \log d}{n}\|\bm{v}\|_1^2,
     \end{aligned}
 \end{equation}
 which holds for all $\bm{v}\in \mathbb{R}^d$ and   $\hat{c}_2:=c_2(\kappa_0, \sigma,\bar{\Delta})$ is a constant depending on 
$\kappa_0, \sigma,\bar{\Delta}$ (we remove the dependence on $\kappa_1$ by $\kappa_1\lesssim \sigma^2$). 
\vspace{1mm}

\noindent{\bf \sffamily  Step 3.  Bounding the Estimation Error}

\vspace{1mm}
We are in a position to bound the estimation error of any $\bm{\widetilde{\theta}}$ satisfying (\ref{4.17}).
Note that definition of $\partial \|\bm{\widetilde{\theta}}\|$ gives $\lambda\|\bm{\theta^\star}\|_1-\lambda\|\bm{\widetilde{\theta}}\|_1\geq \big<\lambda\cdot\partial \|\bm{\widetilde{\theta}}\|_1, -\bm{\widetilde{\Upsilon}}\big>$. Thus, we set $\bm{\theta} = \bm{\theta^\star}$ in (\ref{4.17}) and proceed as follows 
\begin{equation}
    \begin{aligned}\label{4.21}
      0&\geq \big<\bm{Q\widetilde{\theta}} - \bm{b}+ \lambda\cdot \partial \|\bm{\widetilde{\theta}}\|_1, \bm{\widetilde{\Upsilon}}\big> \\&= \bm{\widetilde{\Upsilon}}^\top \bm{Q}\bm{\widetilde{\Upsilon}}+ \big<\bm{Q\theta^\star}-\bm{b},\bm{\widetilde{\Upsilon}}\big> + \big<\lambda \cdot \partial\|\bm{\widetilde{\theta}}\|_1, \bm{\widetilde{\Upsilon}}\big> \\
      &\stackrel{(i)}{\geq}   \frac{\kappa_0}{2}\|\bm{\widetilde{\Upsilon}}\|_2^2 - \frac{2\hat{c}_2R\sqrt{s}\cdot \delta \log d }{n} \|\bm{\widetilde{\Upsilon}}\|_1  - \frac{\lambda}{\beta}\|\bm{\widetilde{\Upsilon}}\|_1 + \lambda\big(\|\bm{\widetilde{\theta}}\|_1-\|\bm{\theta^\star}\|_1\big)\\&\stackrel{(ii)}{\geq} \frac{\kappa_0}{2}\|\bm{\widetilde{\Upsilon}}\|_2^2 - \frac{\lambda}{2}\|\bm{\widetilde{\Upsilon}}\|_1  + \lambda \big(\|\bm{\widetilde{\theta}}\|_1-\|\bm{\theta^\star}\|_1\big),
     \end{aligned}
\end{equation}
where we use $\lambda\geq \beta \|\bm{Q\theta^\star}-\bm{b}\|_\infty~(\beta>2)$  and $\|\bm{\widetilde{\Upsilon}}\|_1\leq \|\bm{\widetilde{\theta}}\|_1+\|\bm{\theta^\star}\|_1\leq 2R\sqrt{s}$ in $(i)$, and  $(ii)$ is due to the scaling $2\hat{c}_2R\delta \sqrt{s}\frac{\log d}{n} \leq (\frac{1}{2}-\frac{1}{\beta})\lambda$ that holds under the assumed $n\gtrsim \delta \log d$ for some hidden constant depending on $(\kappa_0,\sigma,\bar{\Delta},\Delta,M,R)$. Thus, we arrive at $\frac{\lambda}{2}\|\bm{\widetilde{\Upsilon}}\|_1 \geq \lambda\big(\|\bm{\widetilde{\theta}}\|_1-\|\bm{\theta^\star}\|_1\big)$.

For $\bm{a}\in \mathbb{R}^d$, $\mathcal{J}\subset[d]$ we obtain $\bm{a}_{\mathcal{J}}\in \mathbb{R}^d$ by keeping entries of $\bm{a}$  in $\mathcal{J}$ while setting others to zero. Let $\mathcal{A}$ be the support of $\bm{\theta^\star}$, $\mathcal{A}^c = [d]\setminus \mathcal{A}$, then we have \begin{equation}
    \begin{aligned}\label{decom}
        &\frac{1}{2}\|\bm{\widetilde{\Upsilon}}\|_1\geq \|\bm{\theta^\star}+\bm{\widetilde{\Upsilon}}\|_1-\|\bm{\theta^\star}\|_1 = \|\bm{\theta^\star} +\bm{\widetilde{\Upsilon}}_{\mathcal{A}}+\bm{\widetilde{\Upsilon}}_{\mathcal{A}^c}\|_1-\|\bm{\theta^\star}\|_1\\&\geq \|\bm{\theta^\star}\|_1   + \|\bm{\widetilde{\Upsilon}}_{\mathcal{A}^c}\|_1-\|\bm{\widetilde{\Upsilon}}_{\mathcal{A}}\|_1 -\|\bm{\theta^\star}\|_1 = \|\bm{\widetilde{\Upsilon}}_{\mathcal{A}^c}\|_1-\|\bm{\widetilde{\Upsilon}}_{\mathcal{A}}\|_1.
    \end{aligned}
\end{equation} Further use $\frac{1}{2}\|\bm{\widetilde{\Upsilon}}\|_1 \leq \frac{1}{2}\|\bm{\widetilde{\Upsilon}}_{\mathcal{A}}\|_1+\frac{1}{2}\|\bm{\widetilde{\Upsilon}}_{\mathcal{A}^c}\|_1$, we obtain $\|\bm{\widetilde{\Upsilon}}_{\mathcal{A}^c}\|_1\leq 3\|\bm{\widetilde{\Upsilon}}_{\mathcal{A}}\|_1$. Hence, we have $\|\bm{\widetilde{\Upsilon}}\|_1 \leq \|\bm{\widetilde{\Upsilon}}_{\mathcal{A}}\|_1+\|\bm{\widetilde{\Upsilon}}_{\mathcal{A}^c}\|_1\leq 4\|\bm{\widetilde{\Upsilon}}_{\mathcal{A}}\|_1\leq 4\sqrt{s}\|\bm{\widetilde{\Upsilon}}\|_2$. Now, we further substitute this into (\ref{4.21}) and obtain $$\frac{1}{2}\kappa_0\|\bm{\widetilde{\Upsilon}}\|^2_2\leq \frac{\lambda}{2}\|\bm{\widetilde{\Upsilon}}\|_1+\lambda \big(\|\bm{\theta^\star}\|_1- \|\bm{\widetilde{\theta}}\|_1\big) \leq \frac{3\lambda}{2}\|\bm{\widetilde{\Upsilon}}\|_1 \leq 6\lambda \sqrt{s}\|\bm{\widetilde{\Upsilon}}\|_2.$$
 Thus, we arrive at the desired error bound for $\ell_2$-norm\begin{equation}\nonumber
 \|\bm{\widetilde{\Upsilon}}\|_2\lesssim \mathscr{L}\sqrt{\frac{\delta s\log d}{n}},     ~\mathrm{with}~\mathscr{L}:= \frac{(\sigma+\bar{\Delta})^2(\Delta+M^{1/(2l)})}{\kappa_0^{3/2}}.
 \end{equation}
  We   simply use $\|\bm{\widetilde{\Upsilon}}\|_1\leq 4\sqrt{s}\|\bm{\widetilde{\Upsilon}}\|_2$ again  to establish the bound for $\|\bm{\widetilde{\Upsilon}}\|_1$. The proof is complete. \hfill $\square$ 
\subsubsection{Proof of Theorem \ref{thm7}}

\noindent{\it Proof.} The proof is divided into two  steps. Compared to the last proof, due to the heavy-tailedness of $\bm{x}_k$, the step of "verifying RSC" reduces to the simpler argument in (\ref{4.23}).

\noindent{\bf \sffamily  Step 1. Proving $\lambda \geq \beta\|\bm{Q\theta^\star} - \bm{b}\|_{\infty}$   for some pre-specified $\beta>2$}

Recall that $(\bm{Q},\bm{b})$ are constructed from the quantized data as $\bm{Q}=\frac{1}{n}\sum_{k=1}^n \bm{\dot{x}}_k\bm{\dot{x}}_k^\top-\frac{\bar{\Delta}}{4}\bm{I}_d$ and $\bm{b}=\frac{1}{n}\sum_{k=1}^n\dot{y}_k\bm{\dot{x}}_k$. Thus, our main aim in this step is to prove that $\lambda=C_1(R\sqrt{M}+\Delta^2+R\bar{\Delta}^2)\sqrt{\frac{\delta \log d}{n}}$ suffices to ensure $\lambda\geq \beta \|\frac{1}{n}\sum_{k=1}^n(\bm{\dot{x}}_k\bm{\dot{x}}_k^\top-\frac{\bar{\Delta}^2}{4}\bm{I}_d)\bm{\theta^\star}-\frac{1}{n}\sum_{k=1}^n\dot{y}_k\bm{\dot{x}}_k\|_\infty$ with the promised probability and any pre-specified $\beta>2$. We let $\bm{\dot{x}}_k=\bm{\widetilde{x}}_k +\bm{\tau}_k+ \bm{w}_k$, $\dot{y}_k= \widetilde{y}_k+\phi_k+\vartheta_k$ with quantization errors $\bm{w}_k$ and $\vartheta_k$. Analogously to (\ref{628add1}), we have $\mathbbm{E}[\dot{y}_k\bm{\dot{x}}_k]=\mathbbm{E}(\widetilde{y}_k\bm{\widetilde{x}}_k)$ and $\mathbbm{E}(\bm{\dot{x}}_k\bm{\dot{x}}_k^\top\bm{\theta^\star})= \frac{\bar{\Delta}^2}{4}\bm{\theta^\star}+ \mathbbm{E}(\bm{\widetilde{x}}_k\bm{\widetilde{x}}_k^\top\bm{\theta^\star}) .$  Thus, the term we want to bound can be first decomposed into two concentration terms ($I_1,I_3$) and one bias term ($I_2$):\begin{equation}
    \begin{aligned}\label{4.12}
    & \Big\|\Big(\frac{1}{n}\sum_{k=1}^n\bm{\dot{x}}_k\bm{\dot{x}}_k^\top-\frac{\bar{\Delta}^2}{4}\bm{I}_d\Big) \bm{\theta^\star} - \frac{1}{n}\sum_{k=1}^n\dot{y}_k\bm{\dot{x}}_k\Big\|_{\infty} \leq \Big\|\frac{1}{n}\sum_{k=1}^n\dot{y}_k\bm{\dot{x}}_k - \mathbbm{E}(\dot{y}_k\bm{\dot{x}}_k)\Big\|_\infty \\
    & + \Big\|\mathbbm{E}\big(\bm{\widetilde{x}}_k\bm{\widetilde{x}}_k^\top \bm{\theta^\star}\big)-\mathbbm{E}\big(\widetilde{y}_k\bm{\widetilde{x}}_k\big)\Big\|_\infty + \Big\|\Big(\frac{1}{n}\sum_{k=1}^n\bm{\dot{x}}_k\bm{\dot{x}}_k^\top- \mathbbm{E}(\bm{\dot{x}}_k\bm{\dot{x}}_k^\top)\Big)\bm{\theta^\star}\Big\|_\infty:= I_1+I_2+I_3,
    \end{aligned}
\end{equation}

\noindent{\bf \sffamily Step 1.1. Bounding $I_1$}

 Denote the $i$-th entry of $\bm{x}_k,\bm{\widetilde{x}}_k,\bm{\dot{x}}_k,\bm{\tau}_k,\bm{w}_k$ by $x_{ki},\widetilde{x}_{ki},\dot{x}_{ki},\tau_{ki},w_{ki}$, respectively.
  Since $|\dot{y}_k|\leq |\widetilde{y}_k|+|\phi_k|+|\vartheta_k|\leq |\widetilde{y}_k|+\Delta$, $\dot{x}_{ki}\leq |\widetilde{x}_{ki}|+|\tau_{ki}|+|w_{ki}|\leq|\widetilde{x}_{ki}|+\frac{3}{2}\bar{\Delta}$, for any integer $q\geq 2$ we can bound the moments as \begin{equation}
    \begin{aligned}
\nonumber\sum_{k=1}^n\mathbbm{E}\Big|\frac{\dot{y}_k\dot{x}_{ki}}{n}\Big|^q &\leq \frac{(\zeta_x+\frac{3}{2}\bar{\Delta})^{q-2}(\zeta_y+\Delta)^{q-2}}{n^q}\sum_{k=1}^n\mathbbm{E}|\dot{y}_k\dot{x}_{ki}|^2\\
    &\leq \frac{[(\zeta_x+\frac{3}{2}\bar{\Delta}) (\zeta_y+\Delta)]^{q-2}}{n^q}\sum_{k=1}^n\sqrt{\mathbbm{E}|\dot{y}_k|^4\mathbbm{E}|\dot{x}_{ki}|^4}\\&\lesssim \Big(\frac{(\zeta_x+\frac{3}{2}\bar{\Delta})(\zeta_y+\Delta)}{n}\Big)^{q-2}\Big(\frac{M+ \Delta^4+ \bar{\Delta}^4}{n}\Big),
    \end{aligned}    
\end{equation}
and in the last inequality we use $\mathbbm{E}|\widetilde{y}_k|\leq \mathbbm{E}|y_k|^4\leq M$ and $\mathbbm{E}|\widetilde{x}_{ki}|^4 \leq \mathbbm{E}|x_{ki}|^4\leq M$.
With these preparations, we apply Bernstein's inequality (Lemma \ref{bernstein}) and a union bound, yielding that \begin{equation}\nonumber
    \mathbbm{P}\left(I_1\geq C_5\left\{\sqrt{\frac{(M+ \Delta^4+ \bar{\Delta}^4)t}{n}}+\frac{(\zeta_x+\frac{3}{2}\bar{\Delta})(\zeta_y+\Delta)t}{n}\right\}\right)\leq 2d\cdot\exp(-t).
\end{equation}
Set $t=\delta \log d$, $I_1 \lesssim (\sqrt{M}+\Delta^2+\bar{\Delta}^2)\sqrt{\frac{\delta \log d}{n}}$ holds with probability at least $1-2d^{1-\delta}$.

\noindent{\bf \sffamily Step 1.2. Bounding $I_2$}

Noting that $\mathbbm{E}(y_k\bm{x}_k)=\mathbbm{E}(\bm{x}_k\bm{x}_k^\top\bm{\theta^\star})+\mathbbm{E}(\epsilon_k\bm{x}_k)=\mathbbm{E}(\bm{x}_k\bm{x}_k^\top \bm{\theta^\star})$, we could further decompose $I_2$ as 
$I_2\leq \|\mathbbm{E}(\bm{\widetilde{x}}_k\bm{\widetilde{x}}_k^\top\bm{\theta^\star})-\mathbbm{E}(\bm{x}_k\bm{x}_k^\top\bm{\theta^\star})\|_\infty + \|\mathbbm{E}(y_k\bm{x}_k)-\mathbbm{E}(\widetilde{y}_k\bm{\widetilde{x}}_k)\|_\infty:= I_{21}+I_{22}$. To bound $I_{21}$, we note that the assumption and truncation procedure for $\bm{x}_k$ are   the same as in Theorem \ref{thm1}; thus, Step 2 in the proof of Theorem \ref{thm1} can yield that $\|\mathbbm{E}(\bm{\widetilde{x}}_k\bm{\widetilde{x}}_k^\top-\bm{x}_k\bm{x}_k^\top)\|_\infty\leq \sqrt{\frac{\delta M\log d}{n}}$. Thus, we have $I_{21}\leq \|\mathbbm{E}(\bm{\widetilde{x}}_k\bm{\widetilde{x}}_k^\top-\bm{x}_k\bm{x}_k^\top)\|_\infty\|\bm{\theta^\star}\|_1\leq R\sqrt{M}\sqrt{\frac{\delta \log d}{n}}$. To bound $I_{22}$,
 we estimate the $i$-th entry 
\begin{equation}
    \begin{aligned}\nonumber
    \big|\mathbbm{E}(y_kx_{ki}-\widetilde{y}_k\widetilde{x}_{ki})\big|&= \big|\mathbbm{E}(y_kx_{ki}-\widetilde{y}_k\widetilde{x}_{ki})\big(\mathbbm{1}(|y_k|>\zeta_y)+\mathbbm{1}(|x_{ki}|\geq \zeta_x)\big)\big|\\
    &\leq  \mathbbm{E}\big(|y_kx_{ki}|\mathbbm{1}(|y_k|>\zeta_y)\big) + \mathbbm{E}\big(|y_kx_{ki}|\mathbbm{1}(|x_{ki}|\geq \zeta_x)\big)\\&\stackrel{(i)}{\leq} M\Big(\frac{1}{\zeta_x^2}+\frac{1}{\zeta_y^2}\Big)\lesssim \sqrt{\frac{\delta M \log d}{n}},
    \end{aligned}
\end{equation}
where $(i)$ is because $\mathbbm{E}\big(|y_kx_{ki}|\mathbbm{1}(|y_k|>\zeta_y)\big) \leq \big[\mathbbm{E}|y_k^2x_{ki}^2|\big]^{1/2}\sqrt{\mathbbm{P}(|y_k|>\zeta_y)} \leq (\mathbbm{E}y_k^4)^{1/4}(\mathbbm{E}x_{ki}^4)^{1/4}\sqrt{\frac{\mathbbm{E}y_k^4}{\zeta_y^4}}\leq \frac{M}{\zeta_y^2}$, 
and applying similar treatment   to $\mathbbm{E}\big(|y_kx_{ki}|\mathbbm{1}(|x_{ki}|>\zeta_x)\big)$. Overall, we have $I_2\lesssim R\sqrt{M}\sqrt{\frac{\delta \log d}{n}}$.

\vspace{1mm}

\noindent{\bf \sffamily Step 1.3. Bounding $I_3$}

We first note that $$I_3 = \|(\bm{Q}- \bm{\Sigma^\star})\bm{\theta^\star}\|_\infty\leq \|\bm{Q}- \bm{\Sigma^\star}\|_\infty\cdot\|\bm{\theta^\star}\|_1\leq R\|\bm{Q}- \bm{\Sigma^\star}\|_\infty.$$ By Theorem \ref{thm1}, $\|\bm{Q}-\bm{\Sigma^\star}\|_\infty\lesssim (\sqrt{M}+\bar{\Delta}^2)\sqrt{\frac{\delta\log d}{n}}$ holds with probability at least $1-2d^{2-\delta}$, which leads to $I_3 \leq \|\bm{Q}-\bm{\Sigma^\star}\|_{\infty}\|\bm{\theta^\star}\|_1\lesssim  R(\sqrt{M}+\bar{\Delta}^2)\sqrt{\frac{\delta   \log d}{n}}$. Thus, by combining everything, we obtain that $\|\bm{Q\theta^\star}-\bm{b}\|_\infty \lesssim (R\sqrt{M}+\Delta^2+R\bar{\Delta}^2)\sqrt{\frac{\delta \log d}{n}}$ holds with probability at least $1-4d^{2-\delta}$.
 Compared to our choice of $\lambda$, the claim of this step follows.

\vspace{1mm}

\vspace{1mm}

\noindent{\bf \sffamily  Step 2.   Bounding the Estimation Error}

We are now ready to bound the estimation error of any $\bm{\widetilde{\theta}}$ satisfying (\ref{4.17}). Set $\bm{\theta}=\bm{\theta^\star}$ in (\ref{4.12}), it gives $\big<\bm{Q\widetilde{\theta}}-\bm{b} +\lambda\cdot \partial \|\bm{\widetilde{\theta}}\|_1, \bm{\widetilde{\Upsilon}}\big>\leq 0$. Recall that we can assume $\|\bm{Q}-\bm{\Sigma^\star}\|_\infty \leq C_6(\sqrt{M}+\bar{\Delta}^2) \sqrt{\frac{\delta \log d}{n}}$ with probability at least $1-2d^{2-\delta}$, which leads to \begin{equation}\label{4.23}
   \begin{aligned}
        &\bm{\widetilde{\Upsilon}}^\top \bm{Q} \bm{\widetilde{\Upsilon}} = \bm{\widetilde{\Upsilon}}^\top \bm{\Sigma^\star} \bm{\widetilde{\Upsilon}} - \bm{\widetilde{\Upsilon}}^\top (\bm{\Sigma^\star}-\bm{Q}) \bm{\widetilde{\Upsilon}}\\& \geq \kappa_0\|\bm{\widetilde{\Upsilon}}\|_2^2 -C_6(\sqrt{M}+\bar{\Delta}^2)\sqrt{\frac{\delta \log d}{n}}\|\bm{\widetilde{\Upsilon}}\|_1^2.
   \end{aligned}
\end{equation} Thus, it follows that \begin{equation}
    \begin{aligned}\label{4.24}
    0&\geq \big<\bm{Q\widetilde{\theta}}-\bm{b}+\lambda \cdot \partial\|\bm{\widetilde{\theta}}\|_1,\bm{\widetilde{\Upsilon}}\big>\\&= \big<\bm{Q\theta^\star}-\bm{b}, \bm{\widetilde{\Upsilon}}\big> + \bm{\widetilde{\Upsilon}}^\top\bm{Q}\bm{\widetilde{\Upsilon}} + \lambda \big<\partial\|\bm{\widetilde{\theta}}\|_1,\bm{\widetilde{\Upsilon}}\big>\\
    &\stackrel{(i)}{\geq} C_0\sqrt{M}\|\bm{\widetilde{\Upsilon}}\|_2^2 -C_6(\sqrt{M}+\bar{\Delta}^2)\sqrt{\frac{\delta \log d}{n}}\|\bm{\widetilde{\Upsilon}}\|_1^2 \\&~~~~~~~~~~~~~~~~~~~~~~~ - \|\bm{Q\theta^\star}-\bm{b}\|_\infty\|\bm{\widetilde{\Upsilon}}\|_1 + \lambda\big(\|\bm{\widetilde{\theta}}\|_1-\|\bm{\theta^\star}\|_1\big)\\
    &\stackrel{(ii)}{\geq} C_0\sqrt{M}\|\bm{\widetilde{\Upsilon}}\|_2^2 -\Big( 2C_6R(\sqrt{M}+\bar{\Delta}^2)\sqrt{\frac{\delta \log d}{n}}+\|\bm{Q\theta^\star}-\bm{b}\|_\infty\Big) \|\bm{\widetilde{\Upsilon}}\|_1 \\&~~~~~~~~~~~~~~~~~~~~~~~+ \lambda\big(\|\bm{\widetilde{\theta}}\|_1-\|\bm{\theta^\star}\|_1\big) \\
    &\stackrel{(iii)}{\geq} C_0\sqrt{M}\|\bm{\widetilde{\Upsilon}}\|_2^2 - \frac{\lambda}{2}\|\bm{\widetilde{\Upsilon}}\|_1 +\lambda\big(\|\bm{\widetilde{\theta}}\|_1-\|\bm{\theta^\star}\|_1\big).
    \end{aligned}
\end{equation}
Note that $(i)$ is due to (\ref{4.23}) and $\|\bm{\theta^\star}\|_1-\|\bm{\widetilde{\theta}}\|_1\geq \big<\partial\|\bm{\widetilde{\theta}}\|_1, -\bm{\widetilde{\Upsilon}}\big>$, in $(ii)$ we use $\|\bm{\widetilde{\Upsilon}}\|_1\leq \|\bm{\widetilde{\theta}}\|_1+\|\bm{\theta^\star}\|_1\leq 2R$, and   from { Step 1} $(iii)$   holds when $\lambda =C_2 (R\sqrt{M}+\Delta^2+R\bar{\Delta}^2)\sqrt{\frac{\delta \log d}{n}}$ with sufficiently large $C_2$. Therefore, we arrive at $\frac{1}{2}\|\bm{\widetilde{\Upsilon}}\|_1 \geq \|\bm{\widetilde{\theta}}\|_1-\|\bm{\theta^\star}\|_1$. Similar to {Step 3} in the proof of Theorem \ref{thm6}, we can show $\|\bm{\widetilde{\Upsilon}}\|_1 \leq 4\sqrt{s}\|\bm{\widetilde{\Upsilon}}\|_2$. Applying (\ref{4.24}) again, it yields $$\kappa_0\|\bm{\widetilde{\Upsilon}}\|_2^2 \leq \frac{\lambda}{2}\|\bm{\widetilde{\Upsilon}}\|_1+ \lambda \|\bm{\widetilde{\Upsilon}}\|_1 \leq \frac{3\lambda}{2}\|\bm{\widetilde{\Upsilon}}\|_1\leq 6\sqrt{s}\lambda \|\bm{\widetilde{\Upsilon}}\|_2.$$ Thus, we obtain $\|\bm{\widetilde{\Upsilon}}\|_2\lesssim \mathscr{L}\sqrt{\frac{\delta s\log d}{n}}$ with $\mathscr{L}:=\frac{ R\sqrt{M}+\Delta^2+R\bar{\Delta}^2}{\kappa_0}$. The proof can be concluded by further applying $\|\bm{\widetilde{\Upsilon}}\|_1\leq 4\sqrt{s}\|\bm{\widetilde{\Upsilon}}\|_2$ to derive the bound for $\ell_1$-norm. \hfill $\square$
\subsubsection{Proof of Proposition \ref{framework}}
\begin{proof}
     We let $\bm{\theta}=\bm{\theta^\star}$ in (\ref{4.17}), then proceeds as the proof of Theorem \ref{thm7}: \begin{equation}
    \begin{aligned}\label{A.1}
      0&\geq \big<\bm{Q\widetilde{\theta}} -\bm{b}+\lambda \cdot \partial \|\bm{\widetilde{\theta}}\|_1,\bm{\widetilde{\Upsilon}}\big> \\&= \bm{\widetilde{\Upsilon}}^\top \bm{\Sigma^\star}\bm{\widetilde{\Upsilon}} + \big<\bm{Q\theta^\star}-\bm{b}, \bm{\widetilde{\Upsilon}}\big> + \bm{\widetilde{\Upsilon}}^\top(\bm{Q}-\bm{\Sigma^\star})\bm{\widetilde{\Upsilon}} + \lambda \big<\partial\|\bm{\widetilde{\theta}}\|_1,\bm{\widetilde{\Upsilon}}\big> \\
      &\stackrel{(i)}{\geq} \kappa_0 \|\bm{\widetilde{\Upsilon}}\|_2^2 - \|\bm{Q\theta^\star}-\bm{b}\|_\infty \|\bm{\widetilde{\Upsilon}}\|_1 -  \|\bm{Q}-\bm{\Sigma^\star}\|_\infty\|\bm{\widetilde{\Upsilon}}\|_1^2 +\lambda\big(\|\bm{\widetilde{\theta}}\|_1-\|\bm{\theta^\star}\|_1\big)\\
      &\stackrel{(ii)}{\geq} \kappa_0  \|\bm{\widetilde{\Upsilon}}\|_2^2  - \big(\|\bm{Q\theta^\star}-\bm{b}\|_\infty+2R\cdot\|\bm{Q}-\bm{\Sigma^\star}\|_\infty \big)\|\bm{\widetilde{\Upsilon}}\|_1+ \lambda\big(\|\bm{\widetilde{\theta}}\|_1-\|\bm{\theta^\star}\|_1\big) \\
      &\stackrel{(iii)}{\geq}  \kappa_0  \|\bm{\widetilde{\Upsilon}}\|_2^2  - \frac{\lambda}{2}\|\bm{\widetilde{\Upsilon}}\|_1+ \lambda\big(\|\bm{\widetilde{\theta}}\|_1-\|\bm{\theta^\star}\|_1\big),
    \end{aligned}
\end{equation}
where in $(i)$ we use $\lambda_{\min}(\bm{\Sigma^\star}) \geq \kappa_0$ and $\|\bm{\theta^\star}\|_1-\|\bm{\widetilde{\theta}}\|_1\geq \big<\partial\|\bm{\widetilde{\theta}}\|_1, -\bm{\widetilde{\Upsilon}}\big>$, $(ii)$ is by $\|\bm{\widetilde{\Upsilon}}\|_1\leq\|\bm{\widetilde{\theta}}\|_1+\|\bm{\theta^\star}\|_1\leq 2R$, in $(iii)$ we use the assumption (\ref{4.26}). Thus, by $\kappa_0\|\bm{\widetilde{\Upsilon}}\|^2_2\geq 0$ we obtain $2\big(\|\bm{\widetilde{\theta}}\|_1-\|\bm{\theta^\star}\|_1\big)\leq \|\bm{\widetilde{\Upsilon}}\|_1$. Similarly to  {Step 3} in the proof of Theorem \ref{thm6}, we can show $\|\bm{\widetilde{\Upsilon}}\|_1 \leq 4\sqrt{s}\|\bm{\widetilde{\Upsilon}}\|_2$. Again we use (\ref{A.1}), it gives \begin{equation}
    \begin{aligned}\nonumber
      \kappa_0\|\bm{\widetilde{\Upsilon}}\|^2_2\leq \frac{\lambda}{2}\|\bm{\widetilde{\Upsilon}}\|_1+ \lambda \big(\|\bm{\theta^\star}\|_1 - \|\bm{\widetilde{\theta}}\|_1\big) \leq \frac{3}{2}\lambda \|\bm{\widetilde{\Upsilon}}\|_1 \leq 6\lambda\sqrt{s}\|\bm{\widetilde{\Upsilon}}\|_2,
    \end{aligned}
\end{equation}  
thus displaying $\|\bm{\widetilde{\Upsilon}}\|_2 \leq \frac{6\lambda\sqrt{s}}{\kappa_0}$. The proof can be concluded by using $\|\bm{\widetilde{\Upsilon}}\|_1 \leq 4\sqrt{s}\|\bm{\widetilde{\Upsilon}}\|_2$.
\end{proof}
\subsubsection{Proof of Theorem \ref{sg1bitqccs}}
 
\begin{proof}
To invoke Proposition \ref{framework}, it is enough to verify (\ref{4.26}). Recalling $\bm{\Sigma^\star} = \mathbbm{E}(\bm{x}_k\bm{x}_k^\top)$, we first invoke \cite[Thm. 1]{chen2022high} and obtain $\|\bm{Q}-\bm{\Sigma^\star}\|_\infty =O\big(\sigma^2\sqrt{\frac{\delta \log d(\log n)^2}{n}}\big)$ holds with probability at least $1-2d^{2-\delta}$. This confirms $\lambda\gtrsim R\cdot\|\bm{Q}-\bm{\Sigma^\star}\|_\infty$. Then, it remains to upper bound $\|\bm{Q\theta^\star}-\bm{b}\|_\infty$:\begin{equation}
    \begin{aligned}\nonumber
      &\|\bm{Q\theta^\star}-\bm{b}\|_\infty \leq \|(\bm{Q}-\bm{\Sigma^\star})\bm{\theta^\star}\|_\infty +\|\bm{\Sigma^\star\theta^\star}-\bm{b}\|_\infty \\
      &\leq \|\bm{Q}-\bm{\Sigma^\star}\|_\infty\|\bm{\theta^\star}\|_1 + \Big\|\frac{\gamma_x\gamma_y}{n}\sum_{k=1}^n\dot{y}_k\bm{\dot{x}}_{k1} -\mathbbm{E}(y_k\bm{x}_k)\Big\|_\infty\\&\stackrel{(i)}{\lesssim} \sigma^2(R+1)\sqrt{\frac{\delta \log d(\log n)^2}{n}},
    \end{aligned}
\end{equation}
where in $(i)$ we use a known estimate from   \cite[Eq. (A.31)]{chen2022high}:
$$\mathbbm{P}\left(\Big\|\frac{\gamma_x\gamma_y}{n}\sum_{k=1}^n\dot{y}_k\bm{\dot{x}}_{k1}- \mathbbm{E}(y_k\bm{x}_k)\Big\|_\infty\lesssim \sigma^2\sqrt{\frac{\delta \log d(\log n)^2}{n}}\right)\geq 1-2d^{1-\delta}.$$  Thus, by setting $\lambda = C_1\sigma^2R\sqrt{\frac{\delta \log d(\log n)^2}{n}}$, (\ref{4.26}) can be satisfied with probability at least $1-4d^{2-\delta}$, hence using Proposition \ref{framework} concludes the proof.\end{proof}

\subsubsection{Proof of Theorem \ref{ht1bitqccs}}
\begin{proof}The proof is again based on Proposition \ref{framework} and some ingredients from \cite{chen2022high}. From  \cite[Thm. 4]{chen2022high}, $\|\bm{Q}-\bm{\Sigma^\star}\|_\infty \lesssim\big(\frac{M^2\delta \log d}{n}\big)^{1/4}$ holds with probability at least $1-2d^{2-\delta}$, thus confirming $\lambda \gtrsim R\cdot\|\bm{Q}-\bm{\Sigma^\star}\|_\infty$ with the same probability. Moreover, \begin{equation}
    \begin{aligned}\nonumber
      &\|\bm{Q\theta^\star}-\bm{b}\|_\infty \leq \|(\bm{Q}-\bm{\Sigma^\star})\bm{\theta^\star}\|_\infty +\|\bm{\Sigma^\star\theta^\star}-\bm{b}\|_\infty \\
      &\leq \|\bm{Q}-\bm{\Sigma^\star}\|_\infty\|\bm{\theta^\star}\|_1 + \Big\|\frac{\gamma_x\gamma_y}{n}\sum_{k=1}^n\dot{y}_k\bm{\dot{x}}_{k1} -\mathbbm{E}(y_k\bm{x}_k)\Big\|_\infty\\&\stackrel{(i)}{\lesssim} \sqrt{M}(R+1)\Big(\frac{\delta \log d}{n}\Big)^{1/4},
    \end{aligned}
\end{equation}
where $(i)$ is due to a known estimate from   \cite[Eq. (A.34)]{chen2022high}: $$\mathbbm{P}\left(\Big\|\frac{\gamma_x\gamma_y}{n} \sum_{k=1}^n\dot{y}_k\bm{\dot{x}}_{k1}-\mathbbm{E}(y_k\bm{x}_k)\Big\|_\infty \lesssim\sqrt{M}\Big(\frac{\delta \log d}{n}\Big)^{1/4}\right)\geq 1-2d^{1-\delta}$$
Thus, with probability at least $1-4d^{2-\delta}$, (\ref{4.26}) holds if   $\lambda = C_1\sqrt{M}R\big(\frac{\delta \log d}{n}\big)^{1/4}$ with sufficiently large $C_1$. The proof can be concluded by invoking Proposition \ref{framework}. \end{proof}
 \subsection{Uniform Recovery Guarantee}
 We need some auxiliary results to support the proof. The first one is a concentration inequality for product process due to Mendelson \cite{mendelson2016upper}; the following statement can be directly adapted from  \cite[Thm. 8]{genzel2022unified} by specifying the pseudo-metrics as $\ell_2$-distance.
 \begin{lem}
    \label{productpro}
{\rm  (Concentration of   Product Process){\bf \sffamily.}} Let $\{g_{\bm{a}}\}_{\bm{a}\in \mathcal{A}}$ and $\{h_{\bm{b}}\}_{\bm{b}\in\mathcal{B}}$ be stochastic processes indexed by two sets $\mathcal{A}\subset \mathbb{R}^{p}$, $\mathcal{B}\subset \mathbb{R}^q$, both defined on a common probability space $(\Omega,\textsc{A},\mathbbm{P})$. We assume that there exist $K_{\mathcal{A}},K_{\mathcal{B}},r_{\mathcal{A}},r_{\mathcal{B}}\geq 0$ such that 
\begin{equation}
    \begin{aligned}\nonumber
       & \|g_{\bm{a}}-g_{\bm{a}'}\|_{\psi_2}\leq K_{\mathcal{A}}\|\bm{a}-\bm{a}'\|_2,~\|g_{\bm{a}}\|_{\psi_2} \leq r_{\mathcal{A}},~\forall~\bm{a},\bm{a}'\in \mathcal{A};\\
       & \|h_{\bm{b}}-h_{\bm{b}'}\|_{\psi_2} \leq {K}_{\mathcal{B}}\|\bm{b}-\bm{b}'\|_2,~\|h_{\bm{b}}\|_{\psi_2}\leq r_{\mathcal{B}},~\forall~\bm{b},\bm{b}'\in \mathcal{B}.
    \end{aligned}
\end{equation}
Finally, let $X_1,...,X_m$ be independent copies of a random variable $X\sim \mathbbm{P}$, then for every $u\geq 1$ the following holds with probability at least $1-2\exp(-cu^2)$ \begin{equation}
   \begin{aligned}\nonumber
        &\sup_{\substack{\bm{a}\in \mathcal{A}\\\bm{b}\in \mathcal{B}}} ~ \frac{1}{n}\left|\sum_{i=1}^n g_{\bm{a}}(X_i)h_{\bm{b}}(X_i)-\mathbbm{E}\big[g_{\bm{a}}(X_i)h_{\bm{b}}(X_i)\big]\right|\\
        &~~~\leq C\Big(\frac{(K_{\mathcal{A}}\cdot\omega(\mathcal{A})+u\cdot r_{\mathcal{A}})\cdot (K_{\mathcal{B}}\cdot \omega(\mathcal{B})+u\cdot r_{\mathcal{B}})}{n}\\&~~~~~~+\frac{r_{\mathcal{A}}\cdot K_{\mathcal{B}}\cdot \omega(\mathcal{B})+r_{\mathcal{B}}\cdot K_{\mathcal{A}}\cdot \omega(\mathcal{A})+u\cdot r_{\mathcal{A}}r_{\mathcal{B}}}{\sqrt{n}}\Big),
   \end{aligned}
\end{equation}
where $\omega(\mathcal{A})= \mathbbm{E}\sup_{\bm{a}\in\mathcal{A}}(\bm{g}^\top\bm{a})$ with $\bm{g}\sim \mathcal{N}(0,\bm{I}_p)$ is the Gaussian width of $\mathcal{A}\subset \mathbb{R}^p$, and similarly, $\omega(\mathcal{B})$ is the Gaussian width of $\mathcal{B}$.
 \end{lem}

We will use the following result that can be found in \cite[Thm. 8]{liaw2017simple}. 

\begin{lem}\label{talagrand}
   Let $(X_{\bm{u}})_{\bm{u}\in\mathcal{T}}$ be a random process indexed by points in a bounded set $\mathcal{T}\subset \mathbb{R}^n$. Assume that the process has sub-Gaussian increments, i.e., there exists $M>0$ such that $
        \|X_{\bm{u}}-X_{\bm{v}}\|_{\psi_2}\leq M\|\bm{u}-\bm{v}\|_2$ holds for any $\bm{u},\bm{v}\in \mathcal{T}$.
    Then for every $t>0$, the event \begin{equation}\nonumber
        \sup_{\bm{u},\bm{v}\in\mathcal{T}}|X_{\bm{u}}-X_{\bm{v}}|\leq CM\cdot \big(\omega(\mathcal{T})+t\cdot \mathrm{diam}(\mathcal{T})\big)
    \end{equation}
    holds with probability at least $1-\exp(-t^2)$, where $\mathrm{diam}(\mathcal{T}):=\sup_{\bm{x},\bm{y}\in\mathcal{T}}\|\bm{x}-\bm{y}\|_2$ denotes the diameter of $\mathcal{T}$.
\end{lem}


 \subsubsection{Proof of Theorem \ref{uniformtheorem}}

 \begin{proof}
 We start from the optimality $\sum_{k=1}^n (\dot{y}_k-\bm{x}_k^\top\bm{\widehat{\theta}})^2\leq\sum_{k=1}^n (\dot{y}_k-\bm{x}_k^\top\bm{{\theta^\star}})^2.$ By substituting $\bm{\widehat{\theta}}=\bm{\theta^\star}+\bm{\widehat{\Upsilon}}$ and performing some algebra, we obtain $$\sum_{k=1}^n (\bm{x}_k^\top \bm{\widehat{\Upsilon}})^2 \leq 2 \sum_{k=1}^n \big(\dot{y}_k - \bm{x}_k^\top \bm{\theta^\star}\big)\bm{x}_k^\top\bm{\widehat{\Upsilon}}.$$ Due to the constraint we have $\|\bm{\theta^\star}+\bm{\widehat{\Upsilon}}\|_{1}\leq \|\bm{\theta^\star}\|_1$, then similar to (\ref{decom}) we can show $\|\bm{\widehat{\Upsilon}}\|_1\leq 2\sqrt{s}\|\bm{\widehat{\Upsilon}}\|_2$ holds. Thus, we let $\mathscr{V}= \{\bm{v}:\|\bm{v}\|_2=1,\|\bm{v}\|_1\leq 2\sqrt{s}\}$, then the following holds uniformly for all $\bm{\theta^\star}\in \Sigma_{s,R_0}$   
\begin{equation}\label{frame}
    \|\bm{\widehat{\Upsilon}}\|_2^2 \cdot \inf_{\bm{v}\in\mathscr{V}}\sum_{k=1}^n(\bm{x}_k^\top\bm{v})^2 \leq 2\|\bm{\widehat{\Upsilon}}\|_2 \cdot \sup_{\bm{v}\in\mathscr{V}}\sum_{k=1}^n(\dot{y}_k - \bm{x}_k^\top\bm{\theta^\star})\bm{x}_k^\top\bm{v}
\end{equation}
Similarly to previous developments, our strategy is to lower bound the left-hand side while upper bound the right hand side, but with the difference  that the bounds must be valid uniformly for all $\bm{\theta^\star}\in \Sigma_{s,R_0}$.

\vspace{1mm}

\noindent{\bf \sffamily  Step 1. Bounding the Left-Hand Side From Below}

Letting $\bar{\Delta}=0$ and restricting $\bm{v}$ to $\mathscr{V}$, we use (\ref{nB.7}) in the proof of Theorem \ref{thm6}, then for some constant $c(\kappa_0,\sigma)$ depending on $\kappa_0,\sigma$, with probability at least $1-d^{-\delta}$ 
\begin{equation}\nonumber
   \inf_{\bm{v}\in \mathscr{V}} \frac{1}{\sqrt{n}}\left[\sum_{k=1}^n(\bm{x}_k^\top\bm{v})^2\right]^{1/2} \geq \sqrt{\kappa_0}  - c(\kappa_0,\sigma)\cdot \sqrt{\frac{s\log d}{n}}.
\end{equation}
Thus, if $n\geq \frac{4c^2(\kappa_0,\sigma)}{\kappa_0}s\log d$, then it holds that $\inf_{\bm{v}\in \mathscr{V}} \sum_{k=1}^n (\bm{x}_k^\top\bm{v})^2\geq \frac{1}{4}\kappa_0 n$. 

\vspace{1mm}

\noindent{\bf \sffamily  Step 2.      Bounding the Right-Hand Side Uniformly}

To pursue the uniformity over $\bm{\theta^\star}\in \Sigma_{s,R_0}$, we  take a supremum by replacing specific $\bm{\theta^\star}$ with $\sup_{\bm{\theta}\in \Sigma_{s,R_0}}$, then we consider the upper bound on 
\begin{equation}
    \label{defineI}
    I:=\sup_{\bm{\theta}\in \Sigma_{s,R_0}}\sup_{\bm{v}\in \mathscr{V}}~\sum_{k=1}^n (\dot{y}_k - \bm{x}_k^\top\bm{\theta})\bm{x}_k^\top\bm{v},
\end{equation}
where $\dot{y}_k = \mathcal{Q}_\Delta(\widetilde{y}_k + \tau_k),~\widetilde{y}_k = \mathscr{T}_{\zeta_y}(y_k):= \sign(y_k)\min\{|y_k|,\zeta_y\},~y_k=\bm{x}_k^\top\bm{\theta}+\epsilon_k$; note that $\dot{y}_k,\widetilde{y}_k,y_k$ depend on $\bm{\theta}$, and we will use notation  $\dot{y}_{\bm{\theta},k},\widetilde{y}_{\bm{\theta},k},y_{\bm{\theta},k}$ to indicate such dependence when necessary. In this proof, the ranges of $\bm{\theta}$ and $\bm{v}$ (e.g., in  supremum), if omitted, are respectively $\bm{\theta}\in \Sigma_{s,R_0}$ and $\bm{v}\in \mathscr{V}$.  Now let the quantization noise be $\xi_k= \dot{y}_k- \widetilde{y}_k$, observing that $\mathbbm{E}(y_k\bm{x}_k^\top\bm{v})=\mathbbm{E}(\bm{\theta}^\top\bm{x}_k\bm{x}_k^\top\bm{v})+\mathbbm{E}(\epsilon_k\bm{x}_k^\top\bm{v})=\mathbbm{E}(\bm{\theta}^\top\bm{x}_k\bm{x}_k^\top\bm{v})$, then 
we can first decompose $I$ as  
\begin{equation}
    \begin{aligned}\label{B.29}
   I&\leq \sup_{\bm{\theta},\bm{v}}  ~\sum_{k=1}^n \xi_k \bm{x}_k^\top\bm{v}+\sup_{\bm{\theta},\bm{v}}  ~\sum_{k=1}^n \big(\widetilde{y}_k\bm{x}_k^\top\bm{v}-\mathbbm{E}[\widetilde{y}_k\bm{x}_k^\top\bm{v}]\big) \\&+ \sup_{\bm{\theta},\bm{v}}~\sum_{k=1}^n\mathbbm{E}\big((\widetilde{y}_k-y_k)\bm{x}_k^\top\bm{v}\big)+\sup_{\bm{\theta},\bm{v}}~\sum_{k=1}^n \big(\bm{\theta}^\top\bm{x}_k\bm{x}_k^\top\bm{v}-\mathbbm{E}[\bm{\theta}^\top\bm{x}_k\bm{x}_k^\top\bm{v}]\big)\\&:= {I}_0 + {I}_1 + {I}_2+{I}_3,
    \end{aligned}
\end{equation}
where $I_0$ is the term arising from quantization, $I_1$ is the concentration term involving truncation of heavy-tailed data for which we develop some new machinery to bound it, $I_2$ is the bias term, $I_3$ is a more regular concentration term that can be bounded via Lemma \ref{productpro}. In the remainder of the proof, we will bound $I_1,I_2,I_3$ separately and finally deal with $I_0$.

\noindent{\bf \sffamily Step 2.1. Bounding ${I}_1$}

Using $\widetilde{y}_k=\mathscr{T}_{\zeta_y}(\bm{x}_k^\top\bm{\theta}+\epsilon_k)$, $I_1$ is concerned with the concentration of the product process $\big\{\sum_{k=1}^n\mathscr{T}_{\zeta_y}(\bm{\theta}^\top\bm{x}_k+\epsilon_k)\bm{x}_k^\top\bm{v}\big\}_{\bm{\theta},\bm{v}}$ about its mean. It is natural to apply Lemma \ref{productpro} towards this end, but we lack  good bound on $\|\mathscr{T}_{\zeta_y}(\bm{\theta}^\top\bm{x}_k+\epsilon_k)\|_{\psi_2}$ because of the heavy-tailedness of $\epsilon_k$ (on the other hand, the bound $O(\zeta_y)$ is just too crude to yield a sharp rate). Our strategy is already introduced in the mainbody --- we introduce $\widetilde{z}_k: = \widetilde{y}_k - \mathscr{T}_{\zeta_y}(\epsilon_k)$ and decompose ${I}_1$ as 
\begin{equation}
    \nonumber
    I_1\leq \underbrace{\sup_{\bm{v},\bm{\theta}}\sum_{k=1}^n \big(\widetilde{z}_k\bm{x}_k^\top\bm{v}-\mathbbm{E}[\widetilde{z}_k\bm{x}_k^\top\bm{v}]\big)}_{:=I_{11}}+\underbrace{\sup_{\bm{v}} \sum_{k=1}^n \big(\mathscr{T}_{\zeta_y}(\epsilon_k)\bm{x}_k^\top\bm{v}-\mathbbm{E}[\mathscr{T}_{\zeta_y}(\epsilon_k)\bm{x}_k^\top\bm{v}]\big)}_{:=I_{12}}. 
\end{equation}
Thus, it suffices to bound $I_{11}$ and $I_{12}$.

\noindent{\bf \sffamily Step 2.1.1. Bounding ${I}_{11}$}

We use Lemma \ref{productpro} to bound $I_{11}$. For any $\bm{v}_1,\bm{v}_2\in \mathscr{V}$, it is evident that we have $\|\bm{x}_k^\top\bm{v}\|_{\psi_2}\leq \|\bm{x}_k\|_{\psi_2}\leq \sigma$ and $\|\bm{x}_k^\top\bm{v}_1-\bm{x}_k^\top\bm{v}_2\|_{\psi_2} \leq \sigma \|\bm{v}_1-\bm{v}_2\|_2$. Regarding $\widetilde{z}_k=\widetilde{z}_{\bm{\theta},k}:=\mathscr{T}_{\zeta_y}(\bm{x}_k^\top\bm{\theta}+\epsilon_k)-\mathscr{T}_{\zeta_y}(\epsilon_k)$ indexed by $\bm{\theta}\in \Sigma_{s,R_0}$, the $1$-Lipschitzness of $\mathscr{T}_{\zeta_y}(\cdot)$ gives $|\widetilde{z}_k|\leq  |\bm{x}_k^\top\bm{\theta}|$, and then the definition of sub-Gaussian norm yields  $\|\widetilde{z}_k\|_{\psi_2}\leq \|\bm{x}_k^\top\bm{\theta}\|_{\psi_2}\leq \|\bm{x}_k\|_{\psi_2}\|\bm{\theta}\|_2\leq R_0\sigma$ (this addresses the aforementioned issue). Further, for any $\bm{\theta}_1,\bm{\theta}_2\in \Sigma_{s,R_0}$ we verify the sub-Gaussian increments  \begin{equation}
    \begin{aligned}\nonumber
         |\widetilde{z}_{\bm{\theta}_1,k}-\widetilde{z}_{\bm{\theta}_2,k}| =&\big|\mathscr{T}_{\zeta_y}(\bm{x}_k^\top\bm{\theta}_1+\epsilon_k)-\mathscr{T}_{\zeta_y}(\epsilon_k)-\mathscr{T}_{\zeta_y}(\bm{x}_k^\top\bm{\theta}_2+\epsilon_k)+\mathscr{T}_{\zeta_y}(\epsilon_k)\big|\\=&\big|\mathscr{T}_{\zeta_y}(\bm{x}_k^\top\bm{\theta}_1+\epsilon_k)-\mathscr{T}_{\zeta_y}(\bm{x}_k^\top\bm{\theta}_2+\epsilon_k)\big|\leq \big|\bm{x}_k^\top(\bm{\theta}_1-\bm{\theta}_2)\big|,
    \end{aligned}
\end{equation}
which leads to $\|\widetilde{z}_{\bm{\theta}_1,k}-\widetilde{z}_{\bm{\theta}_2,k}\|_{\psi_2}\leq \|\bm{x}_k^\top(\bm{\theta}_1-\bm{\theta}_2)\|_{\psi_2}\leq \|\bm{x}_k\|_{\psi_2}\|\bm{\theta}_1-\bm{\theta}_2\|_2\leq \sigma\|\bm{\theta}_1-\bm{\theta}_2\|_2$. With these preparations, we can invoke Lemma \ref{productpro} use the well-known estimates   $\omega(\Sigma_{s,R_0}),\omega(\mathscr{V})=O(\sqrt{s\log d})$\footnote{In fact, we have the tighter estimate $\omega(\Sigma_{s,R_0}),\omega(\mathscr{V})=O\Big(\sqrt{s\log\big(\frac{ed}{s}\big)}\Big)$ (e.g., \cite{plan2013one})  but we simply put   $\sqrt{s\log d}$ to be consistent with earlier results concerning unconstrained Lasso.} to obtain that,  with probability at least $1-2\exp(-cu^2)$ we have
\begin{equation}
    \begin{aligned}\label{B.26}
        {I}_{11} &\lesssim   \sigma^2\Big[\sqrt{n}\big(\omega(\Sigma_{s,R_0})+\omega(\mathscr{V})+u\big)+\big(\omega(\Sigma_{s,R_0})+u\big)\cdot\big(\omega(\mathscr{V})+u\big)\Big] \\
        & \lesssim \sigma^2\Big[\sqrt{n}\big(\sqrt{s\log d}+u\big)+\big(\sqrt{s\log d}+u\big)\cdot\big(\sqrt{s\log d}+u\big)\Big].
    \end{aligned}
\end{equation}
   Therefore, we can set  $u= \sqrt{\delta s\log d}$ in (\ref{B.26}), under the scaling of $n\gtrsim s\delta \log d$ it provides \begin{equation}\label{b31}
    \mathbbm{P}\big(I_{11}\lesssim \sigma^2 \sqrt{n \delta s\log d}\big) \geq 1-2d^{-\delta \Omega(s)}.
\end{equation}

\noindent{\bf \sffamily Step 2.1.2. Bounding $I_{12}$}

By $\|\bm{v}\|_1\leq 2\sqrt{s}$ we have $ I_{12} \leq 2\sqrt{s}\|\sum_{k=1}^n (\mathscr{T}_{\zeta_y}(\epsilon_k)\bm{x}_k -\mathbbm{E}[\mathscr{T}_{\zeta_y}(\epsilon_k)\bm{x}_k])\|_{\infty}.$   Then to apply Bernstein's inequality, for   integer $q\geq 2$ and $i\in [d]$, analogously to (\ref{nB.2}) in the proof of Theorem \ref{thm6}, we can bound that
\begin{equation}
    \begin{aligned}\nonumber
        \sum_{k=1}^n \mathbbm{E}\left|\frac{\mathscr{T}_{\zeta_y}(\epsilon_k)x_{ki}}{n}\right|^q &\leq \Big(\frac{\zeta_y}{n}\Big)^{q-2} \frac{1}{n^2}\sum_{k=1}^n \mathbbm{E}\big|\mathscr{T}^2_{\zeta_y}(\epsilon_k) x_{ki}^q\big|\\&\leq \Big(\frac{\sigma \zeta_y}{n}\Big)^{q-2}  \Big(\frac{\sigma^2M^{\frac{1}{l}}}{n}\Big)(Cq)^{\frac{q}{2}}\leq \frac{q!}{2}v_0c_0^{q-2},
    \end{aligned}
\end{equation}
for some $v_0=O\big(\frac{\sigma^2M^{1/l}}{n}\big)$, $c_0 = O\big(\frac{\sigma \zeta_y}{n}\big)$. Then we use Lemma \ref{bernstein} to obtain that, with probability at least  $1-2\exp(-t)$ we have\begin{equation}\nonumber
    \left|\frac{1}{n}\sum_{k=1}^n(\mathscr{T}_{\zeta_y}(\epsilon_k) x_{ki}-\mathbbm{E}[\mathscr{T}_{\zeta_y}(\epsilon_k) x_{ki}])\right|\leq C\sigma \Big(M^{\frac{1}{2l}}\sqrt{\frac{t}{n}} + \frac{\zeta_y t}{n} \Big) 
\end{equation} 
Then we use $\zeta_y\asymp \big(\sigma+M^{\frac{1}{2l}}\big)\sqrt{\frac{n}{\delta \log d}}$, set $t \asymp \delta \log d$, and take a union bound over $i\in [d]$ to obtain that,  $\|\sum_{k=1}^n\frac{1}{n} (\mathscr{T}_{\zeta_y}(\epsilon_k)\bm{x}_k -\mathbbm{E}[\mathscr{T}_{\zeta_y}(\epsilon_k)\bm{x}_k])\|_\infty \lesssim \sigma (M^{1/(2l)}+\sigma)\sqrt{\frac{\delta \log d}{n}}$ holds with probability at least $1-2d^{1-\delta}$, which implies the following under the same probability \begin{equation}\label{b34}
  I_{12}\lesssim \sigma\big(M^{\frac{1}{2l}}+\sigma\big)\sqrt{ns\delta \log d}.
\end{equation}
Therefore, combining (\ref{b31}) and (\ref{b34}), we obtain that $I_1\lesssim \sigma\big(M^{\frac{1}{2l}}+\sigma\big)\sqrt{ns\delta \log d}$ with the promised probability.

\noindent{\bf \sffamily Step 2.2. Bounding ${I}_2$}

For this bias term  the supremum does not make things harder. We begin with $I_2 = n \cdot \sup_{\bm{\theta},\bm{v}}\mathbbm{E}\big((\widetilde{y}_k - y_k)\bm{x}_k^\top \bm{v}\big)\leq 2n\sqrt{s}\cdot \sup_{\bm{\theta}}\|\mathbbm{E}(\widetilde{y}_k-y_k)\bm{x}_k\|_\infty.$ Fix any $\bm{\theta}\in \Sigma_{s,R_0}$, we have $\mathbbm{E}|y_k|^{2l}\lesssim \mathbbm{E}|\bm{x}_k^\top\bm{\theta}|^{2l}+ \mathbbm{E}|\epsilon_k|^{2l}\lesssim M +  \sigma^{2l}$. Then following arguments similarly to (\ref{nB.3}) we obtain $$\|\mathbbm{E}(\widetilde{y}_k-y_k)\bm{x}_k\|_\infty\lesssim \frac{\sigma M^{1/l}}{\zeta_y} \lesssim \sigma(M^{1/(2l)}+\sigma)\sqrt{\frac{\delta \log d}{n}},$$which implies $I_2\lesssim \sigma M^{\frac{1}{2l}}\sqrt{n\delta s\log d}$.

\noindent{\bf \sffamily Step 2.3. Bounding $I_3$}

It is evident that we can apply Lemma \ref{productpro} with $(g_{\bm{\theta}}(\bm{x}_k),h_{\bm{v}}(\bm{x}_k))=(\bm{\theta}^\top\bm{x}_k,\bm{v}^\top\bm{x}_k)$, $(\mathcal{A},\mathcal{B})=(\Sigma_{s,R_0},\mathscr{V})$, and hence with $K_{\mathcal{A}},r_{\mathcal{A}},K_{\mathcal{B}},r_{\mathcal{B}}=O(\sigma)$. Along with $\omega(\Sigma_{s,R_0}),\omega(\mathscr{V})\lesssim\sqrt{s\log d}$, we obtain that the following holds with probability at least $1-2\exp(-cu^2)$: \begin{equation}\nonumber
    I_3\leq \sigma^2 \Big[\big(\sqrt{s\log d}+u\big)^2+\sqrt{n}\big(\sqrt{s\log d}+u\big)\Big].
\end{equation}
By taking $u\asymp \sqrt{s\delta \log d}$, under the scaling $n\gtrsim \delta s\log d$, it follows that $I_3\lesssim \sigma^2\sqrt{n\delta s\log d}$ with probability at least $1-2d^{-\delta\Omega(s)}$.

\noindent{\bf \sffamily Step 2.4.  Bounding $I_0$}

It remains to bound $
   I_0= \sup_{\bm{\theta},\bm{v}}\sum_{k=1}^n \xi_k \bm{x}_k^\top \bm{v}$. Bounding $I_0$ is similar to establishing the "limited projection distortion (LPD)" property in \cite{xu2020quantized}, but the   key distinction is  that   $\bm{\theta}$ and $\bm{v}$ in $I_0$ take value in different spaces.

The main difficulty associated with "$\sup_{\bm{\theta}}$"  lies in the discontinuity of the quantization noise $\xi_k:= \mathcal{Q}_{\Delta}(\widetilde{y}_k+\tau_k)-\widetilde{y}_k$, which we overcome by a covering argument and some machinery developed in  \cite[Prop. 6.1]{xu2020quantized}. However, the essential difference from \cite{xu2020quantized} is that we use  Lemma \ref{talagrand} to handle "$\sup_{\bm{v}}$", while \cite{xu2020quantized} again used covering argument for $\bm{v}$ to strengthen  their Proposition 6.1 to their Proposition 6.2, which is unfortunately insufficient in our setting because the covering number of $\mathscr{V}$ significantly increases under smaller covering radius (on the other hand, using covering argument for $\bm{v}$ suffices for the analyses in \cite{xu2020quantized} regarding a different estimator named {\it projected back projection}).

Let us first construct a $\rho$-net of $\Sigma_{s,R_0}$   denoted by $\mathcal{G}=\{\bm{\theta}_1,...,\bm{\theta}_N\}$, so that for any $\bm{\theta}\in \Sigma_{s,R_0}$ we can pick $\bm{\theta}'\in \mathcal{G}$ satisfying $\|\bm{\theta}'-\bm{\theta}\|_2\leq \rho$; here, the covering radius $\rho$ is to be chosen later, and we   assume that $N\leq \big(\frac{9d}{\rho s}\big)^s$ \cite[Lemma 3.3]{plan2013one}. As is standard in a covering argument, we first control $I_0$ over the net $\mathcal{G}$ (by replacing "$\sup_{\bm{\theta}\in \Sigma_{s,R_0}}$" with "$\sup_{\bm{\theta}\in \mathcal{G}}$"), and then bound the approximation error induced by such replacement.

\noindent{\bf \sffamily Step 2.4.1. Bounding $I_0$ over $\mathcal{G}$}

In this step, we want to bound $I_{0,\mathcal{G}}:=\sup_{\bm{\theta}\in \mathcal{G}}\sup_{\bm{v}\in \mathscr{V}}\sum_{k=1}^n\xi_k\bm{x}_k^\top\bm{v}$.
First let us consider a fixed $\bm{\theta}\in \Sigma_{s,R_0}$. Then since $|\xi_k|\leq \Delta$, we have $\|\xi_k\bm{x}_k\|_{\psi_2} \lesssim \Delta\sigma$. Because $\{\xi_k\bm{x}_k:k\in [n]\}$ are independent zero mean, by \cite[Prop. 2.6.1]{vershynin2018high} we have   $\big\|\sum_{k=1}^n \xi_k\bm{x}_k\big\|_{\psi_2}\lesssim \sqrt{n} \Delta\sigma$. Define $\mathscr{V}'=\mathscr{V}\cup \{0\}$, then for any $\bm{v}_1,\bm{v}_2\in \mathscr{V}'$ we have $$\left\|\Big(\sum_{k=1}^n\xi_k\bm{x}_k\Big)^\top\bm{v}_1-\Big(\sum_{k=1}^n\xi_k\bm{x}_k\Big)^\top\bm{v}_2\right\|_{\psi_2}\leq C\sqrt{n}\Delta \sigma \|\bm{v}_1-\bm{v}_2\|_2.$$
Thus, by Lemma \ref{talagrand}, it holds with probability at least $1-\exp(-u^2)$ that \begin{equation}
    \begin{aligned}\nonumber
        \sup_{\bm{v}\in \mathscr{V}}\Big(\sum_{k=1}^n \xi_k\bm{x}_k\Big)^\top \bm{v}&\leq \sup_{\bm{v},\bm{v}'\in \mathscr{V}'}\left|\Big(\sum_{k=1}^n \xi_k\bm{x}_k\Big)^\top \bm{v}-\Big(\sum_{k=1}^n \xi_k\bm{x}_k\Big)^\top \bm{v}'\right| \\&\leq C\sqrt{n}  \Delta  \sigma \big(\omega(\mathscr{V}')+u\big) \leq C_1\sqrt{n} \Delta  \sigma \Big(\sqrt{s\log \frac{9d}{s}}+u\Big).
    \end{aligned}
\end{equation}
Moreover, by a union bound over $\mathcal{G}$, we obtain that $I_{0,\mathcal{G}}\lesssim \sigma \Delta \sqrt{n}\big(\sqrt{s\log \frac{9d}{s}}+u\big)$ holds with probability at least $1-\exp\big(s\log \frac{9d}{\rho s}  -u^2\big)$.
We set $u\asymp \sqrt{s\delta\log \big(\frac{9d}{\rho s}\big)}$ and arrive at \begin{equation}
    \mathbbm{P}\left(I_{0,\mathcal{G}}\lesssim \sigma \Delta\sqrt{ns\delta \log\frac{9d}{\rho s}}\right) \geq 1-\Big(\frac{9d}{\rho s}\Big)^{-\Omega(\delta s)}.\label{B.38}
\end{equation}

\vspace{1mm}
\noindent{\bf \sffamily Step 2.4.2.     Bounding the Approximation Error}

From now on we indicate the dependence of $\xi_k$ on $\bm{\theta}$ by using the notation $ \xi_{\bm{\theta},k}:= \mathcal{Q}_\Delta(\widetilde{y}_{\bm{\theta},k}+\tau_k)-\widetilde{y}_{\bm{\theta},k}$ where $\widetilde{y}_{\bm{\theta},k}=\mathscr{T}_{\zeta_y}(\bm{x}_k^\top\bm{\theta}+\epsilon_k)$. 
For any $\bm{\theta}\in \Sigma_{s,R_0}$ we   pick one $\bm{\theta}'\in \mathcal{G}$   such that $\|\bm{\theta}- \bm{\theta}'\|_2 \leq \rho$; we fix such correspondence and remember that from now on every $\bm{\theta}\in\Sigma_{s,R_0}$ is associated with some $\bm{\theta'}\in \mathcal{G}$, (which of course depends on $\bm{\theta}$ but our notation omits such dependence). Thus, we can bound $I_0=\sup_{\bm{\theta},\bm{v}}\sum_{k=1}^n\xi_{\bm{\theta},k}\bm{x}_k^\top\bm{v}$ as
\begin{equation}\label{B.39}\begin{aligned}
     I_0 &\leq \sup_{\bm{\theta},\bm{v}}\sum_{k=1}^n \xi_{\bm{\theta'},k}\bm{x}_k^\top\bm{v}   + \sup_{\bm{\theta},\bm{v}}  \sum_{k=1}^n(\xi_{\bm{\theta},k}-\xi_{\bm{\theta}',k})\bm{x}_k  ^\top\bm{v} \leq I_{0,\mathcal{G}}+I_{01}.
    \end{aligned}
\end{equation}
Note that the bound on $I_{0,\mathcal{G}}$ is available in  (\ref{B.38}), so it remains to bound $I_{01}:=\sup_{\bm{\theta},\bm{v}}\sum_{k=1}^n (\xi_{\bm{\theta},k}-\xi_{\bm{\theta'},k})\bm{x}_k^\top\bm{v}$, which can be understood as the approximation error of the net $\mathcal{G}$ regarding the empirical process of interest. To facilitate the presentation we switch to the more compact notations --- let $\bm{X}\in \mathbb{R}^{n\times d}$ with rows $\bm{x}_k^\top$ be the sensing matrix, $\bm{\xi}_{\bm{\theta}}=[\xi_{\bm{\theta},k}]\in \mathbb{R}^n$ be the quantization error indexed by $\bm{\theta}$, $\bm{\tau}=[\tau_k]\in\mathbb{R}^n$ be the random dither vector, $\bm{\epsilon}=[\epsilon_k]\in \mathbb{R}^n$ be the heavy-tailed noise vector, $\bm{{y}}_{\bm{\theta}}= [y_{\bm{\theta},k}]= \bm{X\theta}+\bm{\epsilon}\in \mathbb{R}^n$ and  $\bm{\widetilde{y}}_{\bm{\theta}}=[\widetilde{y}_{\bm{\theta},k}]= \mathscr{T}_{\zeta_y}(\bm{y}_{\bm{\theta}})$ be the measurement vector and truncated measurement vector, respectively.  With these conventions we can write $I_{01}=\sup_{\bm{\theta},\bm{v}}(\bm{\xi_{\theta}}-\bm{\xi_{\theta'}})^\top \bm{Xv}$. Recall that a specific $\bm{\theta'}$ has been specified for each $\bm{\theta}\in \Sigma_{s,R_0}$, so defining $\bm{\Psi_{\theta}}:=\bm{\xi_{\theta}}-\bm{\xi_{\theta'}}$ allows us to write 
$I_{01}= \sup_{\bm{\theta},\bm{v}}~\bm{\Psi}_{\bm{\theta}}^\top\bm{Xv}$. Further,     we   define $\bm{\widetilde{\Psi}_{\theta}}:=\bm{\widetilde{y}_{\theta}}-\bm{\widetilde{y}_{\theta'}},\bm{\widehat{\Psi}_{\theta}}:=\mathcal{Q}_\Delta(\bm{\widetilde{y}_{\theta}}+\bm{\tau})-\mathcal{Q}_\Delta(\bm{\widetilde{y}_{\theta'}}+\bm{\tau})$ and make the following observation
\begin{equation}
    \begin{aligned}\label{B.40}
        \bm{\Psi_\theta} = \bm{\xi_\theta}-\bm{\xi_{\theta'}}=\bm{\widehat{\Psi}_{\theta}}- \bm{\widetilde{\Psi}_\theta}.
    \end{aligned}
\end{equation}
We pause to establish a property of  $\bm{X}$ that holds w.h.p.. Specifically, we restrict (\ref{nB.6}) to $\bm{v}\in \mathscr{V}$  (recall that $\bar{\Delta}=0$ and so $\bm{\dot{X}}=\bm{X}$ and $\bm{\dot{\Sigma}}= \bm{\Sigma^\star}$), then under the promised probability it holds for some $c(\kappa_0 ,\sigma)$   that $$\sup_{\bm{v}\in\mathscr{V}}\left|\frac{\|\bm{Xv}\|_2}{\sqrt{n}}- \|\bm{\sqrt{\Sigma^\star} v}\|_2\right|\leq c(\kappa_0, \sigma)\sqrt{\frac{\delta s\log d}{n}}.$$
Thus, when $n\gtrsim \delta s \log d$ with for large enough hidden constant depending on $(\kappa_0,\sigma)$, it holds that \begin{equation}\label{upperevent}
    \sup_{\bm{v}\in \mathscr{V}}\frac{\|\bm{Xv}\|_2}{\sqrt{n}}  \leq  
\sup_{\bm{v}\in \mathscr{V}}\|\sqrt{\bm{\Sigma^\star}}\bm{v}\|_2+\sqrt{\kappa_1}\leq 2\sqrt{\kappa_1}.
\end{equation} We proceed by  assuming we are on this event,   
which allows us to   bound 
$I_{01}$ as 
\begin{equation}\label{start}
\begin{aligned}
    I_{01} =\sup_{\theta,\bm{v}}\bm{\Psi_{\theta}}^\top\bm{Xv} &\leq \sup_{\bm{\theta}}\|\bm{\Psi_{\theta}}\|_2 \sup_{\bm{v}\in\mathscr{V}}\|\bm{Xv}\|_2\\&\leq 2\sqrt{\kappa_1n}\cdot \sup_{\bm{\theta}}\|\bm{\Psi_\theta}\|_2.\end{aligned}
\end{equation}
To bound $\sup_{\bm{\theta}}\|\bm{\Psi_\theta}\|_2$, motivated by (\ref{B.40}), we will investigate $\bm{\widetilde{\Psi}_{\theta}}$ and $\bm{\widehat{\Psi}_{\theta}}$ more carefully. We pick a threshold $\eta\in (0,\frac{\Delta}{2})$ (that is to be chosen later), and by the $1$-Lipschitzness of $\mathscr{T}_{\zeta_y}(\cdot)$ we have \begin{equation}
    \begin{aligned}\label{erraftermap}\sup_{\bm{\theta}}\|\bm{\widetilde{\Psi}_{\theta}}\|_2 &=\sup_{\bm{\theta}}\|\bm{\widetilde{y}}_{\bm{\theta}}-\bm{\widetilde{y}}_{\bm{\theta}'}\|_2 \leq \sup_{\bm{\theta}}\|\bm{y}_{\bm{\theta}}-\bm{y}_{\bm{\theta}'}\|_2\\&=\sup_{\bm{\theta}}\|\bm{X}(\bm{\theta}-\bm{\theta}')\|_2 \leq 2 \sqrt{\kappa_1n}\rho,\end{aligned}
\end{equation}
where the last inequality is because $\bm{\theta}-\bm{\theta}'$ is $2s$-sparse,   hence (\ref{upperevent}) implies $\|\bm{X}(\bm{\theta}-\bm{\theta}')\|_2 \leq 2 \sqrt{\kappa_1n}\|\bm{\theta}-\bm{\theta}'\|_2\leq 2 \sqrt{\kappa_1n}\rho$.

To proceed, we will define for  specific $\bm{\theta}$ the index vectors  $\bm{J}_{\bm{\theta},1},\bm{J}_{\bm{\theta},2}\in \{0,1\}^n$ and use $|\bm{J}_{\bm{\theta},1}|$ to denote the number of $1$s in $\bm{J}_{\bm{\theta},1}$ (similar meaning for $|\bm{J}_{\bm{\theta},2}|$). Specifically, using the entry-wise notation $\bm{\widetilde{\Psi}_\theta}=[\widetilde{\Psi}_{\bm{\theta},k}]$ we  define
$\bm{J}_{\bm{\theta},1} = [\mathbbm{1}(|\widetilde{\Psi}_{\bm{\theta},k}|\geq \eta)]$.   Recall that   (\ref{erraftermap}) gives $\sup_{\bm{\theta}}\|\bm{\widetilde{\Psi}_{\theta}}\|_{2}^2\leq 4\kappa_1n\rho^2$; combined with the simple observation   $\sup_{\bm{\theta}}\|\bm{\widetilde{\Psi}_{\theta}}\|_2^2\geq \sup_{\bm{\theta}}\eta^2|\bm{J}_{\bm{\theta},1}|$, we obtain  a uniform bound on $|\bm{J}_{\bm{\theta},1}|$ as
\begin{equation}
    \sup_{\bm{\theta}\in \Sigma_{s,R_0}}|\bm{J}_{\bm{\theta},1}|\leq \frac{4\kappa_1n\rho^2}{\eta^2}.\label{boundJ1}
\end{equation} Next, we define the index vector $\bm{J}_{\bm{\theta},2}$ for $\bm{\theta}\in \mathcal{G}$: first let
$\mathscr{E}_{\bm{\theta},i}= \{\mathcal{Q}_\Delta(\widetilde{y}_{\bm{\theta},i}+\tau_i +t )\text{ is discontinuous }$ $\text{in }t\in [-\eta,\eta]\},$ and then we define $\bm{J}_{\bm{\theta},2} := [\mathbbm{1}(\mathscr{E}_{\bm{\theta},i} )]$. Then by Lemma \ref{uniformJ2} that we prove later, we have \begin{equation}\label{useJ2bound}
    \mathbbm{P}\Big(\sup_{\bm{\theta}\in \mathcal{G}}|\bm{J}_{\bm{\theta},2}| \leq \frac{Cn\eta}{\Delta}\Big)\geq 1-\exp\Big(-\frac{cn\eta}{\Delta}+s\log \frac{9d}{\rho s} \Big):=1-\mathscr{P}_1.
\end{equation}

Note that $\mathscr{E}_{\bm{\theta},i}$ does not happen (i.e., $\mathbbm{1}(\mathscr{E}_{\bm{\theta},i})=0$) means that $\mathcal{Q}_\Delta(\widetilde{y}_{\bm{\theta},i}+\tau_i +t)$ is continuous in $t\in [-\eta,\eta]$; combined with the definition of $\mathcal{Q}_{\Delta}(\cdot)$, this is also equivalent to the statement that "$\mathcal{Q}_\Delta(\widetilde{y}_{\bm{\theta},i}+\tau_i +t)$ remains constant in $t\in [-\eta,\eta]$." Thus, given a fixed $\bm{\theta}\in \Sigma_{s,R_0}$ and its associated $\bm{\theta'}$, suppose that the $i$-th entry of $\bm{J}_{\bm{\theta'},2}$ is zero (meaning that "$\mathcal{Q}_\Delta(\widetilde{y}_{\bm{\theta'},i}+\tau_i+t)$ remains constant in $t\in [-\eta,\eta]$"),   if additionally $i$-th entry of $\bm{J}_{\bm{\theta},1}$ is zero   (i.e., $|\widetilde{\Psi}_{\bm{\theta},i}|<\eta$), then the $i$-th entry of $\bm{\widehat{\Psi}_{\bm{\theta}}}=\mathcal{Q}_\Delta(\bm{\widetilde{y}_\theta}+\bm{\tau})-\mathcal{Q}_\Delta(\bm{\widetilde{y}_{\theta'}}+\bm{\tau})$ vanishes:
\begin{equation}
    \begin{aligned}
        \nonumber\widehat{\Psi}_{\bm{\theta},i}&=\mathcal{Q}_\Delta(\widetilde{y}_{\bm{\theta},i}+\tau_{i})-\mathcal{Q}_\Delta(\widetilde{y}_{\bm{\theta'},i}+\tau_i)\\
        &=\mathcal{Q}_\Delta(\widetilde{y}_{\bm{\theta'},i}+\widetilde{\Psi}_{\bm{\theta},i}+\tau_i)-\mathcal{Q}_\Delta(\widetilde{y}_{\bm{\theta'},i}+\tau_i)=0;
    \end{aligned}
\end{equation}
combining with (\ref{B.40}), this implies $\Psi_{\bm{\theta},i}=-\widetilde{\Psi}_{\bm{\theta},i}$.
  Recall from (\ref{start}) that we want to bound $\sup_{\bm{\theta}}\|\bm{\Psi_\theta}\|_2$.  Write $\bm{J}_{\bm{\theta},1}^c = \bm{1}- \bm{J}_{\bm{\theta},1}$ and $\bm{J}_{\bm{\theta'},2}^c = \bm{1}- \bm{J}_{\bm{\theta'},2}$, then denoting hadamard  product by $\odot$ and  using the decomposition $\bm{1}= \max\{\bm{J}_{\bm{\theta},1},\bm{J}_{\bm{\theta'},2}\}+ \min\{\bm{J}_{\bm{\theta},1}^c,\bm{J}_{\bm{\theta'},2}^c\}$ and (\ref{B.40}) we have \begin{equation}
    \begin{aligned}\label{B.47}
       &\sup_{\bm{\theta}} \big\|\bm{\Psi_\theta}\big\|_2 = \sup_{\bm{\theta}}\big\|\bm{\Psi_\theta}\odot \max\{\bm{J}_{\bm{\theta},1},\bm{J}_{\bm{\theta'},2}\}+\bm{\Psi_\theta}\odot \min\{\bm{J}_{\bm{\theta},1}^c,\bm{J}^c_{\bm{\theta'},2}\} \big\|_2 \\
       &\leq \sup_{\bm{\theta}} \big\|\bm{\Psi_\theta}\odot \max\{\bm{J}_{\bm{\theta},1},\bm{J}_{\bm{\theta'},2}\}\big\|_2+\sup_{\bm{\theta}}\big\|\bm{\Psi_\theta}\odot \min\{\bm{J}_{\bm{\theta},1}^c,\bm{J}^c_{\bm{\theta'},2}\}\big\|_2\\
       &\stackrel{(i)}{\leq} \sup_{\bm{\theta}}\big\|\bm{\Psi_\theta}\big\|_{\infty}\cdot \sqrt{|\bm{J}_{\bm{\theta},1}|+|\bm{J}_{\bm{\theta'},2}|} ~+\sup_{\bm{\theta}}\big\|\bm{\widetilde{\Psi}_{\bm{\theta}}}\odot\min\{\bm{J}_{\bm{\theta},1}^c,\bm{J}_{\bm{\theta'},2}^c\}\big\|_2 \\&\stackrel{(ii)}{\leq} C\Big\{\Delta\sqrt{n}\Big(\frac{\sqrt{\kappa_1}\rho}{\eta}+ \sqrt{\frac{\eta}{\Delta}}\Big)+  \rho\sqrt{\kappa_1n}\Big\},
    \end{aligned}
\end{equation}
where  $(i)$ is because entries of $\bm{\Psi}_{\bm{\theta}}$ equal to those of $-\bm{\widetilde{\Psi}}_{\bm{\theta}}$ if the index corresponds to $\min\{\bm{J}_{\bm{\theta},1}^c,\bm{J}_{\bm{\theta'},2}^c\}=1$, and in $(ii)$ we use the simple bound $\|\bm{\Psi_{\theta}}\|_\infty \leq \|\bm{\xi_\theta}\|_\infty+\|\bm{\xi_{\theta'}}\|_\infty\leq 2\Delta$ and the derived bounds on $\sup_{\bm{\theta}}|\bm{J}_{\bm{\theta},1}|,\sup_{\bm{\theta}\in \mathcal{G}}|\bm{J}_{\bm{\theta},2}|,\sup_{\bm{\theta}}\|\bm{\widetilde{\Psi}_\theta}\|_2$ in (\ref{boundJ1}), (\ref{useJ2bound}) and   (\ref{erraftermap}), respectively.

\vspace{1mm}

\noindent{\bf \sffamily Step 2.4.3. Concluding the Bound on $I_0$}

We are ready to put pieces together, specify  $\rho,\eta$, and conclude the bound on $I_0$. Overall, with probability at least $1-\mathscr{P}_1-\mathscr{P}_2$ for $\mathscr{P}_1$ defined in (\ref{useJ2bound}) and some $\mathscr{P}_2$ within the promised probability, combining (\ref{B.38}), (\ref{B.39}), (\ref{start}) and (\ref{B.47}) we obtain \begin{equation} \nonumber
    I_0\lesssim\sigma\Delta\sqrt{ns\delta \log \frac{9d}{\rho s}}+\frac{\kappa_1n\Delta \rho}{\eta}+n\sqrt{\kappa_1\eta\Delta}+\kappa_1\rho n.
\end{equation}
Thus, we take the (near-optimal) choice of $(\rho,\eta)$ as $\rho \asymp \frac{\Delta}{\sqrt{\kappa_1}}\big(\frac{s\delta }{n}\big)^{3/2}$ and $\eta\asymp \frac{\delta \Delta s}{n}\log\frac{9d}{\rho s} ,$
under which we obtain that, with the promised probability (as $\mathscr{P}_1$ is also sufficiently small), we obtain the bound on $I_0$ as \begin{equation}
    \label{I0bound}
    I_0\lesssim \sigma\Delta\sqrt{ns\delta \log\Big(\frac{\kappa_1d^2n^3}{\Delta^2 s ^5\delta ^3}\Big)}
\end{equation}



 We can conclude the proof with all the works above. Substituting $\inf_{\bm{v}\in\mathscr{V}} \sum_{k=1}^n(\bm{x}_k^\top\bm{v})^2\geq \frac{1}{4}\kappa_0n$ and the definition of $I$ in (\ref{defineI}) into (\ref{frame}), then we obtain $\frac{1}{4}\kappa_0n\|\bm{\widehat{\Upsilon}}\|_2^2\leq 2\|\bm{\widehat{\Upsilon}}\|_2I$ that holds uniformly for all $\bm{\theta}\in \Sigma_{s,R_0}$, which implies $\sup_{\bm{\theta}}\|\bm{\widehat{\Upsilon}}\|_2\leq \frac{8I}{\kappa_0n}$. Substituting the derived bounds on $I_1,I_2,I_3,I_0$ into (\ref{B.29}), with the promised probability we have \begin{equation}\nonumber
    I \lesssim \sigma\big(\sigma+M^{\frac{1}{2l}}\big)\sqrt{ns\delta\log d}+\sigma\Delta\sqrt{ns\delta \log \Big(\frac{\kappa_1d^2n^3}{\Delta^2s^5\delta^3}\Big)},
\end{equation}
so 
the uniform bound on $\|\bm{\widehat{\Upsilon}}\|_2$ follows immediately. Further using $\|\bm{\widehat{\Upsilon}}\|_1\leq 2\sqrt{s}\|\bm{\widehat{\Upsilon}}\|_2$ completes the proof.\end{proof}
\begin{lem}
    \label{uniformJ2}{\rm (Bounding $\sup_{\bm{\theta}\in\mathcal{G}}|\bm{J}_{\bm{\theta},2}|$)\textbf{.}}
    Along the proof of Theorem \ref{uniformtheorem}, it holds that \begin{equation}\label{J2bound}
        \mathbbm{P}\Big(\sup_{\bm{\theta}\in \mathcal{G}}|\bm{J}_{\bm{\theta},2}| \geq \frac{Cn\eta}{\Delta}\Big)\leq \exp\Big(-\frac{cn\eta}{\Delta}+s\log \frac{9d}{\rho s} \Big).
    \end{equation}
\end{lem}

\begin{proof}
Notation and details in the proof of Theorem \ref{uniformtheorem} will be used.
We first consider a fixed $\bm{\theta}\in \mathcal{G}$, and by a simple shifting $\mathscr{E}_{\bm{\theta},i}$ happens if and only if $\mathcal{Q}_\Delta(\cdot)$ is discontinuous in $[\widetilde{y}_{\bm{\theta} ,i}+\tau_i-\eta,\widetilde{y}_{\bm{\theta} ,i}+\tau_i+\eta]$, which is also equivalently to $[\widetilde{y}_{\bm{\theta} ,i}+\tau_i-\eta,\widetilde{y}_{\bm{\theta} ,i}+\tau_i+\eta]\cap (\Delta\cdot \mathbb{Z})=\varnothing$. Because $\tau _i \sim \mathscr{U}\big([-\frac{\Delta}{2},\frac{\Delta}{2}]\big)$ and $\eta < \frac{\Delta}{2}$, $\mathbbm{P}(\mathscr{E}_{\bm{\theta} ,i}) = \frac{2\eta}{\Delta}$ is valid independent of the location of $[\widetilde{y}_{\bm{\theta} ,i}-\eta,\widetilde{y}_{\bm{\theta} ,i}+\eta]$.
Thus, for fixed $\bm{\theta}$, by conditioning on $(\bm{X},\bm{\epsilon})$, $|\bm{J}_{\bm{\theta},2}|$ follows the binomial distribution with $n$ trials and probability of success $p:= \frac{2\eta}{\Delta}$. This allows us to write $|\bm{J}_{\bm{\theta},2}| = \sum_{k=1}^n J_k$ with $J_k$ i.i.d. following Bernoulli distribution with success probability $\mathbbm{E}J_k = p$. Then
for any integer $q\geq 2$ we have $ \sum_{k=1}^n\mathbbm{E}|J_k - \mathbbm{E}J_k|^q\leq \sum_{k=1}^n \mathbbm{E}|J_k-\mathbbm{E}J_k|^2 \leq np(1-p)\leq \frac{q!}{2}np.$  Now we invoke Bernstein's inequality (Lemma \ref{bernstein}) to obtain that for any $t>0$,
\begin{equation}\nonumber
    \mathbbm{P}\big(|\bm{J}_{\bm{\theta},2}| -np \geq \sqrt{2np t}+t\big)\leq \exp(-t).
\end{equation}
We let $t = cnp$ and take a union bound over $\bm{\theta} \in\mathcal{G}$; this yields the desired claim since $|\mathcal{G}|\leq \big(\frac{9d}{\rho s}\big)^s$. 
\end{proof}

\end{appendix} 
\end{document}